%% file: MAIN.tex
\documentclass[11pt]{article}
\newif\ifarxiv
\usepackage[activeacute,english]{babel} 
\usepackage[utf8]{inputenc}
\usepackage{verbatim} 
\usepackage{graphicx} 
\usepackage[usenames,dvipsnames]{color} 
\usepackage{enumerate,mdwlist}  
\usepackage{multirow} 
\usepackage{amsmath,amsfonts,amssymb,amsthm} 
\usepackage{empheq} 
\usepackage{bbm,dsfont} 
\usepackage{mathrsfs} 
\numberwithin{equation}{section} 
\usepackage{float} 
\usepackage[labelfont=bf,labelsep=space]{caption} 
\usepackage{empheq} 

\usepackage[OT2,OT1]{fontenc}
\DeclareRobustCommand\cyr{%
  \renewcommand\rmdefault{wncyr}%
  \renewcommand\sfdefault{wncyss}%
  \renewcommand\encodingdefault{OT2}%
  \normalfont
  \selectfont}
\DeclareTextFontCommand{\textcyr}{\cyr}

\usepackage{subcaption}

\usepackage{geometry}
\usepackage{color}
\definecolor{red}{rgb}{.7,0,0}
\definecolor{blue}{rgb}{0,0,1}

\def\yuproj{\textrm{\cyr Yu}}
\def\mcA{\mathcal{A}}

\def\mcD{\mathcal{D}}
\def\mcG{\mathcal{G}}
\def\mcH{\mathcal{H}}
\def\mcK{\mathcal{K}}
\def\mcL{\mathcal{L}}

\def\mcP{\mathcal{P}}

\def\mcS{\mathcal{S}}

\def\mcX{\mathcal{X}}

\def\mcV{\mathcal{V}}
\def\mcZ{\mathcal{Z}}

\def\msO{\mathscr{O}}
\def\msU{\mathscr{U}}

\def\bbR{\mathbb{R}}

\def\bbN{\mathbb{N}}

\def\bbC{\mathbb{C}}
\def\bbP{\mathbb{P}}
\def\bbS{\mathbb{S}}
\def\bbE{\mathop{\mathbb{E}}}
\def\bbF{\mathbb{F}}
\def\bbone{\mathds{1}}
\def\bbtwo{\mathbf{2}}
\def\fkA{\mathfrak{A}}

\def\fkc{\mathfrak{c}}
\def\fkf{\mathfrak{f}}
\def\fkg{\mathfrak{g}}
\def\fkh{\mathfrak{h}}
\def\fkq{\mathfrak{q}}
\def\fkx{\mathfrak{x}}
\def\fky{\mathfrak{y}}
\def\fkz{\mathfrak{z}}

\def\cM{\mathsf{M}}
\def\cK{\mathsf{K}}

\def\bfd{\boldsymbol{d}}
\def\Hd{\mcH_{\bfd}}
\def\Pd{\mcP_{d}}

\def\hm{^{\mathsf{h}}}
\def\kappabar{\overline{\kappa}}

\def\prob{\mathrm{Prob}}

\def\Oh{\mathcal{O}}

\def\Tg{\mathrm{T}}
\def\diff{\mathrm{D}}

\def\diffa{\overline{\diff}}

\def\uR{^{\bbR}}
\def\uC{^{\bbC}}
\def\uF{^{\bbF}}
\def\D{\mathbf{D}} 
\def\binfty{\overline{\infty}}

\def\rst{\rho_{\mathsf{std}}}
\def\dist{\mathrm{dist}}
\DeclareMathOperator{\vol}{vol}
\DeclarePairedDelimiter\abs{\lvert}{\rvert}
\DeclarePairedDelimiter\norm{\lVert}{\rVert}
\def\enumber{\mathrm{e}}

\usepackage{ifpdf}
\ifpdf
\usepackage[pdfencoding=auto,unicode]{hyperref} 
\hypersetup{
 pdftoolbar=true,
 pdfmenubar=true,
 pdfstartview={FitV},
 pdftitle={Functional norms, condition numbers and numerical algorithms in algebraic geometry},
 pdfauthor={Felipe Cucker, Alperen Ali Erg\"{u}r, Josu\'{e} Tonelli-Cueto},
 pdfsubject={semialgebraic geometry, algebraic topology, numerical algorithms},
 pdfkeywords={semialgebraic geometry,algebraic topology},
 pdfcreator={overleaf}, 
 pdfproducer={overleaf},
 linkcolor=red     
 citecolor=green   
 urlcolor=cyan     
 filecolor=magenta 
}
\fi
\usepackage{bookmark}

\title{Functional norms, condition numbers\\and numerical algorithms in algebraic geometry\thanks{This work was supported by the Einstein Foundation Berlin. \newline 2010 Mathematics Subject Classification: 14Q20, 65Y20 (primary); 68Q25, 68U05 (secondary).}}
\author{
Felipe Cucker\thanks{Partially supported by GRF grant 
CityU 11300220.}\\
Dept. of Mathematics\\ 
City University of Hong Kong\\ 
HONG KONG\\
{\tt macucker@cityu.edu.hk} 
\and
Alperen A. Ergür\thanks{Partially supported by NSF CCF 2110075.} \\
Dept. of Mathematics\\
University of Texas at San Antonio\\ 
San Antonio, Texas 78249, USA\\
{\tt alperen.ergur@utsa.edu} 
\and
Josu\'{e} Tonelli-Cueto\thanks{Supported by a postdoctoral fellowship of the 2020 ``Interaction'' program of the \emph{Fondation Sciences Mathématiques de Paris}. Partially supported by the ANR JCJC
GALOP (ANR-17-CE40-0009), the PGMO grant ALMA, and the PHC GRAPE.}\\
Inria Paris \& IMJ-PRG\\ 
Sorbonne Université\\
Paris, FRANCE\\
{\tt  josue.tonelli.cueto@bizkaia.eu}
}
\date{}

\makeatletter
\def\th@plain{%
  \thm@notefont{}
  \slshape 
}
\def\th@definition{%
  \thm@notefont{}
  \normalfont 
}
\makeatother
\theoremstyle{plain}
\newtheorem{lem}{Lemma}[section]
\newtheorem{prop}[lem]{Proposition}
\newtheorem{theo}[lem]{Theorem}
\newtheorem*{theo*}{Theorem}
\newtheorem*{temptheo*}{Template Theorem}
\newtheorem{cor}[lem]{Corollary}

\theoremstyle{definition}
\newtheorem{defi}[lem]{Definition}
\theoremstyle{remark}

\newtheorem{exam}[lem]{Example}

\newtheorem{remark}[lem]{Remark}
\newcommand{\eproof}{\hfill\qed}

\usepackage{multicol}
\usepackage[ruled,algosection,vlined]{algorithm2e}
\SetKwInOut{postcondition}{Postcondition}
\SetKwInOut{precondition}{Precondition}
\makeatletter
\let\original@algocf@latexcaption\algocf@latexcaption
\long\def\algocf@latexcaption#1[#2]{%
  \@ifundefined{NR@gettitle}{%
    \def\@currentlabelname{#2}%
  }{%
    \NR@gettitle{#2}%
  }%
  \original@algocf@latexcaption{#1}[{#2}]%
}
\makeatother

\begin{document}
\maketitle
\input{0abstract.tex}
\input{1TEXT.tex}

\bibliographystyle{plain}
{\small 
\bibliography{biblio}
}
\end{document}

%% file: 0abstract.tex
\begin{abstract}
In numerical linear algebra, a well-established practice is to choose a norm that exploits the structure of the problem at hand in order to optimize accuracy or computational complexity. In numerical polynomial algebra, a single norm (attributed to Weyl) dominates the literature. This article initiates the use of $L_p$ norms for numerical algebraic geometry, with an emphasis on $L_{\infty}$. This classical idea yields strong improvements in the analysis of the number of steps performed by numerous iterative algorithms. In particular,  we exhibit three algorithms where, despite the complexity of computing $L_{\infty}$-norm, the use of $L_p$-norms substantially reduces computational complexity: a subdivision-based algorithm in real algebraic geometry for computing the homology of semialgebraic sets, a well-known meshing algorithm in computational geometry, and the computation of zeros of systems of complex quadratic polynomials (a particular case of Smale's 17th problem). 
\end{abstract}

%% file: 1TEXT.tex
\section{Introduction}

In numerical analysis, it matters how me measure errors. 
Change the metric we measure the perturbations with and 
a well-conditioned input may turn badly-conditioned
(a remarkable example is in~\cite{ChC:09}). 
Because of this, a careful choice of how we measure errors 
is a fundamental step in the design and analysis of algorithms. 
A main example is numerical linear algebra where it is 
commonplace to carefully choose a matrix norm depending on 
the problem at hand: the goal is to exploit the structure of 
the problem and to optimize  computational efficiency.

Unlike numerical linear algebra, a single norm ---the Weyl norm--- prevails in numerical algebraic geometry. The nice properties of the Weyl norm, ease to compute and unitary invariance, explain 
this prevalence. Nevertheless, the absence of complexity 
analyses using other norms in numerical algebraic geometry 
reflects badly on the theoretical strength of our analyses which appear to rely on a specific choice of metric.

In this paper, we aim to show that using other norms is possible 
in numerical algebraic geometry. 
To do so we consider an $L_\infty$-norm in the space of polynomial systems and show how this leads to numerical algorithms and a complexity framework analogous to the one we have with the Weyl norm. Furthermore, we show that the change of norms leads to significant improvements in complexity bounds thanks to the better probabilistic behaviour of this $L_\infty$ norm with 
respect to the Weyl norm. We show this in three relevant
cases: 1) computation of the homology of algebraic sets, 2) the
Plantinga-Vegter algorithm, and 3) the homotopy continuation 
method for quadratic polynomial systems.

We now discuss more in detail the aspects that we have 
mentioned in passing in order to put in context our results 
within the
wider setting of the complexity theory for numerical algorithms 
and numerical algebraic geometry.
\smallskip

\noindent{\bf Complexity paradigm.} The behaviour of numerical algorithms varies from input to input. This phenomenon is not 
necessarily due to the algorithm themselves, but to the numerical sensitivity ---how much does the output varies with respect to a perturbation of the input--- of the input we are processing. The numerical sensitivity of an input is captured by the so-called \emph{condition number}. Then, in turn, condition numbers allow 
one to analyze numerical algorithms and to explain why numerical algorithms handle some inputs faster than others.

Central to our paper is the fact that the choice of the metric
under which we measure perturbations determines the condition
number of the data. An example of this is given by 
the polynomial $X^d-1$ which is well-conditioned (for the zero finding problem) with respect 
to the standard norm~\eqref{eq:std} but badly-conditioned with respect to the 
Weyl norm~\cite[Example~14.3]{Condition}.

A drawback of condition-based complexity analyses is that, as we 
don't know a priori the condition of the input at hand, we cannot foresee the running time for this input. We can nonetheless 
get an idea of how the algorithm behaves in general via 
randomizing the input. 
This allows one to obtain probabilistic estimates 
for the practical performance of the numerical algorithm.

Again, we note that the metric we choose to measure perturbations
affects the probabilistic models we consider. This is so because probabilistic parameters, such as the variance, are always given
with respect to some metric and so, when we change the metric, 
we are changing the value of these parameters.

We refer to~\cite{Condition} for a more detailed overview on 
this complexity paradigm based on condition numbers. In the 
rest of the paper we will 
show how this complexity framework works for each of the
three cases mentioned above.

\smallskip
\noindent{\bf Choice of the norm.} 
%
Our choice of the $L_\infty$-norm is, in a way,  unfortunate. Currently, we don't have an efficient way to approximate
$\|~\|_{\infty}$. For polynomials in 
$n+1$ homogeneous variables whose degrees are bounded by $\D$ our 
current fastest algorithm takes time polynomial in $\D$ 
and exponential in $n$. The computation of $\|~\|_\infty$ amounts, however, to a 
polynomial optimization problem and efficient algorithms exist for 
particular classes of polynomials. This is the case e.g., with 
sums of squares~\cite{laurent2009,bhattiproluguruswamilee2017},
sparse 
polynomials~\cite{dresslerilimandewolff2017,chandrasekaranshah2016}, and other 
structures~\cite{barvinok2007}. 
Unrestricted efficient algorithms are not expected 
to be designed because it is well-known that polynomial 
optimization reduces to the feasibility problem over the reals and 
the latter is ${\mathrm{NP}_{\bbR}}$-complete. The fact that for 
most applications we only need a coarse approximation of $\|~\|_{\infty}$ allows, nonetheless, for some optimism. 

Our choice of the $L_\infty$-norm is due to the inequalities
shown in Kellogg's theorem (Theorem~\ref{theo:kellogg}) 
which we haven't found for other $L_p$-norms. A way around
Kellogg's theorem for general $L_p$-norms would certainly lead to new results regarding the use of these norms in algorithm analysis. 

Despite the high cost of computing the $L_{\infty}$ norm, its use may yield substantially better cost bounds for some algorithms. This 
improvement rests on two facts: 
\begin{enumerate}
\item For a homogenous polynomial $f$ with $n+1$ variables and degree $\D$ we always have $\norm{f}_{\infty} \leq \norm{f}_W$, and for a random homogenous polynomial $\fkf$ we have $\norm{\fkf}_{\infty} \precsim  \sqrt{n \log \D}$  whereas $\norm{\fkf}_W \sim \binom{n+\D}{n}^{\frac12}$. An analogous situation holds for polynomial systems (see Theorem \ref{theo:normquitentreal} and Proposition~\ref{prop:taildobroinfty}). 
\item Condition numbers with respect to the $L_\infty$-norm yield condition-based complexity estimates (i.e., cost bounds in terms 
of both $n$, $\D$ and a condition number) almost identical to 
those obtained using the condition numbers with respect to Weyl norm (see Section \ref{sec:conditionbasics}). 
\end{enumerate}

In this way, the reduction in the probabilistic estimates 
in passing to $\norm{~}_{\infty}$ from $\norm{~}_W$ 
immediately translates to
reductions in the magnitude of the corresponding condition 
numbers and, in turn, to reductions in the complexity 
estimates.

\smallskip
\noindent{\bf Considered algorithms.} 
We showcase three algorithms where, despite the high cost of
computing the $L_{\infty}$-norm, the reductions in the total
cost bounds remain significant. 

Firstly, in \S~\ref{subsec:gridmethod}, we consider a family of 
algorithms (we refer to them as \emph{grid-based}) that solve
various 
problems in real algebraic and semialgebraic geometry. 
The best numerical algorithms for these problems have exponential complexity.  
In \S\ref{subsec:gridmethod} we replace the Weyl norm by $\|~\|_{\infty}$ 
in the design of one such algorithm (to compute Betti numbers) and in~\S\ref{subsec:ckvskappa} we show a decrease in its cost bounds. 
We take advantage of the fact that there is only one norm 
computation and it is done, so to speak, along the way. The gain 
in the reduction of the estimate for the number of iterations 
directly yields a reduction in the total cost bound (see Corollary~\ref{cor:betti-gain}).

Secondly, in \S\ref{subsec:PV}, we consider the Plantinga-Vegter 
algorithm as it is described and 
analyzed in~\cite{CETC-PV}. Again, replacing the Weyl norm by
$\|~\|_{\infty}$ in the algorithm's design results in 
improved cost bounds. And again, the computation of 
$\|~\|_\infty$ 
is not a burden as it is done only once and its cost 
is dominated by 
that of the rest of the algorithm. The Plantinga-Vegter 
algorithm is 
usually considered with $n=2$ or $n=3$. Remark~\ref{rem:costPV} 
exhibits the improvement achieved on average complexity 
bounds for these two 
cases. For larger values of $n$ the improvement is more substantial. 

Thirdly, in Section~\ref{sec:homotopy}, we consider the problem 
of computing a zero of a system of 
complex quadratic equations.  For this question, a particular 
case of 
Smale's 17th problem, we consider the algorithms proposed 
in~\cite{BePa:11,BuCu11} and, again, we design versions of them 
where the Weyl norm is replaced by $\|~\|_{\infty}$. 
Again, this results in a small, but measurable,
reduction in the cost bounds (from $n^7$ to 
$n^{6.875}$). A crucial fact to achieve this is that, 
even though $n$ is general, we can find an efficient way to 
compute $\|~\|_{\infty}$ using the fact that $\D=2$. 

{\em In all three cases, we are able to show that the use of $L_{\infty}$-norm yields a clear reduction on the estimates for the expected number of iterations.} We 
believe that this is a common pattern. But, in general, 
the reduction in the number of steps does not immediately 
translate into a reduction in total computational cost. 
This motivates the search for efficient algorithms that 
(roughly) approximate $\|~\|_{\infty}$, and for a better 
understanding of the complexity and accuracy of computing with
$L_p$-norms in polynomial spaces.
\smallskip

\noindent{\bf Organization of the paper.}\quad
In Section~\ref{sec:norms}, 
we define the norms that will be considered in this paper and  work out several examples. We also recall basic properties of these norms and highlight their differences from the Weyl norm. Then, in Section \ref{sec:conditionbasics}, we 
define condition numbers $\cM$ and $\cK$  that scale with the $L_{\infty}$-norm . These condition numbers are similar to their widely used Weyl versions
$\mu_{\mathrm{norm}}$ and $\kappa$ (for complex and real problems, 
respectively). We also prove in Section \ref{sec:conditionbasics} that the main properties of 
$\mu_{\mathrm{norm}}$ and $\kappa$ ---those allowing  them to feature in condition-based cost estimates--- hold for 
$\cM$ and $\cK$.  Section \ref{subsec:gridmethod}, Section \ref{subsec:PV}, and Section \ref{sec:homotopy} are the home of three algorithms that are designed using $L_{\infty}$-scaled condition numbers. We compare the cost bounds of these 
algorithms to those for their Weyl counterparts, and 
highlight computational gains.

We conclude the paper, in Section~\ref{sec:condnumber}, with a minor 
digression. Because a natural habitat for functional norms are spaces 
of continuous functions we consider extensions of the real condition 
number $\kappa$ to the space $C^1[q]:=C^1(\bbS^n,\bbR^q)$ and we prove  
(somehow unexpectedly) Condition Number Theorems for these extensions. We do not 
analyze algorithms here. We nonetheless point out that substantial 
literature on algorithms on spaces of continuous functions 
exists~\cite{IBC:88,Plaskota:96,IBC:survey} where these theorems might 
be useful. 


{\small
\tableofcontents
}
\medskip

\noindent
{\bf Acknowledgments.}\quad
The second author is grateful to Hakan and Bahad{\i}r Erg\"ur for their cheerful response to his sudden all-day availability  throughout the pandemic times.
The third author is grateful to Evgenia Lagoda for moral support and Gato Suchen for useful suggestions for this paper. We are thankful to the reviewers of this paper for useful suggestions that helped improving the presentation, and to Khazhgali Kozhasov for pointing an error in a constant used in Proposition~\ref{prop:subgaussiantailbound}.

\section{Norms for polynomials}\label{sec:norms}

Let $\bbF$ be either $\bbR$ or $\bbC$. 
Let also $n,d\in\bbN$, $n,d\ge 1$. We denote by 
$\mcH^{\bbF}_d[1]$ the linear space of homogeneous 
polynomials of degree $d$ in the $n+1$ variables 
$X_0,X_1,\ldots,X_n$ with coefficients in $\bbF$. 
Let $\bfd=(d_1,\ldots,d_q)\in\bbN^q$ and $n\in\bbN$ as above.
We denote by $\Hd^{\bbF}[q]$ the space 
$\mcH^{\bbF}_{d_1}[1]\times\cdots\times\mcH^{\bbF}_{d_q}[1]$.
If $\bbF$ is clear from the context, or if it is not 
relevant to the argument, we will omit the 
superscript.  We will use the following conventions for dimension counting:
\[N_i:=\binom{n+d_i}{d_i}=\dim_{\bbF}\mcH^{\bbF}_{d_i}[1]
\text{\quad and\quad }
 N:=\sum_{i=1}^q\binom{n+d_i}{d_i}=\dim_{\bbF} \Hd^{\bbF}[q].\]
We also use $\D:=\max\{d_1,\ldots,d_q\}$ and denote by $\Delta$ the $q\times q$ diagonal matrix with $d_i$ in its $i$th diagonal entry.

In all what follows, $\bbS^n:=\{x\in\bbR^{n+1}\mid \|x\|_{2}=1\}$ will be the (real) $n$-sphere and $\bbP^n:=\bbC^{n+1}/\bbC^*$ the complex projective space of dimension $n$. We note that there will be no ambiguity, as the sphere is the usual space to work with real polynomials and the projective space the usual one for complex polynomials. 

\begin{remark}
In what follows, we will write $z\in\bbP^n$ instead of
$[z]\in\bbP^n$ and we will assume that the representative $z\in\bbC^{n+1}$
always satisfies $\|z\|_{2}=1$. This simplifies the form of many of our 
definitions. This convention can be made w.l.o.g.~as  
every point in $\bbP^n$ has a representative of norm $1$.
\end{remark}

\subsection{Euclidean norms}
The simplest norm considered on $\Hd\uR[q]$ is the one 
induced by the standard Euclidean inner product in monomial 
basis. Every $f\in\mcH^{\bbF}_d[1]$ can be uniquely represented as
\begin{equation}\label{eq:poly-standard}
  f=\sum_{|\alpha|=d} f_\alpha X^\alpha
\end{equation}
where $\alpha=(\alpha_0,\ldots,\alpha_n)\in{\bbN}^{n+1}$ 
and $|\alpha|=\alpha_0+\cdots+\alpha_n$. The norm induced 
by the standard Euclidean inner product is therefore
\begin{equation}\label{eq:std}
    \|f\|_{\mathrm{std}}:=\sqrt{\sum_{|\alpha|=d}|f_\alpha|^2}.
\end{equation}
For $f=(f_1,\ldots,f_q)\in\Hd[q]$ the norm extends as 
 $\|f\|_{\mathrm{std}}^2:=\|f_1\|_{\mathrm{std}}^2+\cdots+\|f_q\|_{\mathrm{std}}^2$. 

The most commonly used norm 
on $\Hd[q]$ is the {\em Weyl norm}. For a polynomial 
as in~\eqref{eq:poly-standard}, this is given by 
\begin{equation}\label{eq:weylde}
      \|f\|_W:=\sqrt{\sum_{|\alpha|=d}\binom{d}{\alpha}^{-1}
  |f_\alpha|^2}
\end{equation}
where $\binom{d}{\alpha}$ is the multinomial 
coefficient $\frac{d!}{\alpha_0!\ldots\alpha_n!}$. 
Again, for $f\in\Hd[q]$ this extends by 
$\|f\|_W^2:=\|f_1\|_W^2+\cdots+\|f_q\|_W^2$. The Weyl norm is also induced by an inner product, and this inner product is invariant under the action of the unitary group (respectively the orthogonal group when the underlying field is $\bbR$).  It is straightforward to check that, for $f\in\Hd[q]$,
$$
\|f\|_W \le \|f\|_{\mathrm{std}}\le \max_{i\le q}\max_{|\alpha|=d_i}\binom{d_i}{\alpha} \|f\|_W.
$$

 Here, and in all what follows, for any $x\in\bbS^n$ and $f\in\Hd[q]$,
$\diff_xf:\Tg_x\bbS^n\to\bbR^q$ is the derivative of 
$f$ at $x$ restricted to the tangent space $\Tg_x\bbS^n$ 
of $\bbS^n$ at $x$. A similar convention applies in the complex 
case replacing $\bbS^n$ and $\Tg_x\bbS^n$ by $\bbP^n$ and $\Tg_{z}\bbP^n$. The following property (see~\cite[Prop.~16.16]{Condition}) is one of the most important properties of the Weyl norm from the viewpoint of the complexity of numerical algorithms.

\begin{prop}\label{prop:evaluationorthogonal}
For all $x\in\bbS^n$ the map
\[
\Hd[q]\ni f\mapsto \mathrm{ev}_xf:=\left(f(x),\Delta^{-\frac{1}{2}}\diff_x f\right)
\]
is an orthogonal projection from $\Hd[q]$ endowed with the Weyl norm onto $\bbR^{q}\times \Tg_x\bbS^n\simeq \bbR^{q+n}$ equipped  with the standard Euclidean norm. An analogous statement holds in the complex case.\eproof
\end{prop}

\subsection{Functional norms}

We will consider functional norms that arise from evaluating 
polynomials at points on the sphere. One might consider other 
norms (as we do in Section~\ref{sec:condnumber}), but $L_p$-norms suffice for obtaining the computational improvements we aim for. Although in the sequel we will only use the $L_\infty$-norm, we present the full family of $L_p$-norms since we consider that these norms will be useful in the future. Moreover, presenting the full family of $L_p$-norms allows to appreciate how the $L_\infty$ differs and relate to these other norms.

We will consider the two following classes of $L$-norms on $\Hd[q]$:
\begin{enumerate}
\item[($\bbR$)] \emph{Real $L_p$-norm}: For $p \in [1,\infty]$, \[\|f\|_{p}\uR:=\begin{cases}
\displaystyle\max_{x\in\bbS^n}\|f(x)\|_{\infty}=\max_{x\in\bbS^n}\max_i|f_i(x)|&\text{if }p=\infty\\[6pt]
\displaystyle\left(\bbE_{\fkx\in\bbS^n}\|f(\fkx)\|_p^p\right)^{1/p}=\left(\bbE_{\fkx\in\bbS^n}\left(\sum_{i=1}^q|f_i(\fkx)|^p\right)\right)^{1/p}&\text{otherwise}\end{cases}\]
where the expectations are taken over the uniform distribution of the $n$-dimensional sphere $\bbS^n\subseteq\bbR^{n+1}$.
\item[($\bbC$)] \emph{Complex $L_p$-norm}: For $p \in [1,\infty]$, \[\|f\|_{p}\uC:=\begin{cases}
\displaystyle\max_{z\in\bbP^n}\left\|f(z)\right\|_{\infty}=\max_{z\in\bbP^n}\max_i\left|f_i(z)\right|&\text{if }p=\infty\\[6pt]
\displaystyle\left(\bbE_{\fkz\in\bbP^n}\left\|f(\fkz)\right\|_p^p\right)^{1/p}=\left(\bbE_{\fkz\in\bbP^n}\left(\sum_{i=1}^q\left|f_i(\fkz)\right|^p\right)\right)^{1/p}&\text{otherwise}\end{cases}\]
where the expectations are taken over the uniform distribution of the complex $n$-dimensional projective space $\bbP^n:=\bbP^n_{\bbC}$. 
\end{enumerate}

\begin{remark}
In the case of a single polynomial, the definitions above 
become simpler. For $f\in\Hd[1]$,
\[
\|f\|_{p}\uR:=\begin{cases}
\displaystyle\max_{x\in\bbS^n}\abs{f(x)}&\text{if }p=\infty\\[6pt]
\displaystyle\left(\bbE_{\fkx\in\bbS^n}|f(\fkx)|^p\right)^{1/p}&\text{otherwise}\end{cases}~~\text{ and }~~\|f\|_{p}\uC:=\begin{cases}
\displaystyle\max_{z\in\bbP^n}\left|f(z)\right|&\text{if }p=\infty\\[6pt]
\displaystyle\left(\bbE_{\fkz\in\bbP^n}\left|f(\fkz)\right|_p^p\right)^{1/p}&\text{otherwise}\end{cases}
\]
which amount to taking the $p$-mean of $|f|$ over, 
respectively, $\bbS^n$ and $\bbP^n$.
\end{remark}

In general, we will omit the superscript when the context is clear. It will be 
common for us to work with the norms $\|~\|_{p}\uR$ in $\Hd^\bbR[q]$ and the norms $\|~\|_{p}\uC$ in $\Hd^\bbC[q]$.\footnote{Observe, however, that the $\|~\|_{p}\uR$ are also norms for $\Hd^\bbC[q]$ since a complex homogeneous polynomial cannot vanish on the real 
sphere without being zero.} 

Our definition has some arbitrary choices. These are motivated by the following two properties:
\begin{enumerate}
    \item[(D)] For $p\in[1,\infty]$ and $f\in\Hd[q]$,
    \[\|f\|_p^{\bbR}=\left\|\left(\|f_1\|_p^{\bbR},\ldots,\|f_q\|_p^{\bbR}\right)\right\|_p~~\text{and}~~\|f\|_p^{\bbC}
    =\left\|\left(\|f_1\|_p^{\bbC},\ldots,\|f_q\|_p^{\bbC}\right)\right\|_p.\]
    This identity is why we take the $p$-mean of the $p$-norm of $f(x)$ instead of taking the $p$-mean of a fixed norm. 
    \item[(I)] We have actions of the $q$th power of the (real) orthogonal group, $\msO(n+1)^q$, on $\Hd\uR[q]$, given by $(A,f)\mapsto (f_i^{A_i}):=(f_i(A_iX))$. Similarly, we have an action of the $q$th power of the unitary group, $\msU(n+1)^q$, on $\Hd\uC[q]$. The norms $\|~\|_{p}\uR$ and $\|~\|_{p}\uC$ are invariant under these actions.
\end{enumerate}

We perform some simple computations to have a better grasp on the introduced norms.

\begin{exam}[Monomials]\label{ex:monomials}
We consider the value of the norms for a monomial $X^{\alpha}\in\mcH_d[1]$ of degree $d$. In this case we have that for $p\in[1,\infty)$,
\[
\left\|X^\alpha\right\|_p\uR=\left(\frac{\Gamma\left(\frac{n+1}{2}\right)\prod_{i=0}^n\Gamma\left(\frac{p\alpha_i+1}{2}\right)}{\pi^{\frac{n+1}{2}}\Gamma\left(\frac{pd+n+1}{2}\right)}\right)^{\frac{1}{p}}~~\text{and}~~
\left\|X^\alpha\right\|_p\uC=\left(n!\frac{\prod_{i=0}^n\Gamma\left(\frac{p\alpha_i}{2}+1\right)}{\Gamma\left(\frac{pd}{2}+n+1\right)}\right)^{\frac{1}{p}}
\]
where $\Gamma$ is Euler's Gamma function, and that
\[\left\|X^\alpha\right\|_\infty\uR=\left\|X^\alpha\right\|_\infty\uC=\prod_{i=0}^n\left(\frac{\alpha_i}{d}\right)^{\frac{\alpha_i}{2}}=\sqrt{\frac{1}{d^d}\prod_{i=0}^n\alpha_i^{\alpha_i}}.\]

For the calculations of $L_p$-norms of monomials we refer the reader to~\cite{folland2001}. Although the calculation is only illustrated over the reals in the reference, the complex case is similar. For the second one, note that for monomials real and complex
$\infty$-norms are equivalent. Once this is clear, we are just using the
method of Lagrange multipliers to compute the maximum over the sphere.
\end{exam}

\begin{exam}[Linear functions]\label{ex:normoflinearfunction}
Let $\mathds{1}=(1,1,\ldots,1)\in\bbN^q$ and $f\in \mcH_\mathds{1}[q]$. Then $f$ can be identified with a matrix $A$ of size $q\times (n+1)$. We can see that 
\[\|f\|_{\infty}=\|A\|_{2,\infty}:=
\sup_{x\neq0}\frac{\|Ax\|_{\infty}}{\|x\|_2}\]
where $\|~\|_{2,\infty}$ is the operator norm where the domain vector space has the usual Euclidean norm $\|~\|_2$ and the codomain the $\infty$-norm $\|~\|_{\infty}$.

For $p\in[1,\infty)$,
\[
\|f\|_p\uR=\|X_0\|_p\uR\left\|\left(\|A^1\|_2,\ldots,\|A^q\|_2\right)\right\|_p
~~\text{and}~~
\|f\|_p\uC=\|X_0\|_p\uC\left\|\left(\|A^1\|_2,\ldots,\|A^q\|_2\right)\right\|_p
\]
where $A^i$ is the $i$th row of $A$ and $X_0$ is a 
variable (and hence $\|X_0\|^{\bbF}_p$ is given by the expressions in Example~\ref{ex:monomials}). Note that $\left\|\left(\|A^1\|_2,\ldots,\|A^q\|_2\right)\right\|_p$ is just the $p$-norm of the vector of $2$-norms of the rows of $A$.
\end{exam}

\begin{exam}[Sum of squares]
Let $f:=\sum_{i=0}^nX_i^2\in \mcH_{2}[1]$. As this function is constant on the real sphere, we have that for all $p\in[1,\infty]$,
\[\|f\|_p\uR=1.\]
However, on $\bbP^n$, $f$ does not behave as a constant function as it 
has a positive dimensional zero set. Again, arguing as in~\cite{folland2001}, we can conclude that
\[
\|f\|_p\uC=\left(\frac{1}{\pi^{n+1}}\frac{n!}{(n+p)!}\int_{z\in\bbC^{n+1}}|f(z)|^p\enumber^{-|z|^2}\right)^{\frac{1}{p}}
\]
for $p\in[1,\infty)$. Now, if $p$ is even, we can obtain the expression
\[
\|f\|_p\uC=\left(\binom{n+2}{2}^{-1}\sum_{\substack{\alpha\in\bbN^{n+1}\\|\alpha|=p/2}}\binom{p/2}{\alpha}^2\binom{p}{2\alpha}^{-1}\right)^{\frac{1}{p}},
\]
after writing $|f(z)|^p=f(z)^{\frac{p}{2}}\overline{f(z)}^{\frac{p}{2}}$, expanding and using separation of variables. In particular, for $p=2$, we obtain that
\[
\|f\|_2^{\bbC}=\sqrt{\frac{2}{n+2}}\neq 1.
\]
This shows how the norms $\|~\|_p\uC$ may be smaller than their corresponding norm $\|~\|_p\uR$ for $p\in[1,\infty)$.
\end{exam}

\begin{exam}[Cosine polynomials]\label{ex:cosinepolynomials}
Let $d\geq 2$ and consider the family of homogenous polynomials
\[c_d:=\sum_{k=0}^{\lfloor d/2\rfloor}\binom{d}{2k}(-1)^kX^{d-2k}Y^{2k}=\frac{1}{2}(X+iY)^d+\frac{1}{2}(X-iY)^d\in\mcH_d[1].\]
Since $c_d(\cos\theta,\sin\theta)=\cos d\theta$, we have that
\[\|c_d\|_{\infty}\uR=1.\]
Also, $c_d$ is unitarily equivalent to $2^{\frac{d}{2}-1}(X^d+Y^d)$. Hence
\[\|c_d\|_{\infty}\uC=2^{\frac{d}{2}-1},\]
since $\|X^d+Y^d\|_{\infty}\uC=1$ for $d\geq 2$. This shows that for degrees $d\geq 3$, the norms $\|~\|_{\infty}\uR$
and $\|~\|_{\infty}\uC$ disagree on real polynomials.
\end{exam}

The following proposition lists simple inequalities between the functional norms. For a converse of some of the inequalities below, where the $L_\infty$ norm is bounded in terms of $L_p$ norms, see \cite{barvinok2002estimating}.

\begin{prop}\label{prop:norm-bounds}
Let $1\leq p < p'<\infty$ and $\bbF\in\{\bbR,\bbC\}$. Then for all $f\in\Hd\uF[q]$, the following inequalities hold: 
\[\frac{1}{q^{\frac{1}{p}}}\|f\|_p\uF\leq \frac{1}{q^{\frac{1}{p'}}}\|f\|_{p'}\uF\leq \|f\|_{\infty}\uF\leq \|f\|_{\infty}\uC.\]
\end{prop}

\begin{proof}[Sketch of proof]
It is a direct consequence of the inequalities between $p$-means.
\end{proof}

The Weyl norm is essentially a scaled version of the complex $L_2$ norm.

\begin{prop}\label{prop:Weylintregalformula}
Let $f\in \Hd^{\bbC}[q]$, then
\[\|f\|_W=\sqrt{\sum_{i=1}^qN_i
\left(\|f_i\|_{2}^{\bbC}\right)^2}.\]
In particular, for $f\in\Hd^{\bbC}[1]$,
\begin{equation}
\tag*{\qed} \|f\|^{\bbC}_W=\sqrt{N}\|f\|_2^{\bbC}.
\end{equation}
\end{prop}

\begin{proof}[Sketch of proof]
We only need to show this in the case $q=1$. Now, both the Weyl 
norm and the complex $L_2$-norm are unitarily invariant Hermitian
norms of $\mcH_d\uC$. For the Weyl norm, 
see~\cite[Theorem~16.3]{Condition}; for the complex $L_2$-norm, 
this is property (I). Since $\mcH_d\uC$ is an irreducible
representation of $\msU(n+1)$, this means that the two norms are 
equal up to a constant. Using Example~\ref{ex:monomials} with
$f=X_0^d$, one can check that this 
constant is $\sqrt{N}$.
\end{proof}

From Proposition~\ref{prop:evaluationorthogonal} 
we get the following result.

\begin{prop}\label{prop:normcomparisson1}
Let $\bbF\in\{\bbR,\bbC\}$ and $f\in\Hd[q]$. Then for all $p\geq 2$,
\[\|f\|_p\uF\leq \|f\|_W.\]
\end{prop}

\begin{proof}[Sketch of proof]
By Proposition~\ref{prop:evaluationorthogonal}, $f\mapsto f(x)$ is an orthogonal projection with respect to the Weyl norm, and so
$\|f(x)\|_2\leq \|f\|_W$. Hence, for every $x \in S^{n-1}$,
$\|f(x)\|_p \leq  \|f(x)\|_2\leq \|f\|_W$,
where the first inequality follows from Minkowski's inequality.
\end{proof}

We finish this subsection by noting how the $L_\infty$-norms relate 
to the Weyl norm. We note that this is related to the so-called 
best rank-one approximation of a symmetric
tensor~\cite{agrachevkozhasovuschmajew2020,zhanglingqi2012}, and the 
inequality for the real case below was already present
in~\cite[Theorem~2.4]{zhanglingqi2012}.

\begin{prop}\label{prop:inftyvsweyl}
Let $f\in\Hd[q]$. Then
\[\|f\|_\infty\uC\leq \|f\|_W\leq \sqrt{N}\|f\|_\infty\uC.\]
If $f\in\Hd\uR[q]$. Then
\[\|f\|_\infty\uR\leq \|f\|_W\leq (n+1)^{\frac{\D}{2}}\|f\|_\infty\uR.\]
\end{prop}
\begin{proof}
The first part follows from Proposition~\ref{prop:Weylintregalformula} and \ref{prop:normcomparisson1}. The left-hand side of the second part uses Proposition~\ref{prop:normcomparisson1}.  

Now, for $f\in\mcH_d[1]$, Corollary~\ref{cor:higherderivativeestimateF} implies that for each $\alpha$, $|f_\alpha|=\left\|\frac{1}{\alpha!}\diffa_xf\right\|\leq \binom{d}{\alpha}$.  The right-hand inequality follows from here.
\end{proof}

\begin{exam}
Proposition~\ref{prop:inftyvsweyl} is almost optimal for $n=1$.
In~\cite{agrachevkozhasovuschmajew2020}, it was shown that for the 
cosine polynomials $c_d$ of Example~\ref{ex:cosinepolynomials} we have
\[\|c_d\|_W=2^{\frac{d-1}{2}}\]
and that $c_d$ is the real polynomial of real $L_\infty$ 
norm~1 with 
largest Weyl norm. Curiously, in this case, the Weyl norm and 
the complex $L_\infty$ are almost equal, the former being the latter 
times $\sqrt{2}$.
\end{exam}

\subsection{Kellogg's Theorem}

We will denote by $\diffa$ the operation of taking all 
partial derivatives with respect to all variables, i.e., 
$f\mapsto \diffa f$ is a linear map
$\Hd[q]\rightarrow\mathcal{H}_{\bfd-\bbone}[(n+1)q]$ 
and, for $x\in\bbF^{n+1}$, $\diffa_x f:\bbF^{n+1}\to\bbF^q$ 
is a linear map. We will write $\diffa_Xf$, with capital $X$, to emphasize that we view $\diffa_Xf$ as a polynomial tuple in $\mathcal{H}_{\bfd-\bbone}[(n+1)q]$, and $\diffa_xf$, with 
lowercase $x$, to emphasize that we view $\diffa_xf$ as the 
linear map $\bbF^{n+1}\rightarrow \bbF^q$ defined at the point $x$. 
We also recall that $\diff_xf$ is the tangent map
$\Tg_x\bbS^n\rightarrow \bbR^q$ in the real case, and the tangent 
map $\Tg_x\bbP^n\rightarrow \bbC^q$ in the complex case.

The following result plays the role of
Proposition~\ref{prop:evaluationorthogonal} for the infinity 
norm instead of the Weyl one. It is a reformulation of a well-known inequality proved in~\cite{kellog}. 

\begin{theo}[Kellogg's Inequality]\label{theo:kellogg}
Let $\bbF\in\{\bbR,\bbC\}$, $f\in\Hd\uF[q]$ and $v\in\bbF^{n+1}$, 
then
\[\left\|\Delta^{-1}\diffa_X f v\right\|_\infty\uF \leq \|f\|_\infty\uF\|v\|.\]
\end{theo}

\begin{cor}\label{cor:kelloggF}
Let $f\in\Hd\uF[q]$ and $z\in\bbS^n$ (if $\bbF=\bbR$) 
or $z\in\bbP^n$ (if $\bbF=\bbC$). Then
\[
\max\left\{\left\|f(z)\right\|_{\infty},\left\|\Delta^{-1}\diff_z f\right\|_{2,\infty}\right\}
\leq \|f\|_\infty\uF.
\]
\end{cor}

Before proving Theorem~\ref{theo:kellogg} and 
Corollary~\ref{cor:kelloggF} we discuss some features of 
these results.

\begin{remark}
We note that the left-hand side in Corollary~\ref{cor:kelloggF}
is not optimal. In general, 
we have that
\[\|\Delta^{-1}\diffa_xf\|_{2,\infty}=\max_i \sqrt{|f_i(x)|^2+\frac{1}{d_i^2}\|\diff_xf_i\|_{2,\infty}^2}.\]
\end{remark}

The following examples show how the bound of Theorem~\ref{theo:kellogg} 
looks like in a few particular cases.

\begin{exam}\label{ex:cos}
Consider the cosine polynomials $c_d$ of Example~\ref{ex:cosinepolynomials}. A direct computation shows that
\[\frac{1}{d}\diffa_Xc_d\,v=v_Xc_{d-1}-v_Ys_{d-1}\]
where $s_{d-1}:=-\frac{i}{2}(X+iY)^{d-1}+\frac{i}{2}(X-iY)$ is the \emph{sine polynomial} for which $s_{d}(\cos\theta,\sin\theta)=\sin d\theta$.

In the real case, this gives
\[\left\|\frac{1}{d}\diffa_Xc_dv\right\|_{\infty}\uR=\|v\|_{2}=\|c_d\|_\infty\uR\|v\|_{2},\]
using the Cauchy-Schwarz inequality. In the complex case, $\frac{1}{d}\diff_Xc_dv=v_Xc_{d-1}-v_Ys_{d-1}$ is unitarily equivalent to
\[\frac{2^{\frac{d-1}{2}}}{d}\left[(v_{X}-iv_Y)
X^{d-1}+(v_{X}+iv_Y)Y^{d-1}\right].\]
Now, $\left\|(v_{X}-iv_Y)x^{d-1}
+(v_{X}+iv_Y)y^{d-1}\right\|
\leq \sqrt{2} \|v\|_{2}(|x|^{d-1}+|y|^{d-1})
\leq \|v\|_{2}$ for $d\leq 3$ and $v$ real, when $|x|^2+|y|^2\leq 1$. Thus
\[\left\|\frac{1}{d}\diffa_Xc_dv\right\|_{\infty}\uC=\frac{2^d}{d}\|v\|_{2}=\frac{\sqrt{2}}{d}\|c_d\|_\infty\uC\|v\|_{2}.\]
This shows that the real version of Kellogg's theorem is tight for $c_d$, 
but the complex version is not. 
\end{exam}

\begin{exam}\label{ex:monom}
The reverse situation is true for the polynomial 
$X_{0}^d$. One can see that
\[\left\|\frac{1}{d}\diffa_X X_{0}^d e_{0}\right\|_\infty\uC
=\|X_{0}^d\|_\infty\uC.\]
Now it is the complex Kellogg's theorem the one which is tight. We 
note, however, that one might still 
improve Corollary~\ref{cor:kelloggF}. For example, is it possible to substitute $\Delta$ by $\Delta^{\frac{1}{2}}$ in this corollary?
\end{exam}

\begin{remark}
Examples~\ref{ex:cos} and~\ref{ex:monom} motivate the search 
of a randomized Kellog's theorem that holds with high probability 
for random polynomials and has a tighter right-hand side. 
\end{remark}

\begin{proof}[Proof of Theorem~\ref{theo:kellogg}]
We only prove the real case. The complex case is proven in an analogous way (see~\cite[\S 8]{kellog} for the complex version of the results we use in the real case).

By~\cite[Theorem~IV]{kellog}, we have that for all $i$ and all $x\in\bbS^n$,
\[\left|\diffa_xf_i v\right|\leq d_i\|f_i\|_{\infty}\uR\|v\|,\]
since $\diffa_xf_iv$ is the directional derivative of $f$ 
at $x$ in the direction of $v$. Therefore for all $x\in \bbS^n$,
\[\left\|\Delta^{-1}\diffa_xfv\right\|_{\infty}=\max_i \frac{1}{d_i}\left|\diffa_xfv\right|\leq \max_i\|f_i\|_{\infty}\uR\|v\|
=\|f\|_\infty\uR\|v\|.\]
Now, $\left\|\Delta^{-1}\diffa_{X}fv\right\|_{\infty}\uR
=\max_{x\in\bbS^n}\|\Delta^{-1}\diffa_xfv\|_{\infty}$ 
by definition of $\|~\|_{\infty}\uR$, so we are done.
\end{proof}

\begin{remark}
We note that the application of~\cite[Theorem~IV]{kellog} using the scaling with the diagonal matrix was not used in~\cite[Theorem~2.4]{EPR18} and \cite{EPR19}. This can be used to improve by a factor of the degree some of the bounds there.
\end{remark}

\begin{proof}[Proof of Corollary~\ref{cor:kelloggF}]
We only prove the real case, the proof for the complex case 
being essentially the same.
Recall that, by Euler's formula for homogeneous functions,
\begin{equation}\label{eq:Euler}
 \Delta^{-1}\diffa_xfx=f(x).   
\end{equation}
In this way, for $x\in \bbS^n$, $\lambda\in\bbR$ and $w\in\Tg_x\bbS^n=x^{\perp}$,
\[\Delta^{-1}\diffa_xf(\lambda x+w)=\lambda f(x)+\Delta^{-1}\diff_xfw.\]
When $\lambda x+w=x$, this expression yields $f(x)$; and when $\lambda x+w=w$, 
it yields $\Delta^{-1}\diff_xfw$. In this way,
\[\max_{\lambda x+w\neq 0}\frac{\|\Delta^{-1}\diffa_xf(\lambda x+w)\|_{\infty}}{\sqrt{|\lambda|^2+\|w\|^2}}\geq \max\left\{\|f(x)\|_\infty,\max_{v\in\Tg_x\bbS^n\setminus 0}\frac{\|\Delta^{-1}\diff_xv\|_{\infty}}{\|v\|}\right\}.\]
The left-hand side is bounded by $\|f\|_\infty\uR$ by Theorem~\ref{theo:kellogg}, and the right-hand side equals $\max\{\|f(x)\|_{\infty},\|\Delta^{-1}\diff_xf\|_{2,\infty}\}$. Thus the desired inequality follows.
\end{proof}

Following the notations introduced above, we will write 
$\diffa^k_xf$ to denote the $k$th derivative map of $f\in\Hd[q]$ 
at $x\in\bbF^{n+1}$. This is the $k$-multilinear map
$(\bbF^{n+1})^k\rightarrow\bbF^q$ given by the $k$th derivatives of 
$f$ at $x$. Also, $\diffa^k_Xf(v_1,\ldots,v_k)$, where
$v_1,\ldots,v_k\in\bbF^{n+1}$, will denote the corresponding polynomial 
tuple in $\mcH_{\bfd-k\bbone}[q]$.
For a real $k$-multilinear map $A:(\bbR^n)^k\rightarrow \bbR^q$, we define
\begin{equation}
    \|A\|_{2,\infty}\uR:=\sup_{v_1,\ldots,v_k\neq 0}\frac{\|A(v_1,\ldots,v_k)\|_\infty}{\|v_1\|\cdots\|v_k\|}.
\end{equation}
We define $\|A\|_{2,\infty}\uC$ for a complex $k$-multilinear map
$A:(\bbC^n)^k\rightarrow \bbC^q$ in a similar manner. 
Note that, for $k>2$, by 
the following corollary and Example~\ref{ex:cosinepolynomials},
\[\left\|\frac{1}{k!}\diffa_{0}^kc_k\right\|_{2,\infty}\uC
=\left\|c_k\right\|_\infty\uC=2^{\frac{k}{2}-1}>1
=\left\|c_k\right\|_\infty\uR
=\left\|\frac{1}{k!}\diffa_{0}^kc_k\right\|_{2,\infty}\uR,
\]
so for real $A$, $\|A\|_{2,\infty}\uR$ and $\|A\|_{2,\infty}\uC$ 
are not necessarily equal and can differ by a factor exponential in $k$. The following corollary (which is closely related 
to~\cite[Theorem~2.1]{zhanglingqi2012}) will be useful later. 

\begin{cor}\label{cor:higherderivativeestimateF}
Let $f\in\Hd\uF[q]$ and $z\in\bbS^n$ (if $\bbF=\bbR$) 
or $z\in\bbP^n$ (if $\bbF=\bbC$). Then, for all $k\geq 1$ and $v_1,\ldots,v_k\in\bbF^{n+1}$, 
\[
\left\|\frac{1}{k!}\Delta^{-1}\diffa_X^kf(v_1,\ldots,v_k)\right\|_{\infty}\leq \frac{1}{k}\binom{\D-1}{k-1}\|f\|_{\infty}\uF\|v_1\|\cdots \|v_k\|.
\]
In particular, $\left\|\frac{1}{k!}\Delta^{-1}\diffa_z^kf\right\|_{2,\infty}\leq \frac{1}{k}\binom{\D-1}{k-1}\|f\|_{\infty}\uF$.
\end{cor}

\begin{proof}
It follows from Theorem~\ref{theo:kellogg} 
by induction, followed by an application of Corollary~\ref{cor:kelloggF}.
\end{proof}

\begin{remark}
Although the results in this section were proved 
only for $\|~\|^{\bbF}_{\infty}$, some of them can be generalized 
to other norms. For example, similar results can be obtained for $\|~\|^{\bbR}_2$ (see~\cite{seeley1966}) and certainly for other norms. We defer to future 
work the application of these extensions to the analysis of numerical algorithms in algebraic geometry. We also note that Corollary~\ref{cor:kelloggF} for $\bbF=\bbR$ can be generalized to smooth real algebraic varieties other than the sphere (see~\cite{boslevenbermilmantaylor1998}).
\end{remark}

\section{Condition numbers for the \texorpdfstring{$L_\infty$}{Linfty}-norm} \label{sec:conditionbasics}

In this section, we will consider condition numbers that capture ``how near to being singular" a system $f\in\Hd[q]$ is at a point $x\in\bbS^n$. We will define condition numbers and develop a geometric understanding of them for the $L_\infty$-norms defined in the preceding section.

Recall the local and global versions of the 
real condition number
$\kappa$ used in~\cite{CKMW1,CKMW2,CKMW3,CKS16}.  For $f\in\Hd\uR[q]$ and $x\in\bbS^n$, they are defined by
\begin{equation}\label{eq:defkappa}
    \kappa(f,x):=\frac{\|f\|_W}{\sqrt{\|f(x)\|_2^2+\left\|\diff_xf^\dagger\Delta^{1/2}\right\|^{-2}_{2,2}}}~\text{ and }~\kappa(f):=\sup_{y\in\bbS^n}\kappa(f,y).
\end{equation}
Here, for a surjective linear map $A$, 
$A^\dagger:=A^*(AA^*)^{-1}$ denotes its Moore-Penrose inverse~\cite[\S1.6]{Condition}. Also, recall 
the $\mu$-condition number introduced by Shub and Smale~\cite{Bez1}: For $f\in\Hd\uC[q]$ and $\zeta\in\bbP^n$, $\mu(f,\zeta)$ is defined by
\begin{equation}\label{eq:defmu}
    \mu_{\mathrm{norm}}(f,\zeta):=
    \|f\|_W\left\|\diff_{\zeta} f^\dagger\Delta^{1/2}\right\|_{2,2}.
\end{equation}

\begin{remark}
By convention, we assume that $\|A^\dagger\|_{2,2}=\infty$ when $A$ is not surjective. We do this, because, for $A\in\bbC^{q\times n}$ surjective,
\[
\left\|A^\dagger\right\|^{-1}_{2,2}=\sigma_q(A)
\]
where $\sigma_q$ is the $q$th singular value. As the latter is continuous, this choice guarantees that $A\mapsto \|A^\dagger\|_{2,2}^{-1}$ is continuous.
\end{remark}

Following these ideas we define the 
\emph{real local condition number} ---of 
$f\in\Hd\uR[q]$ at $x\in\bbS^n$--- as
\begin{equation}\label{eq:realcond}
\cK(f,x):=\frac{\sqrt{q}\|f\|_\infty\uR}
{\max\left\{\|f(x)\|,\left\|\diff_xf^\dagger\Delta\right\|_{2,2}^{-1}\right\}}
\end{equation}
and the \emph{real global condition number} 
---of $f\in\Hd\uR[q]$--- as
\begin{equation}\label{eq:realcondglob}
\cK(f):=\sup_{y\in\bbS^n}\cK(f,y).
\end{equation}
And we define the \emph{complex local condition number} 
---of $f\in\Hd\uC[q]$ at $\zeta\in\bbP^n$---
as
\begin{equation}\label{eq:complexcond}
\cM(f,\zeta)=\sqrt{q}\|f\|_{\infty}\uC\left\|\diff_{\zeta}f^\dagger\Delta\right\|_{2,2}
\end{equation}
and the \emph{complex global condition number} 
---of $f\in\Hd\uC[q]$ (with $q\le n$)--- as
\begin{equation}\label{eq:complexcondglob}
\cM(f):=\sup\{\cM(f,\zeta)\mid \zeta\in\bbP^n,\,f(\zeta)=0\}.
\end{equation}

We can see that $\cK$ is a variant of $\kappa$ 
and that $\cM$ is a variant of $\mu_{\mathrm{norm}}$.
We note that the main
difference lies in the fact that we are substituting all occurrences of $\|~\|_W$ with occurrences of $\|~\|_\infty$. The fact that we 
use a different scaling factor ($\Delta^{1/2}$ instead of $\Delta$) 
or different norms for vectors ($\|~\|_\infty$ instead of 
$\|~\|_2$ and so on) only affects these quantities up to a 
$\sqrt{2q\D}$ factor. This have little consequences for 
complexity. We will be more explicit in 
Proposition~\ref{prop:kappavsck}. Note that despite these 
changes, we still have that the local condition numbers, 
$\cK$ and $\cM$, become $\infty$ at a singular zero and that 
they are finite otherwise.

The remainder of this section is devoted to prove the main 
properties of $\cK$ and $\cM$, which are the reason we have defined 
these numbers the way we did. The properties we will show are those 
needed for a condition-based complexity analyses of the algorithms
in Sections~\ref{sec:real} and~\ref{sec:homotopy} 
following the lines of the analyses
in~\cite{CKMW1,CKS16,BCL17,BCTC1,BCTC2} 
(see also~\cite{tonellicuetothesis}) and
in~\cite[Ch.~17]{Condition}. 

\subsection{Properties of the real condition number \texorpdfstring{$\cK$}{K}}

Recall (see, e.g.,~\cite[Def.~16.35]{Condition}) that for $f\in\Hd[q]$ and $x\in\bbS^n$, the \emph{Smale's projective gamma} is given by
\[\gamma(f,x):=\sup_{k\geq2}
\left\|\frac{1}{k!}\diff_xf^\dagger\diffa_x^kf\right\|^{\frac{1}{k-1}}\]
where $\|~\|=\|~\|_{2,2}$ is the operator norm (with respect to
Euclidean norms) of a multilinear map.

\begin{theo}\label{theo:propertiesrealcond}
Let $f\in\Hd\uR[q]$ and $x\in\bbS^n$. The following holds:
\begin{itemize}
\item \textbf{Regularity Inequality:} Either
\[\frac{\|f(x)\|}{\sqrt{q}\|f\|_\infty\uR}\geq\frac{1}{\cK(f,x)}\text{ or }\sqrt{q}\|f\|_\infty\uR\left\|\diff_xf^\dagger\Delta\right\|_{2,2}\leq\cK(f,x).\]
In particular, if $\cK(f,x)\frac{\|f(x)\|}{\sqrt{q}\|f\|_\infty\uR}<1$, then $\diff_xf:\Tg_x\bbS^n\rightarrow \bbR^q$ is surjective and its pseudoinverse $(\diff_xf)^{\dagger}$ exists.
\item \textbf{1st Lipschitz property:} The maps

\noindent\begin{minipage}{.45\linewidth}\centering
\begin{align*}
    \Hd\uR[q]&\rightarrow [0,\infty)\\
    g&\mapsto \frac{\|g\|_\infty\uR}{\cK(g,x)}
\end{align*}
\end{minipage}%
\begin{minipage}{.1\linewidth}\centering
and
\end{minipage}%
\noindent\begin{minipage}{.45\linewidth}\centering
\begin{align*}
    \Hd\uR[q]&\rightarrow [0,\infty)\\
    g&\mapsto \frac{\|g\|_\infty\uR}{\cK(g)}
\end{align*}
\end{minipage}\vskip.3\baselineskip
are $1$-Lipschitz with respect the real $L_\infty$-norm. In particular, \[\cK(f,x)\geq 1~\text{ and }~\cK(f)\geq 1.\]
\item \textbf{2nd Lipschitz property:} The map
\begin{align*}
    \bbS^n&\rightarrow [0,1]\\
    y&\mapsto \frac{1}{\cK(f,y)}
\end{align*}
is $\D$-Lipschitz with respect the geodesic distance on $\bbS^n$.
\item \textbf{Higher Derivative Estimate:} If $\cK(f,x)\frac{|f(x)|}{\|f\|_\infty\uR}<1$, then
\[
\gamma(f,x)\leq \frac{1}{2}(\D-1)\cK(f,x).
\]
\end{itemize}
\end{theo}

We now discuss the role of the above properties.

\textbf{Regularity Inequality}. The regularity inequality guarantees that, when $\cK(f,x)<\infty$, either $x$ is far 
away from the zero set of $f$ or $\diff_xf^\dagger$ exist 
and is well-defined. The latter is important, because it 
allows us to do various geometric arguments that rely on 
this pseudoinverse being defined or, equivalently, on 
$\diff_x f$ being surjective. In the particular case 
of $\cK$ we could state it with equalities (see its proof 
below) but we leave the statement with inequalities 
as this is the one holding for $\kappa$ as well and 
it is enough for our purposes.

\textbf{1st Lipschitz Property}. The main use of the 1st Lipschitz inequality is to control the variation of $\cK$ with respect to $f$. 
This property implies that
\begin{equation}
    \frac{1-\frac{\left\|f-\tilde{f}\right\|_\infty\uR}{\left\|f\right\|_\infty\uR}}{1+\cK(f,x)\frac{\left\|f-\tilde{f}\right\|_\infty\uR}{\left\|f\right\|_\infty\uR}}\cK(f,x)
    \leq\cK\left(\tilde{f},x\right)
    \leq \frac{1+\frac{\left\|f-\tilde{f}\right\|_\infty\uR}{\left\|f\right\|_\infty\uR}}{1-\cK(f,x)\frac{\left\|f-\tilde{f}\right\|_\infty\uR}{\left\|f\right\|_\infty\uR}}\cK(f,x)
\end{equation}
whenever $\cK(f,x)\frac{\left\|f-\tilde{f}\right\|_\infty\uR}{\left\|f\right\|_\infty\uR}<1$. This formula shows how the condition number of an approximation of $f$ relates to that of $f$.

\textbf{2nd Lipschitz Property}. The 2nd Lipschitz property
allows us to gauge the variation of $\cK$ with respect to $x$. 
In this sense, it is very similar to the first Lipschitz 
property and it implies that
\begin{equation}
    \frac{1}{1+\cK(f,x)\dist_\bbS(x,\tilde{x})}\cK(f,x)
    \leq\cK\left(f,\tilde{x}\right)\leq
    \frac{1}{1-\cK(f,x)\dist_\bbS(x,\tilde{x})}\cK(f,x)
\end{equation}
whenever $\cK(f,x)\dist_\bbS(x,\tilde{x})<1$. 
Here $\dist_{\bbS}$ denotes the geodesic distance 
in $\bbS^n$. 

\textbf{Higher Derivative Estimate}. Smale's projective gamma, $\gamma(f,\zeta)$, controls many aspects of the local geometry around a zero $\zeta$ of the function $f$. Notably, in the case $q=n$, the radius of the basin of attraction at $\zeta$ 
of Newton's operator $N_f$ associated with $f$. 
Recall (see~\cite[Def.~16.34]{Condition}) that 
we say that $x\in\bbS^n$ is an {\em approximate zero} 
of $f\in\Hd[n]$ with associated zero $\zeta\in\bbS^n$ 
when for all $k\geq 1$, the $k$th iteration $N_f^k$ 
of $N_f$ satisfies
$$
   \dist_{\bbS}(N_f^k,x)\le \left(\frac12\right)^{2^k-1}
   \dist_{\bbS}(x,\zeta).
$$
We have the following 
result (see~\cite[Thm.~16.38 and Table~16.1]{Condition}).

\begin{theo}\label{thm:gamma}
Let $f\in\Hd[n]$ and $\zeta\in\bbS^n$ such that $f(\zeta)=0$. 
Let $z\in\bbS^n$ be such that 
$\dist_{\bbS}(z,\zeta)\leq \frac{1}{45}$ and 
$\dist_{\bbS}(z,\zeta)\gamma(f,\zeta)\le 0.17708$. Then, 
$z$ is an approximate zero of $f$ with associated zero 
$\zeta$. \eproof
\end{theo}

The computation of $\gamma(f,x)$ appears to require 
all the derivatives of $f$. The Higher 
Derivative Estimate allows one to estimate 
$\gamma(f,x)$ in terms of the first derivative only. 

\begin{proof}[Proof of Theorem~\ref{theo:propertiesrealcond}]
\textbf{Regularity Inequality.} By definition,
\[\frac{1}{\cK(f,x)}=\max\left\{\frac{\|f(x)\|}{\sqrt{q}\|f\|_\infty\uR},\frac{1}{\sqrt{q}\|f\|_\infty\uR\left\|\diff_xf^\dagger\Delta\right\|_{2,2}}\right\}.\]
Hence either $\frac{1}{\cK(f,x)}=\frac{\|f(x)\|}{\sqrt{q}\|f\|_\infty\uR}$ or $\cK(f,x)=\sqrt{q}\|f\|_\infty\uR\left\|\diff_xf^\dagger\Delta\right\|_{2,2}$, which finishes the proof.
\smallskip

\textbf{1st Lipschitz property.} We have that
\[
\frac{\|g\|_\infty\uR}{\cK(g,x)}
=\max\left\{\frac{\|g(x)\|}{\sqrt{q}},
\frac{\sigma_q\left(\Delta^{-1}\diff_xg\right)}{\sqrt{q}}\right\}.
\]
Hence, we only need to show that $g\mapsto \|g(x)\|/\sqrt{q}$ and $g\mapsto \sigma_q\left(\Delta^{-1}\diff_xg\right)/\sqrt{q}$ are $1$-Lipschitz. Now, 
\[
\left|\frac{\|g(x)\|}{\sqrt{q}}-\frac{\left\|\tilde{g}(x)\right\|}{\sqrt{q}}\right|\leq \frac{\left\|\left(g-\tilde{g}\right)(x)\right\|}{\sqrt{q}}\leq \left\|\left(g-\tilde{g}\right)(x)\right\|_{\infty}\leq \left\|g-\tilde{g}\right\|_\infty\uR,
\]
by the reverse triangle inequality, $\|~\|\leq \sqrt{q}\|~\|_\infty$ and the definition of 
the real $L_{\infty}$-norm; and
\[
\left|\frac{\sigma_q\left(\Delta^{-1}\diff_xg\right)}{\sqrt{q}}-\frac{\sigma_q\left(\Delta^{-1}\diff_x\tilde{g}\right)}{\sqrt{q}}\right|\leq \frac{\left\|\Delta^{-1}\diff_x\left(g-\tilde{g}\right)\right\|_{2,2}}{\sqrt{q}}\leq \left\|\Delta^{-1}\diff_x\left(g-\tilde{g}\right)\right\|_{\infty,2}\left\|g-\tilde{g}\right\|_\infty\uR,
\]
because $\sigma_q$ is $1$-Lipschitz with respect to  $\|~\|_{2,2}$, $\|~\|\leq \sqrt{q}\|~\|_\infty$ and Kellogg's Inequality (Theorem~\ref{theo:kellogg}). Thus our claims follow.

The claim for $g\mapsto \|g\|_\infty\uR/\cK(g)$ follows from the fact that the minimum of a family of $1$-Lipschitz functions is $1$-Lipschitz and from
\[
\frac{\|g\|_\infty\uR}{\cK(g)}
=\min_{x\in\bbS^n}\frac{\|g\|_\infty\uR}{\cK(g,x)}.
\]

For the lower bound, just note that
\[\frac{\|f\|_\infty\uR}{\cK(f,x)}=\left|\frac{\|f\|_\infty\uR}{\cK(f,x)}-\frac{\|0\|_\infty\uR}{\cK(0,x)}\right|\leq \|f-0\|_{\infty}\uR=\|f\|_\infty\uR\]
by the proven Lipschitz property, and so $\cK(f,x)\geq 1$. Similarly with $\cK(f)$.
\smallskip

\textbf{2nd Lipschitz property.} Without loss of generality, assume that $\|f\|_\infty\uR=1$, after scaling $f$ by an appropriate constant ---note that this does not change the value of $\cK$---. Let $y,\tilde{y}\in \bbS^n$ and $u\in\msO(n+1)$ be the planar rotation taking $y$ into $\tilde{y}$. Then 
\[
\left|\frac{1}{\cK\left(f,y\right)}-\frac{1}{\cK\left(f,\tilde{y}\right)}\right|=\left|\frac{1}{\cK\left(f,y\right)}-\frac{1}{\cK\left(f^u,y\right)}\right|\leq \|f-f^u\|_\infty\uR,
\]
where $f^u:=f(uX)$ and where the equality follows from the fact that the $L_\infty$-norm is orthogonally invariant along with 
the inequality from the 1st Lipschitz property.

Now, arguing as when proving the 1st Lipschitz property, we have that for all $z\in \bbS^n$,
\[\left|f(z)-f(uz)\right|\leq \D\,\dist_{\bbS}(z,uz).\]
By the choice of $u$, we have that $\dist_\bbS(z,uz)\leq \dist_\bbS(y,\tilde{y})$. Therefore $\|f-f^u\|_\infty\uR\leq \D\,\dist_\bbS(y,\tilde{y})$ and we are done.

We note that a variational argument showing that both $y\mapsto \|g(y)\|/\sqrt{q}$ and $y\mapsto \sigma_q(\Delta^{-1}\diff_yf))/\sqrt{q}$ are Lipschitz is possible. This argument would be almost identical to the one used for proving the 1st Lipschitz property, but varying the point in the sphere instead of the polynomial. We use the above argument since it is simpler and it gives a slightly better bound. 
\smallskip

\textbf{Higher Derivative Estimate.} Again, without loss of generality, we assume that $\|f\|_\infty\uR=1$, since 
multiplying $f$ by a scalar affects neither the value of 
$\cK$ nor Smale's projective gamma. Then
\begin{align*}
    \left\|\frac{1}{k!}\diff_xf^\dagger\diffa_x^kf\right\|&\leq \left\|\diff_xf^\dagger\Delta\right\|_{2,2}\left\|\frac{\Delta^{-1}}{k!}\diffa_x^kf\right\|_{2,2} &\text{(Inequalities for operator norms)}\\
    &\leq \sqrt{q}\left\|\diff_xf^\dagger\Delta\right\|_{2,2}\left\|\frac{\Delta^{-1}}{k!}\diffa_x^kf\right\|_{2,\infty} &\|~\|/\sqrt{q}\leq \|~\|_\infty\\
    &\leq \cK(f,x)\left\|\frac{\Delta^{-1}}{k!}\diffa_x^kf\right\|_{2,\infty} &\text{(Assumption + Regularity Inequality)}\\
    &\leq \frac{1}{k}\binom{\D-1}{k-1}\cK(f,x).&\text{(Corollary~\ref{cor:higherderivativeestimateF})}
\end{align*}
Taking $(k-1)$th roots, we have that $\cK(f,x)^{\frac{1}{k-1}}\leq \cK(f,x)$, since $\cK(f,x)\geq 1$ by Corollary~\ref{cor:kelloggF}; and that
\[
\left(\frac{1}{k}\binom{\D-1}{k-1}\right)^{\frac{1}{k-1}}\leq \frac{\D-1}{2},
\]
using that $\frac{1}{k}\binom{\D-1}{k-1}\leq (\D-1)^{k-1}/2^{k-1}$. Putting this together, we obtain the desired bound for Smale's projective gamma.
\end{proof}


The following proposition, which we state here for the sake of 
completeness, will be proved in Section~\ref{sec:condnumber}.

\begin{prop}\label{prop:condnumbtheo}
Let $f\in\Hd\uR[q]$ and $x\in \bbS^n$. Then
\[
\frac{\|f\|_{\infty}\uR}{\dist_{\infty}\uR(f,\Sigma_{\bfd,x}\uR[q])}\leq \cK(f,x)\leq2\sqrt{\sum_{i=1}^qd_i^2}\,\frac{\|f\|_{\infty}\uR}{\dist_{\infty}\uR(f,\Sigma_{\bfd,x}\uR[q])}
\]
and
\[
\frac{\|f\|_{\infty}\uR}{\dist_{\infty}\uR(f,\Sigma_{\bfd}\uR[q])}\leq \cK(f)\leq2\sqrt{\sum_{i=1}^qd_i^2}\,\frac{\|f\|_{\infty}\uR}{\dist_{\infty}\uR(f,\Sigma_{\bfd}\uR[q])}
\]
where $\dist_\infty\uR$ is the distance induced by $\|~\|_\infty\uR$, 
\[
  \Sigma_{\bfd,x}\uR[q]:=\left\{g\in \Hd\uR[q]\mid g(x)=0,\,\mathrm{rank}\,\diff_x g < q\right\},~\text{ and }~\Sigma_{\bfd}\uR[q]:=\bigcup_{x\in\bbS^n}\Sigma_{\bfd,x}\uR[q]. 
\]
\end{prop}

\subsection{Properties of the complex condition number \texorpdfstring{$\cM$}{M}}

In the complex case, Theorem~\ref{theo:propertiesrealcond} 
takes the form of the following result, whose proof is identical and so 
we omit it. We do not consider 
a regularity inequality for $\cM$ since over complex numbers 
one  usually considers $\cM(f,\zeta)$ 
for a zero $\zeta$ of $f$ (or a point nearby).

\begin{theo}\label{theo:propertiescomplexcond}
Let $f\in\Hd\uC[q]$ and $\zeta\in\bbP^n$. The following holds:
\begin{itemize}
\item \textbf{1st Lipschitz property:} The maps

\noindent\begin{minipage}{.45\linewidth}\centering
\begin{align*}
    \Hd\uC[q]&\rightarrow [0,\infty)\\
    g&\mapsto \frac{\|g\|_\infty\uC}{\cM(g,\zeta)}
\end{align*}
\end{minipage}%
\begin{minipage}{.1\linewidth}\centering
and
\end{minipage}%
\noindent\begin{minipage}{.45\linewidth}\centering
\begin{align*}
    \Hd\uC[q]&\rightarrow [0,\infty)\\
    g&\mapsto \frac{\|g\|_\infty\uC}{\cM(g)}
\end{align*}
\end{minipage}\vskip.3\baselineskip
are $1$-Lipschitz with respect the complex $L_\infty$-norm. In particular, \[\cM(f,\zeta)\geq 1~\text{ and }~\cM(f)\geq 1.\]
\item \textbf{2nd Lipschitz property:} The map
\begin{align*}
    \bbP^n&\rightarrow [0,1]\\
     \eta&\mapsto \frac{1}{\cM(f,\eta)}
\end{align*}
is $\D$-Lipschitz with respect the geodesic distance 
$\dist_{\bbP}$ on $\bbP^n$.
\item \textbf{Higher Derivative Estimate:} We have
\begin{equation}\tag*{\qed}
\gamma(f,\zeta)\leq \frac{1}{2}(\D-1)\cM(f,\zeta).
\end{equation}
\end{itemize}
\end{theo}

We finish with the following proposition, which combines the 1st and 2nd Lipschitz properties of $\cM$, as it will play a fundamental role in our analysis of linear homotopy in Section~\ref{sec:homotopy}. We note that this proposition is to $\cM$ what \cite[Proposition~16.55]{Condition} is to $\mu_{\mathrm{norm}}$.

\begin{prop}\label{prop:B61}
Let $f,\tilde{f}\in\Hd\uC[q]$, $\zeta,\tilde{\zeta}\in\bbP^n$ and $\varepsilon\in (0,1)$. If
\[
\cM(f,\zeta)\max\left\{\frac{2\|\tilde{f}-f\|_\infty\uR}{\|f\|_{\infty}\uC},\D\,\dist_{\bbP}(\zeta,\tilde{\zeta})\right\}\leq\frac{\varepsilon}{4},
\]
then
\[\frac{1}{1+\varepsilon}\cM\left(f,\zeta\right)\leq \cM\left(\tilde{f},\tilde{\zeta}\right)\leq (1+\varepsilon)\cM\left(f,\zeta\right).\]
\end{prop}
\begin{proof}
Note that
\[
\left|\frac{1}{\cM(f,\zeta)}-
\frac{1}{\cM\left(\tilde{f},\tilde{\zeta}\right)}\right|
\leq \left|\frac{1}{\cM(f,\zeta)}
-\frac{1}{\cM\left(\tilde{f},\zeta\right)}\right|
+\left|\frac{1}{\cM\left(\tilde{f},\zeta\right)}
-\frac{1}{\cM\left(\tilde{f},\tilde{\zeta}\right)}\right|.
\]

For the first term in the sum, we have
\[\left|\frac{1}{\cM\left(f,\zeta\right)}
-\frac{1}{\cM\left(\tilde{f},\zeta\right)}\right|
=\left|\frac{1}{\cM\left(\frac{f}{\|f\|_\infty\uC},\zeta\right)}
-\frac{1}{\cM\left(\frac{\tilde{f}}
{\|\tilde{f}\|_\infty\uC},\zeta\right)}\right|
\leq \left\|\frac{f}{\|f\|_\infty\uC} 
-\frac{\tilde{f}}{\|\tilde{f}\|_\infty\uC}\right\|_\infty\uC
\]
by the 1st Lipschitz property of $\cM$ (Theorem~\ref{theo:propertiescomplexcond}). Now,
\[
\left\|\frac{f}{\|f\|_\infty\uC}-\frac{\tilde{f}}{\|\tilde{f}\|_\infty\uC}\right\|_\infty\uC\leq \left\|\frac{f}{\|f\|_\infty\uC}-\frac{\tilde{f}}{\|f\|_\infty\uC}\right\|_\infty\uC+\left\|\frac{\tilde{f}}{\|f\|_\infty\uC}-\frac{\tilde{f}}{\|\tilde{f}\|_\infty\uC}\right\|_\infty\uC\leq \frac{2\|\tilde{f}-f\|_\infty\uC}{\|f\|_\infty\uC}.
\]

For the second term, we have 
\[
\left|\frac{1}{\cM\left(\tilde{f},\zeta\right)}\right|
+\left|\frac{1}{\cM\left(\tilde{f},\zeta\right)}
-\frac{1}{\cM\left(\tilde{f},\tilde{\zeta}\right)}\right|
\leq \D\,\dist_{\bbP}(\zeta,\tilde{\zeta})
\]
by the 2nd Lipschitz property of $\cM$ (Theorem~\ref{theo:propertiescomplexcond}).

Hence, we have
\[
\left|\frac{1}{\cM(f,\zeta)}
-\frac{1}{\cM\left(\tilde{f},\tilde{\zeta}\right)}\right|
\leq \frac{2\|\tilde{f}-f\|_\infty\uC}{\|f\|_\infty\uC}+
\D\,\dist_{\bbP}(\zeta,\tilde{\zeta}).
\]
By assumption, after multiplying by $\cM(f,\zeta)$, we have
\[
\left|1-\frac{\cM(f,\zeta)}{\cM\left(\tilde{f},\tilde{\zeta}\right)}\right|\leq \frac{\varepsilon}{2}
\]
and so, from
\[
1-\frac{\cM(f,\zeta)}{\cM\left(\tilde{f},\tilde{\zeta}\right)}\leq \frac{\varepsilon}{2}~\text{ and }~\frac{\cM(f,\zeta)}{\cM\left(\tilde{f},\tilde{\zeta}\right)}-1\leq \frac{\varepsilon}{2},
\]
we get
\[
\frac{1}{1+\frac{\varepsilon}{2}}\cM(f,\zeta)\leq \cM\left(\tilde{f},\tilde{\zeta}\right)\leq \frac{1}{1-\frac{\varepsilon}{2}}\cM(f,\zeta).
\]
Since $\varepsilon<1$, the desired inequalities follow.
\end{proof}

\section{Numerical Algorithms in Real Algebraic Geometry}
\label{sec:real}

There is a growing literature on numerical algorithms that 
addresses basic
computational tasks in  real algebraic geometry such as counting real 
zeros~\cite{CKMW1,CKMW2,CKMW3}, computing homology of 
algebraic~\cite{CKS16} and semialgebraic
sets~\cite{BCL17,BCTC1,BCTC2}, 
and meshing real curves and surfaces~\cite{PV,CETC-PV}. These 
works rely on condition numbers to control precision, and to 
estimate computational complexity. 

In this section we show how the complexity estimates in these 
works are improved by using the real $L_\infty$-norm in the algorithm's design. 
These improvements rely on three observations: 
\begin{enumerate}
\item The only properties of the real condition number $\kappa$ 
that are used in the complexity analyses are those stated in
Theorem~\ref{theo:propertiesrealcond}: the regularity inequality, the 1st 
and 2nd Lipschitz properties and the Higher Derivative Estimate. 
As these properties hold as well for $\cK$, 
an almost identical condition-based cost analysis can be 
derived when we pass from the Weyl norm to the real $L_\infty$-norm 
and from $\kappa$ to $\cK$. We showcase this in~\S\ref{subsec:gridmethod}
and~\S\ref{subsec:PV}.
\item When we consider random
input models, the gains in the complexity estimates become more evident. In~\S\ref{subsec:ckvskappa}, we show 
that the ratio of the new $\cK$ to $\kappa$ is, typically,  of the order of $\sqrt{n}/\sqrt{N}$ 
for a random polynomial system. 
Since $N \sim n^d$ for $n >d$ 
and $N \sim d^n$ for $d>n$, this yields a significant reduction 
in the complexity estimates .
\item Computing the Weyl norm is cheaper than computing the real $L_\infty$-norm, but this does not affect the overall complexity: We only compute the 
$L_\infty$-norm once, and the cost of this computation is 
dominated by that of the remaining steps.
\end{enumerate}

In what follows, we will focus on algorithms dealing with real algebraic sets. The algorithms we have in mind are the ones in~\cite{CKMW1,CKMW2,CKMW3,CKS16} and 
the Plantinga-Vegter algorithm~\cite{PV} as described and 
analyzed in~\cite{CETC-PVjorunal} (cf.~\cite{CETC-PV}). Our condition number  $\cK$ as defined in preceding section will improve the overall computational complexity of these algorithms. Similar results can be obtained for the algorithms dealing with
semialgebraic sets in~\cite{BCL17,BCTC1,BCTC2} (cf.~\cite{tonellicuetothesis}) using 
natural extensions $\overline{\cK}$ and $\cK_*$ of 
the condition numbers $\kappabar$ and $\kappa_*$ used in 
these papers.

\subsection{A grid-based algorithm and its condition-based complexity}\label{subsec:gridmethod}

A \emph{grid-based algorithm} is a subdivision-based method which 
constructs a grid to discretize the original problem and solves 
the latter by working on the grid points only (selecting and 
finding proximity relations between its points). The algorithms 
in~\cite{CKMW1,CKMW2,CKMW3},~\cite{CKS16}, and~\cite{BCL17,BCTC1,BCTC2} (cf.~\cite{tonellicuetothesis}) 
are grid-based. Their basic structure is (simplifying to the extreme) 
the following:
\begin{enumerate}
    \item Estimate the condition number of the problem (with a 
    sequence of grids of increasing fineness).
    \item Create an extra grid (if necessary), whose mesh is determined by the condition number.
    \item Select points in the grid and use them to obtain a solution to 
    the problem.
\end{enumerate}
In general, grid-based algorithms have complexity $\Omega(\D^n)$. This 
fact allows us to estimate the norm $\|f\|_\infty\uR$ of the data $f$ 
without affecting the overall complexity of the algorithms. Moreover, 
the fact that $\cK$ is smaller than $\kappa$ results in a cost reduction.

In this subsection, we focus on an algorithm for the computation of the 
Betti numbers of a spherical algebraic set. This covers the case of 
counting zeros of a square polynomial system treated
in~\cite{CKMW1,CKMW2,CKMW3} and the computation of the Betti numbers of 
a projective real variety~\cite{CKS16}. For simplicity of exposition 
we omit some computational aspects: 1) The presentation we do of 
the algorithms follows the construction-selection paradigm 
of~\cite{BCL17,BCTC1,BCTC2} instead of the inclusion-exclusion paradigm of~\cite{CKMW1,CKMW2,CKMW3,CKS16}. 
This makes easier the exposition of the algorithms without compromising their
computational complexity. 2) We focus on Betti numbers to avoid describing 
the more involved computation of torsion coefficients in the 
homology groups. 
3) We deal with neither parallelization nor finite precision. The 
interested reader can find details about these in the cited
references.

The backbone of existing grid-based algorithms in numerical real algebraic
geometry~\cite{CKMW1,CKMW2,CKMW3,CKS16,BCL17,BCTC1,BCTC2} is an effective
construction of spherical nets. The basic construction was done originally
in~\cite{CKMW1} and it is based on projecting the uniform grid in the 
boundary of a unit cube onto the unit sphere. 

Recall that a \emph{(spherical) $\delta$-net} is a finite subset 
$\mcG\subset \bbS^n$ such that for all $x\in\bbS^n$, 
$\dist_\bbS(x,\mcG)< \delta$. We will omit the term `spherical' as all 
nets we consider are so.

\begin{prop} \label{prop:gridconstructor}
There is an algorithm \textsc{GRID} that on input $(n,k)\in\bbN\times \bbN$ outputs a $2^{-k}$-net $\mcG_k\subset\bbS^n$ with
\[
|\mcG_k|=\Oh\left(2^{n\log n+nk}\right).
\]
The cost of this algorithm is $\Oh\left(2^{n\log n+nk}\right)$. 
\eproof
\end{prop}

\begin{remark}
The grid construction in Proposition~\ref{prop:gridconstructor},  
which occurs 
in~\cite{CKMW1,CKMW2,CKMW3,CKS16,BCL17,BCTC1,BCTC2}, 
is not optimal. This is due to the $2^{n\log n}$ factor 
in the estimates, which can be decreased to 
$2^{\Oh(n)}$. An algorithm doing this, i.e., 
constructing a spherical $2^{-k}$-net of size $2^{\Oh(n)}2^{k(n+1)}$ 
in $2^{\Oh(n)}2^{k(n+1)}$-time is given
in~\cite[Theorem~1.9(1)]{alonleeshraibmanvempala2013}.
We use the sub-optimal result of Proposition~\ref{prop:gridconstructor} 
to focus on the effect 
of just changing the norm when comparing between the old and new
versions of the algorithms. But we observe here that by using the 
nets in~\cite{alonleeshraibmanvempala2013} one can remove the 
$\log(n)$ factors in the exponents.
\end{remark}

\subsubsection{Computation of \texorpdfstring{$\|~\|_\infty\uR$}{||||inftyR}}

The following is an easy consequence of Kellogg's theorem.

\begin{prop}\label{prop:estimationnorm}
Let $f\in\Hd\uR[q]$ and $\mcG\subset\bbS^n$ be a $\delta$-net. 
If $\D \delta <\sqrt{2}$, then
\[
\max_{x\in\mcG}\|f(x)\|_\infty\leq \|f\|_\infty\uR\leq \frac{1}{1-\frac{\D^2}{2}\delta^2} \max_{x\in\mcG}\|f(x)\|_\infty.
\]
\end{prop}

\begin{proof}
We only need to show the right-hand inequality, the other being trivial. 
Without loss of generality, assume that $q=1$, i.e., $f$ is a homogeneous
polynomial of degree $\D$. 

Let $x_*$ be the maximum of $|f|$ on $\bbS^n$,
$x\in\mcG$ such that $\dist_{\bbS}(x_\ast,x)\leq \delta$ and 
$[0,1]\ni t\mapsto x_t$ the geodesic on $\bbS^n$ going from $x_*$ 
to $x$ with constant speed. Then, for the function $t\mapsto M(t):=f(x_t)$, 
we have that $|M(1)|\leq |M(0)|+| M'(0)|+\max_{s\in [0,1]}\frac{M''(s)}{2}$ 
by Taylor's theorem. Furthermore, $|M(0)|=|f(x_*)|=\|f\|_\infty\uR$,
$|M(1)|=|f(x)|$ and $M'(0)=0$. The latter is due to the fact that $x_*$ 
is an extremal point of $f$ and so of $M$. Now,
\[M''(t)=\diffa_{x_t}^2f(\dot x_t,\dot x_t)-\D f(x_t)\dist_\bbS(x_\ast,x)^2,\]
since $\ddot x_t=-\dist_{\bbS}(x_\ast,x)^2x_t$, as $x_t$ is a geodesic on $\bbS^n$ of constant speed $\dist_\bbS(x_\ast,x)$, and 
$\diffa_{x_t}f(x_t)=\D f(x_t)$ by Euler's formula~\eqref{eq:Euler}. 
Then, by  Corollary~\ref{cor:higherderivativeestimateF},
\[
\max_{s\in [0,1]}\frac{|M''(s)|}{2}\leq
\binom{\D}{2}\|f\|_\infty\uR+\frac{\D}{2}\|f\|_\infty\uR
=\frac{\D^2}{2}\|f\|_\infty\uR.
\]
Thus $\|f\|_\infty\uR\leq |f(x)|+\frac{\D^2}{2}\|f\|_\infty\uR\delta^2$, 
and the desired inequality follows.
\end{proof}

\begin{remark}
Proposition~\ref{prop:estimationnorm} is a slight improvement
of~\cite[Lemma~2.5]{EPR18}.
\end{remark}

Proposition~\ref{prop:estimationnorm} suggests the following algorithm.

\begin{minipage}[t]{0.9\textwidth}
\begin{algorithm*}[H]
\DontPrintSemicolon
\SetKwInOut{input}{Input}
\SetKwInOut{output}{Output}
\caption{\textsc{NormApprox$\bbR$}}\label{alg:normapprox}
\input{$f\in\Hd\uR[q]$, $k\in \bbN$}
\hrulefill\\
$\mcG\leftarrow \textsc{Grid}\left(n,\lceil (k-1)/2+\log\D\rceil\right)$\;
$t\leftarrow \left(1-2^{-k}\right)^{-1}\max\{\|f(x)\|_\infty\mid x\in\mcG\}$\;
\hrulefill\\
\output{$t\in [0,\infty)$}
\postcondition{$\left(1-2^{-k}\right)t\leq \|f\|_\infty\uR\leq t$}
\end{algorithm*}
\end{minipage}
\medskip

\begin{prop}\label{prop:algorithmestimationnorm}
Algorithm~\nameref{alg:normapprox} is correct. On input
$(f,k)\in\Hd\uR[q]\times \bbN$, its cost is bounded by
\[
 \Oh\left(2^{n\log n}\D^n2^{\frac{(k+1)n}{2}}N\right).
\]
\end{prop}

\begin{proof}
This is a direct consequence of Propositions~\ref{prop:gridconstructor}
and~\ref{prop:estimationnorm} and the fact that $f$ can be evaluated at $x\in\bbS^n$ with $\Oh(N)$ arithmetic operations (see~\cite[Lemma~16.31]{Condition}).
\end{proof}

\begin{remark}
The ideas here can also be applied to compute $\|f\|_\infty\uC$. 
\end{remark}

\subsubsection{Estimation of \texorpdfstring{$\cK$}{K}}

In many grid-based algorithms, the estimation of condition numbers is 
done implicitly 
along the way; this does not affect the overall computational cost and 
it makes for an easier 
understanding of these algorithms. The next proposition is the core of the 
estimation of $\cK$. Note that the mesh of the grid needed 
to estimate $\cK$ depends on $\cK$ itself.

\begin{prop}\label{prop:estimationcK}
Let $f\in\Hd\uR[q]$ and $\mcG\subset\bbS^n$ be a $\delta$-net. 
If 
\[\delta\,\D\max_{x\in\mcG}\cK(f,x)<1,\]
then
\[
\max_{x\in\mcG}\cK(f,x)\leq \cK(f)\leq \frac{1}{1-\delta\,\D\,\max_{x\in\mcG}\cK(f,x)} \max_{x\in\mcG}\cK(f,x).
\]
\end{prop}

\begin{proof}
We only have to prove the right-hand side inequality, since the other 
one is obvious. Let $x_\ast\in\bbS^n$ such that $\cK(f)=\cK(f,x_\ast)$ 
and $x\in\mcG$ such that $\dist_\bbS(f,x)\leq\delta$. Then, by the 2nd Lipschitz property (Theorem~\ref{theo:propertiesrealcond}), we have 
\[
\frac{1}{\cK(f,x)}-\frac{1}{\cK(f,x_\ast)}\leq \D\,\dist_\bbS(x_\ast,x)\leq \D\,\delta.
\]
Hence $1/\cK(f,x_\ast)\leq (1-\delta\,\D\,\cK(f,x))/\cK(f,x)$ and the 
desired inequality follows from the hypothesis.
\end{proof}

Proposition~\ref{prop:estimationcK} suggests the following algorithm 
which involves only one $L_\infty$-norm computation.

\begin{minipage}[t]{0.9\textwidth}
\begin{algorithm*}[H]
\DontPrintSemicolon
\SetKwInOut{input}{Input}
\SetKwInOut{output}{Output}
\caption{\textsc{$\cK$-Estimate}}\label{alg:cKaestimation}
\input{$f\in\Hd[q]$, $k\in\bbN$, $b\in\bbN\cup\{\infty\}$}
\hrulefill

$t\leftarrow \nameref{alg:normapprox}(f,k+1)$\;
$\ell\leftarrow 0$\;
\Repeat{$\D\, K 2^{-\ell} \le 2^{-(k+1)}$ or $2^b\leq K$}{
$\ell\leftarrow \ell+1$\;
$K\leftarrow \max\{\sqrt{q}t/\max\{\|f(x)\|,\|\diff_x f^\dagger \Delta\|^{-1}\}\,|\,x\in \textsc{Grid}(n,\ell)\}$ \;
}
\eIf{
$2^b\le K$
}{
\KwRet{\tt fail}
}{
$\mathcal{K}\leftarrow (1-2^{-k})^{-1}K$\;
\KwRet{$\mathcal{K}$}}

\hrulefill

\output{{\tt fail} or $ \mathcal{K} \in (0,\infty)$}
\postcondition{$2^b\leq \cK(f)$, if {\tt fail};\\$(1-2^{-k}) \cK(f) \leq \mathcal{K} \leq \cK(f)$, otherwise}
\end{algorithm*}
\end{minipage}
\medskip

\begin{prop}\label{prop:complexityestimationcK}
Algorithm~\nameref{alg:cKaestimation} is correct. On input
$(f,k,b)\in\Hd\uR[q]\times \bbN\times (\bbN\cup\{\infty\})$, its 
cost is bounded by
\[
2^{\Oh(n(k+\log n))}D^n N\min\{\cK(f)^n,2^{nb}\}.
\]
\end{prop}

\begin{proof}
The correctness follows from Propositions~\ref{prop:algorithmestimationnorm} 
and~\ref{prop:estimationcK}, and $(1-2^{-(k+1)})^2>1-2^{-k}$.

The cost of the first line of the algorithm is bounded by
Proposition~\ref{prop:algorithmestimationnorm}. 
The number of evaluations of 
$$
  \sqrt{q}t/\max\{\|f(x)\|,\|\diff_x f^\dagger \Delta\|^{-1}\}
$$ 
in the $\ell$th iteration of the loop is given by
Proposition~\ref{prop:gridconstructor}. We 
need $\Oh(N+n^3)$ operations for each such evaluation, by~\cite[Proposition~16.32]{Condition}. 

In this way, if the loop runs $\ell_0$ iterations, it performs a 
total of 
\[
\Oh(2^{n\log n}(D^n2^{\frac{(k+2)n}{2}}N+2^{n(\ell_0+1)}(N+n^3)))
\]
operations. 

If the algorithm outputs $\mcK$, then $\ell_0=\lceil k+\log \D+
\log \mathcal{K}-\log (1-2^{-k})\rceil$. Moreover, from the correctness, 
$\log \mathcal{K}-\log (1-2^{-k})\leq \log\cK(f)$, and so 
$\ell_0\leq k+1+\log \D+\log\cK(f)$.

If the algorithm outputs {\tt fail}, then the first criterion had to 
fail and so, as long as the second criterion fails too, we have
\[
\ell<k+\log \D+b.
\]
And so, in this case, $\ell_0\leq k+1+\log \D+\log b$.

We conclude from the bounds above and some straightforward computations.
\end{proof}

By setting $k$ to $7$ and $b=\infty$ we have the following important
corollary.

\begin{cor}\label{cor:cKestimation}
There is an algorithm, \textsc{$\cK$-Estimate$^*$}, that on input
$(f)\in\Hd\uR[q]$ computes $\mcK\in [1,\infty)$ such that
\[0.99\mcK\leq\cK(f)\leq\mcK.\]
This algorithm halts if and only if $\cK(f)<\infty$ and its cost 
is bounded by
\begin{equation}\tag*{\qed}
2^{\Oh(n\log n)}D^n N\cK(f)^n.
\end{equation}
\end{cor}

\subsubsection{Complexity analysis of grid-based algorithms using \texorpdfstring{$\cK$}{K}}

To get the grid-method to work, we need two ingredients: a method for
selecting the points in the grid near the geometric object of interest 
and a way of controlling distances between these two sets. 

\begin{theo}[Construction-Selection]\label{theo:gridselection}
Let $f\in\Hd\uR[q]$ and $\mcG\subseteq \bbS^n$ be a $\delta$-net. If
\[
4\D^2\cK(f)^2\delta<1,
\]
and $Q\in\bbR$ is such that $0.99Q\leq \|f\|_\infty\uR\leq Q$, then 
\[
\dist_H\left(\left\{x\in \mcG\mid \frac{\|f(x)\|}{\sqrt{q}Q}
< \D\,\delta\right\},\mcZ_{\bbS}(f)\right)< 2\D\cK(f)\delta,
\]
where $\dist_H(A,B):=\max\{\sup\{\dist(a,B)\mid a\in A\},
\sup\{\dist(b,A)\mid b\in B\}\}$ is the Hausdorff distance.
\end{theo}

Following~\cite{federer1959}, recall that the \emph{medial axis} 
$\Delta_X$ of a closed set $X\subset\bbR^n$ is the set
\[
\Delta_X:=\{p\in\bbR^n\mid \#\{x\in X\mid \dist(p,x)=\dist(p,X)\}\geq 2\},
\]
consisting of those points for which there is more than one nearest 
point in $X$, and that the \emph{reach} $\tau(X)$ of $X$ is the quantity 
\[
\tau(X):=\dist(X,\Delta_X),
\]
measuring the size of the neighborhood of $X$ within which the nearest 
point projection is well-defined. If $X$ is finite, then $\Delta_X$ 
is the union of the boundaries of the cells of the Voronoi diagram of $X$, 
and $\tau(X)$ is half the minimum distance between two distinct 
points of $X$. Thus, when $\mcZ_\bbS(f)$ is zero-dimensional, $2\tau(\mcZ_\bbS(f))$ is 
the separation of the zeros of $f$ in the sphere. 

\begin{theo}\label{theo:gridseparation}
Let $f\in\Hd\uR[q]$. Then
\[\tau(\mcZ_\bbS(f))\geq \frac{1}{7\D\cK(f)}.\]
\end{theo}

\begin{proof}[Proof of Theorem~\ref{theo:gridselection}]
Let $x_0\in\mcZ_{\bbS}(f)$, then there is some $x_1\in\mcG$ such that $\dist_\bbS(x_0,x_1)\leq \delta$. Let $[0,1]\ni t\mapsto x_t$ be the geodesic joining them. By Taylor's theorem,
\[\|f(x_1)\|\leq \|f(x_0)\|+ \delta \sup_{s\in [0,1]}\|\diff_xf\|,\]
and so, by Kellogg's theorem (Corollary~\ref{cor:kelloggF}) and 
$f(x_0)=0$, we have that 
$\frac{\|f(x_1)\|}{\sqrt{q}\|f\|_\infty\uR}\leq \D\,\delta$. Hence
$\frac{\|f(x_1)\|}{\sqrt{q}Q}\leq \D\,\delta$ and
\[
\dist\left(x_0,\left\{x\in \mcG\mid \frac{\|f(x)\|}{\sqrt{q}Q}\leq \D\,\delta\right\}\right)\leq 
\dist(x_0,x_1)\le \dist_{\bbS}(x_0,x_1)\leq\delta.
\]

Let now $x_2\in\mcG$ be such that $\frac{\|f(x_2)\|}{\sqrt{q}Q}< \D\,\delta$.
Then 
\begin{equation}\label{eq:alpha0}
\frac{\|f(x_2)\|}{\sqrt{q}\|f\|_\infty\uR} < 1.02\D\,\delta
\le \frac{1}{4\D\cK(f)^2} <\frac{1}{\cK(f,x_2)}
\end{equation}
the second inequality by our hypothesis. Because of the Regularity 
Inequality (Theorem~\ref{theo:propertiesrealcond}) we must then 
have $\sqrt{q}\|f\|^{\bbR}_{\infty}\|\diff_{x_2}f^{\dagger}\Delta^{1/2}\|
\le \cK(f,x_2)$. It follows that 
\begin{eqnarray*}
\|\diff_{x_2}f^{\dagger}f(x_2)\|\gamma(f,x_2) &<& 
 \frac{\cK(f,x_2)}{\sqrt{q}\|f\|^{\bbR}_{\infty}}\gamma(f,x_2) 
 \le \frac{1}{2}\D\cK(f)^2\frac{\|f(x_2)\|}{\sqrt{q}\|f\|_\infty\uR}\\
&<& \frac{1.02}{2}\D^2\cK(f)^2\delta< 
\frac{1.02}{8} < 0.13071\ldots
\end{eqnarray*}
where we used the Higher Derivative Estimate
(Theorem~\ref{theo:propertiesrealcond}) in the first line,  
and~\eqref{eq:alpha0} and the hypothesis in the second. 
This means that Smale's $\alpha$-criterion holds for $x_2$ and $f_{|\Tg_{x_0}\bbS^n}$ by~\cite[Th\'eor\`eme~128]{dedieubook}. 
Hence there is $x_3\in\Tg_{x_2}\bbS^n$ such that $f(x_3)=0$ and
\[
\dist(x_2,x_3)\leq 1.64 \|\diff_{x_2}f^{\dagger}f(x_2)\|
\leq 1.64\cdot 1.02\,\D\cK(f)\delta < 2\D\cK(f)\delta.
\]
Since $\dist(x_2,x_3/\|x_3\|)
=\arctan\dist(x_2,x_3)\leq \dist(x_2,x_3)$, we are done.
\end{proof}

\begin{remark}
The proof also shows the convergence of Newton's method 
associated with $f_{|\Tg_x\bbS^n}$ for every $x\in \mcG$ such that
$\frac{\|f(x)\|}{\sqrt{q}\|f\|_\infty\uR}\leq \D\,\delta$. 
Hence, we can refine our approximations if needed.
\end{remark}

\begin{proof}[Sketch of proof of Theorem~\ref{theo:gridseparation}]
The proof is very similar to the one of~\cite[Theorem~4.12]{BCL17}. By~\cite[Lemma~2.7]{BCL17} and~\cite[Theorem~3.3]{BCL17}, we have that
\[
\tau(\mcZ_\bbS(f))\geq \min\left\{1,\frac{1}{14\max\{\gamma(f,x)\mid x\in \mcZ_\bbS(f)\}}\right\}.
\]
Hence, by the Higher Derivative Estimate (Theorem~\ref{theo:propertiesrealcond}), the desired bound follows.
\end{proof}

The following theorem is a variant of the so-called Niyogi-Smale-Weinberger
theorem~\cite[Propoposition~7.1]{niyogismaleweinberger2008}.

\begin{theo}\label{thm:nsw}
Let $f\in\Hd\uR[q]$, $\mcG\subset\bbS^n$ be a  
$\delta$-net, and $Q\in\bbR$ be such that 
$0.99Q\leq \|f\|_\infty\uR\leq Q$. 
If $90\D^2\cK(f)^2\delta<1$, then for every 
\[
 \varepsilon\in \left(6\D\cK(f)\delta,\frac{1}{14\D\cK(f)}\right),
\]
the sets $\mcZ_\bbS(f)$ and
\[
\bigcup\left\{B(x,\varepsilon)\mid x\in\mcG,\,\frac{\|f(x)\|}{\sqrt{q}Q}<\D\delta\right\}
\]
are homotopically equivalent. In particular, they have the same Betti numbers.
\end{theo}

\begin{proof}
This is just~\cite[Theorem~2.8]{BCL17} combined with Theorems~\ref{theo:gridselection} and~\ref{theo:gridseparation}.
\end{proof}

We can now describe the algorithm. We will call a black box 
\textsc{Betti} for computing the Betti numbers of a union of balls. 
This is a standard procedure in topological data analysis~\cite{edelsbrunner1995}.

\begin{minipage}[t]{0.9\textwidth}
\begin{algorithm*}[H]
\DontPrintSemicolon
\SetKwInOut{input}{Input}
\SetKwInOut{output}{Output}
\caption{\textsc{PolyBetti}$_\infty$}\label{alg:polybetti}
\input{$f\in\Hd[q]$}
\precondition{$q\leq n$, $f$ has no singular zeros (i.e. $\cK(f)<\infty$)}
\hrulefill

$Q\leftarrow \nameref{alg:normapprox}(f,7)$\;
$\mcK\leftarrow$ \textsc{$\cK$-Estimate$^*$}$(f)$\;
$\ell\leftarrow 7+\lceil 2\log\D+2\log k\rceil$\;
$\mcG\leftarrow \textsc{Grid}(n,\ell)$\;
$\mcX\leftarrow \{x\in\mcG\mid \|f(x)\|<\sqrt{q}\,\D Q2^{-\ell}\}$\;
$\varepsilon\leftarrow 3/(50\D\mcK)$\;
$(\beta_0,\ldots,\beta_n)\leftarrow \textsc{Betti}(\mcX,\varepsilon)$\;
\KwRet{$\beta_0,\ldots,\beta_n$}

\hrulefill

\output{$\beta_0,\ldots,\beta_n\in\bbN$}
\postcondition{$\beta_0,\ldots,\beta_n$ are the Betti numbers of $\mcZ_\bbS(f)$}
\end{algorithm*}
\end{minipage}
\medskip

\begin{prop}\label{prop:runtimebetti}
Algorithm~\nameref{alg:polybetti} is correct and its cost is bounded by
\[
2^{\Oh(n^2\log n)}D^{10n^2}\cK(f)^{10n^2}.
\]
\end{prop}

\begin{proof}
Correctness is a consequence of Theorem~\ref{thm:nsw} 
and the fact that the computed $Q$ satisfies 
$0.99Q\leq \|f\|_\infty\uR\leq Q$ by
Proposition~\ref{prop:algorithmestimationnorm}.

For the complexity, we apply 
Proposition~\ref{prop:estimationnorm} for the 
first line, Corollary~\ref{cor:cKestimation} for the second 
line, and
Proposition~\ref{prop:gridconstructor} for the fourth and fifth line. 
We know that $\textsc{Betti}$ has cost 
$\Oh\left(2^{\Oh(n\log n)}|\mcX|^{5n}\right)$
(see~\cite[\S5]{CKS16} for example) and that 
$|\mcX|=\Oh(2^{n\log n}\D^{2n}\cK(f)^{2n})$, by
Proposition~\ref{prop:gridconstructor}. Note that we have eliminated
$N$ from the bounds. We have done so using the fact that, as $q\le n$ 
(by the precondition of the input), 
$N\leq 2^{n\log n}\D^{n}$.

We note that our bound uses $\mcK\leq 1.02\cK(f)$ 
in order to get the cost dependent on $\cK(f)$ instead of 
on the computed estimate $\mcK$.
\end{proof}

The complexity estimate in Proposition~\ref{prop:runtimebetti} does 
not differ much from those in other grid-based algorithms. We will 
see in~\S\ref{subsec:ckvskappa}, however, that the occurrence of 
$\cK$ in the place of $\kappa$ leads to substantial improvements 
when one goes beyond the worst-case framework and considers random 
input models. 

\subsection{Complexity of the Plantinga-Vegter algorithm}\label{subsec:PV}

The ideas above can also be applied to the Plantinga-Vegter 
algorithm~\cite{PV}. In a recent work~\cite{CETC-PVjorunal}
(cf.~\cite{CETC-PV}) we performed an extensive analysis of 
this algorithm including details for finite precision arithmetic. 
So, we will be brief here, referring the reader 
to~\cite{CETC-PVjorunal} for details, 
and will only focus on the (exact) interval version of the algorithm.

\subsubsection{The Plantinga-Vegter Subdivision Algorithm}

Let $\Pd$ be the space of polynomials in $X_1,\ldots,X_n$ of degree at 
most $d$. The Plantinga-Vegter 
algorithm~\cite{PV}\footnote{The original algorithm~\cite{PV} only 
dealt with dimensions two and three. For the 
extension to dimensions four or higher see~\cite{galehousethesis}.} is a 
subdivision-based algorithm 
for obtaining a piecewise linear approximation 
of the zero set of 
$f\in\Pd$ inside $[-a,a]^n$. As customary, we will focus on the 
complexity analysis of the subdivision routine only. The idea is to 
iteratively subdivide some boxes ---i.e., sets of the form $B=m(B)+[-w(B)/2,w(B)/2]^n$ (here $m(B)\in\bbR^n$ is the center of 
$B$ and $w(B)>0$ is its width)--- in $[-a,a]^n$ until every box $B$ 
in the subdivision satisfies the 
following condition:
\[
C_f(B)\,:\,\text{either }0\notin f(B)\text{ or }0\notin 
\langle \nabla f(B),\nabla f(B)\rangle
\]
where $\langle~,~\rangle$ is the standard inner product 
and $\nabla f$ 
is the gradient vector of $f$. 
Once this criterion is satisfied by all 
boxes in the subdivision the Plantinga-Vegter 
algorithm returns a topologically accurate approximation of 
the zero set of $f$ in the region $[a,-a]^n$ and halts 
(see~\cite{PV} ($n\leq 3$) and~\cite{galehousethesis} (arbitrary $n$) for details on how this is done).

For $f\in\Pd$, we define
\[
\|f\|_\infty:=\max\{|f\hm(x)|\mid x\in\bbS^n\}=\|f\hm\|_\infty\uR
\]
where $f\hm\in\mcH_d[1]$ is the homogenization of $f$. Taking the maps 
(2.3), (2.4), (2.5) in~\cite{CETC-PVjorunal} and substituting on them 
the Weyl norm by the real $L_\infty$-norm we get
\begin{equation}\label{eq:definitonhhprime}
    h(x)=\frac{1}{\|f\|_\infty(1+\|x\|^2)^{(d-1)/2}}
\quad\text{ and }\quad h'(x)=\frac{1}{d\|f\|_\infty(1+\|x\|^2)^{d/2-1}}
\end{equation}
together with
\begin{equation}\label{eq:fhat}
    \widehat{f}:x\mapsto h(x)f(x)=\frac{f(x)}{\|f\|_\infty(1+\|x\|^2)^{(d-1)/2}}
\end{equation}
and
\begin{equation}\label{eq:derfhat}
\widehat{\nabla f}:x\mapsto h'(x)\diff f(x)
=\frac{\nabla f(x)}{d\|f\|_\infty(1+\|x\|^2)^{d/2-1}}.
\end{equation}

One can use these maps to produce interval 
approximations as we do
in~\cite{CETC-PVjorunal}. For $X\subseteq \bbR^m$, we denote by $\square X$ 
the set of boxes contained in $X$. Recall that an \emph{interval approximation} 
of $f:\bbR^n\rightarrow \bbR^q$ is a function 
$\square f:\square \bbR^n\rightarrow \square\bbR^q$ that maps boxes in 
$\bbR^n$ to boxes in $\bbR^q$ in such a way that 
$f(B)\subseteq \square f(B)$.

\begin{prop}\label{prop:hats}
Let $f\in \Pd$. Then
\[\square[hf]:B\mapsto \widehat{f}(m(B))+(1+d)\sqrt{n}\left[-\frac{w(B)}{2},\frac{w(B)}{2}\right]
\]
is an interval approximation of $hf$, and
\[\square[\|h'\diff f\|]:B\mapsto \|\widehat{\nabla f}(m(B))\|+d\sqrt{n}\left[-\frac{w(B)}{2},\frac{w(B)}{2}\right]\]
is an interval approximation of $\|h'\diff f\|$.
\end{prop}

\begin{proof}[Sketch of proof]
Using the bounds from Kellogg's theorem (Theorem~\ref{theo:kellogg}) and 
its corollaries, we can easily deduce (as it is done in the proof of
Theorem~\ref{theo:propertiesrealcond}) that the maps
\[
g/\|g\|_\infty\uR:\bbS^n\rightarrow [-1,1]~\text{ and }
\diffa g(v)/(d\|g\|_\infty\uR\|v\|):\bbS^n\rightarrow [-1,1]
\]
are $d$- and $(d-1)$-Lipschitz (with respect to the geodesic distance) 
for $g\in\mcH_d\uR[1]$.

We now argue as in \cite[\S4]{CETC-PVjorunal}, but using 
these Lipschitz
properties, to prove that $\hat f$ and 
$\widehat{\nabla f}$ are 
$(1+d)$- and $d$-Lipschitz, respectively. 
For the latter, we use the fact that for $v\in\bbR^n$,
$\diffa_Xf\hm\begin{pmatrix}0\\v\end{pmatrix}
=(\langle \nabla f,v\rangle)\hm$
and that $\|\widehat{\nabla f}\|$ is $d$-Lipschitz if 
$\langle \widehat{\nabla f},v\rangle$ is so for 
every $v\in\bbS^{n-1}$.
\end{proof}

Using the interval approximations and their Lipschitz properties 
in Proposition~\ref{prop:hats} we can rewrite the condition $C_f(B)$. 
We only need to use~\cite[Lemma~4.2]{CETC-PVjorunal} for the second clause 
of the condition.

\begin{theo}\label{theo:condprime}
Let $B\in\square \bbR^n$. If the condition
\[
C_f^{\square}(B)\,:\,
\left|\widehat{f}(m(B))\right|>2d\sqrt{n}w(B) 
\text{ or }~\left\|\widehat{\nabla f}
(m(B))\right\|>2\sqrt{2}d\sqrt{n}w(B).
\]
is satisfied, then $C_f(B)$ is true.\eproof
\end{theo}

The subdivision procedure of the Plantinga-Vegter algorithm thus takes 
the following form (where \textsc{StandardSubdivision} is a procedure  
that given a box divides it into $2^n$ equal boxes. Recall that 
$\square [-a,a]^n$ is the set of boxes within $[-a,a]^n$. 

\begin{minipage}[t]{0.9\textwidth}
\begin{algorithm*}[H]
\DontPrintSemicolon
\SetKwInOut{input}{Input}
\SetKwInOut{output}{Output}
\caption{\textsc{PV-In\-ter\-val$_\infty$}}\label{alg:PVAlgorithmconcrete}
\input{$f\in\Pd$\\
$a \in (0,\infty)$
}
\precondition{$\mcZ(f)$ is smooth inside $[-a,a]^n$}
\hrulefill

$Q\leftarrow \nameref{alg:normapprox}(f,7)$\;
$\tilde{\mcS}\leftarrow \{[-a,a]^n\}$\;
$\mcS\leftarrow\varnothing$\;
\Repeat{$\tilde{\mcS}=\varnothing$}{
Take $B$ in $\tilde{\mcS}$\;
$\tilde{\mcS}\leftarrow \tilde{\mcS}\setminus\{B\}$\;
\uIf{$\left|f(m(B))\right|>2d\sqrt{n}w(B)Q(1+\|m(B)\|^2)^{\frac{d-1}{2}}$}{
$\mcS\leftarrow \mcS\cup\{B\}$\;}
\ElseIf{$\left\|\nabla f(m(B))\right\|>2\sqrt{2}d\sqrt{n}w(B)Q(1+\|m(B)\|^2)^{\frac{d}{2}-1}$}{
$\mcS\leftarrow \mcS\cup\{B\}$\;}
\Else{
$\tilde{\mcS}\leftarrow \tilde{\mcS}\cup \textsc{StandardSubdivision}(B)$\;
}}
\KwRet{$\mcS$}

\hrulefill

\output{Subdivision $\mcS\subseteq \square[-a,a]^n$ of $[-a,a]^n$}
\postcondition{For all $B\in\mcS$, $C_f(B)$ is true}
\end{algorithm*}
\end{minipage}

\subsubsection{Complexity of \texorpdfstring{{\sc PV-Interval}$_\infty$}{PV-Interval}}

Without much effort,~\cite[Proposition~5.1]{CETC-PVjorunal} transforms into the
following proposition. The essential step is to multiply the inequalities in 
that proposition by $\|f\hm\|_W/\|f\|_\infty$.

\begin{prop}\label{prop:fundamentalproposition_aff}
Let $f\in\Pd$ and $x\in\bbR^n$. Then either
\[\left|\widehat{f}(x)\right|> \frac{1}{2\sqrt{2d}\,\cK(f\hm,\yuproj(x))}  
\quad\text{ or }\quad
\left\|\widehat{\nabla f}(x)\right\|> \frac{1}{2\sqrt{2d}\,\cK(f\hm,\yuproj(x))},\]
where $\yuproj(x)=\frac{1}{\sqrt{1+\|x\|^2}}\begin{pmatrix}1\\x\end{pmatrix}
\in\bbS^n$.\eproof
\end{prop}

With Proposition~\ref{prop:fundamentalproposition_aff} and 
the Lipschitz properties shown for $\hat f$ and
$\widehat{\nabla f}$, one can produce 
a \emph{local size bound} for $C^{\Box}_f(B)$. This is a function that
evaluated at a point $x$ gives a lower bound on the volume of any
possible box containing $x$ and not satisfying the predicate $C'_f(B)$.

\begin{theo}\label{thm:PVMAIN1}
The map
\[
  x\mapsto 1/\left(2^{3/2}d^{\frac{3}{2}}\sqrt{n}\cK(f\hm,\yuproj(x))\right)^n
\]
is a local size bound for $C_f^{\Box}$ (of Theorem~\ref{theo:condprime}).\eproof
\end{theo}

Then using the continuous amortization of \cite{burr2009,burr2016,burr2017}
(see~\cite[Theorem~6.1]{CETC-PVjorunal}), we conclude the following, 
which takes into account the cost of 
calling~\nameref{alg:normapprox} (Proposition~\ref{prop:estimationnorm}).

\begin{theo}\label{thm:PVMAIN2}
The number of boxes in the final subdivision $\mcS$
of~\nameref{alg:PVAlgorithmconcrete} on input $(f,a)$ 
is at most
\[
d^{\frac{3}{2}n}\max\{1,a^n\}2^{\frac{1}{2}n\log{n}+11n}\,
\bbE_{\fkx\in[-a,a]^n}\left(\cK(f\hm,\yuproj(\fkx))^n\right).
\]
The number of arithmetic operations performed by~\nameref{alg:PVAlgorithmconcrete} on input $(f,a)$ is at most
\begin{equation}\tag*{\qed}
  \Oh\left(d^{\frac{3}{2}n+1}\max\{1,a^n\}2^{\frac{1}{2}n\log{n}+11n}N\,
  \bbE_{\fkx\in[-a,a]^n}
  \left(\cK(f\hm,\yuproj(\fkx))^n\right)\right).
\end{equation}
\end{theo}

The condition-based estimates in Theorem~\ref{thm:PVMAIN2} 
are very similar to those of~\cite[Theorem~6.3]{CETC-PVjorunal}. 
It is important to observe that only one norm computation 
is performed by~\nameref{alg:PVAlgorithmconcrete} (in its very first 
step) and that the cost of this computation is already included 
in the cost bound in Theorem~\ref{thm:PVMAIN2}.
We will see in~\S\ref{sec:PV-average} that the occurrence 
of $\cK$ in the place of $\kappa$ results in significant improvements in 
overall complexity when we consider average or smoothed analysis.

\subsection{Probabilistic Analysis of Algorithms}\label{subsec:ckvskappa}

In the preceding sections, we have shown that existing grid-based 
and subdivision-based algorithms that use (in their design 
and/or in their analysis) $\kappa$ can be modified 
to use $\cK$ instead. Moreover, we have shown that the
condition-based complexity estimates in terms of $\cK$ are similar
to those in terms of $\kappa$. In this section we will show that
when we consider random inputs, in contrast, the cost (expected or in 
probability) substantially decreases. 

We first introduce the randomness model along with some useful 
probabilistic results. Then we prove a general comparison result 
which shows that when 
substituting $\kappa$ by $\cK$ one can expect to reduce the size 
of the condition number by a factor of $\sqrt{N}$. Finally, we
apply these estimates to both {\sc PolyBetti} and the  
Plantinga-Vegter algorithm and highlight the complexity 
improvements.

For most algorithms in real algebraic geometry, 
condition-based estimates show a dependence on either $\kappa^n$ or
on $\cK^n$. When this occurs the complexity estimates 
improve by a factor of the form $N^{\frac{n}{2}}$ when we pass from
$\kappa$ to $\cK$. The final complexity estimates thus  
change from having an exponent 
quadratic in $n$ to an exponent quasilinear in $n$.

\subsubsection{The Randomness Model: Dobro Random Polynomials}

Given a random variable $\fkx\in\bbR$ we say that:
\begin{enumerate}[(i)]
\item $\fkx$ is \emph{centered} if $\bbE\fkx=0$.
\item $\fkx$ is \emph{subgaussian} if there is a 
constant $K>0$ such that for all $p\geq 1$,
\[
  \left(\bbE|\fkx|^p\right)^{\frac{1}{p}}\leq K\sqrt{p}.
\]
The smallest $K$ satisfying this condition is called the 
\emph{$\psi_2$-norm} of $\fkx$, and is denoted $\|\fkx\|_{\psi_2}$. 
\item $\fkx$ has the \emph{anti-concentration property with constant
$\rho$} if for all $u\in \bbR$ and $\varepsilon>0$,
\[\bbP(|\fkx-u|<\varepsilon)\leq 2\rho\varepsilon.\]
Note that this is equivalent to $\fkx$ having a density (with 
respect to the Lebesgue measure) bounded by $\rho$.
\end{enumerate}

We now extend to tuples the class of real random polynomials
introduced in~\cite{CETC-PV}. 

\begin{defi}
A \emph{dobro random polynomial tuple} $\fkf\in\Hd\uR[q]$ 
with parameters $K$ and $\rho$ is a tuple of random polynomials
\[
\left(\sum_{|\alpha|=d_1}\binom{d_1}{\alpha}^{\frac{1}{2}}
\fkc_{1,\alpha}X^\alpha,\ldots,\sum_{|\alpha|=d_q}
\binom{d_q}{\alpha}^{\frac{1}{2}}\fkc_{q,\alpha}X^\alpha\right)
\]
such that the $\fkc_{i,\alpha}$ are independent centered subgaussian 
random variables with $\psi_2$-norm at most $K$ and 
anti-concentration property with constant $\rho$. 
\end{defi}
\begin{remark}
Probabilistic estimates for a dobro polynomial $\fkf$ will depend on
$K\rho$. This product is invariant under scalar
multiplication of $\fkf$ since $\lambda\fkf$ is dobro 
with parameters $|\lambda|K$ and $\rho/|\lambda|$. Moreover,
note that\footnote{This follows from $2tK\rho\geq \bbP_\fkx(|\fkx|\leq Kt)\geq 1-\bbP_\fkx(|\fkx|>Kt)\geq 1-2\enumber^{-t^2/2}$ and optimizing, 
where $\fkx$ is subgaussian with $\psi_2$-norm $K$ and the
anti-concentration property with constant $\rho$.} $6K\rho\geq 1$. 
\end{remark}

\begin{exam}\label{ex:KSS}
A dobro random polynomial tuple $\fkf\in\Hd\uR[q]$ such that the $\fkc_\alpha$ are are i.d.d. normal random variables of mean zero and variance one is called a \emph{KSS (real) polynomial tuple}\footnote{In this definition, KSS refers to
Kostlan-Shub-Smale. An alternative term is ``Shub-Smale random polynomial tuple", following~\cite{azaiswschbor2005}, but we use ``KSS" instead, as this is consistent with the use we have made of the term in the case of a single polynomial.}. In this case, we can take $K\rho=2/\sqrt{\pi}$.
\end{exam}
\begin{exam}
A dobro random polynomial tuple $\fkf\in\Hd\uR[q]$ such that the $\fkc_\alpha$ are are i.d.d. uniform random variables in $[-1,1]$ is a \emph{Weyl uniform (real) polynomial tuple}. In this case, we can take $K\rho=1/2$.
\end{exam}

We now state and prove several probabilistic results that will be used later on.

\begin{prop}[Subgaussian tail bounds]\label{prop:subgaussiantailbound}
Let $\fkx\in\bbR$ be a random variable.
\begin{enumerate}
    \item If $\fkx$ is subgaussian with $\psi_2$-norm at most $K$, then 
    for all $t>0$, $\bbP(|\fkx|\geq t)\leq \enumber^{1-\frac{t^2}{6K^2}}$.
    \item If there are $C\geq\enumber$ and $K>0$ such that for all $t>0$,
    $\bbP(|\fkx|\geq t)\leq C\enumber^{-\frac{t^2}{K^2}}$,
    then $\fkx$ is subgaussian with $\psi_2$-norm at most 
    $K\left(\sqrt{\pi/2}+\sqrt{2\ln C}\right)$.
\end{enumerate}
\end{prop}

\begin{prop}[Hoeffding inequality]\label{prop:subgaussianprojection}
Let $\fkx\in\bbR^N$ be a random vector such that its components $\fkx_i$ 
are centered subgaussian random variables with $\psi_2$-norm at most 
$K$ and $a\in\bbS^{N-1}$. Then for all $t\geq 0$,
\[
\bbP_\fkx\left(|a^*\fkx|\geq t\right)\leq 2\enumber^{-\frac{t^2}{11 K^2}}.
\]
In particular, $a^*\fkx$ is a subgaussian random 
variable with $\psi_2$-norm at most $5K$.
\end{prop}

\begin{prop}[Anti-concentration bound]\label{prop:anticoncentrationprojection}
Let $\fkx\in\bbR^N$ be a random vector such that its components 
$\fkx_i$ are independent random variables with anti-concentration property 
with constant $\rho$. Then, for every 
$A\in\bbR^{k\times N}$ with rank $k$ and measurable $U\subseteq \bbR^k$,
\[\bbP_\fkx\left(A\fkx\in U\right)\leq \frac{\vol(U)(\sqrt{2}\rho)^k}{\sqrt{\det\left(AA^*\right)}}.\]
\end{prop}

\begin{proof}[Proof of Proposition~\ref{prop:subgaussiantailbound}]
This is just \cite[Proposition 2.5.2]{vershyninbook} with improved constants.

For the first part, we give a proof since we don't use explicitly the constants in the proof of \cite[Proposition 2.5.2]{vershyninbook}. Fix $\lambda>0$. Then, by Markov's inequality and expanding the exponential as a power series, 
\[
\bbP(|\fkx|\geq t) = \bbP\left(\enumber^{\lambda^2\fkx^2}\geq \enumber^{\lambda^2t^2}\right) \leq
\enumber^{-\lambda^2t^2}\sum_{p=0}^\infty \frac{\lambda^{2p}\bbE \fkx^{\,2p}}{p!} \leq \enumber^{-\lambda^2t^2}\sum_{p=0}^\infty \frac{(\lambda^22pK^2)^p}{p!}.
\]
Now, by setting the value of $\lambda$ to $\frac{1}{\sqrt{6}K}$, 
$
\bbP(|\fkx|\geq t) \leq \enumber^{-\frac{t^2}{8K^2}}\sum_{p=0}^\infty \frac{(p/3)^p}{p!}.
$
The right-hand series is convergent and, after adding numerically the series, we can see that
$
\sum_{p=0}^\infty \frac{(p/3)^p}{p!}= 2.625\ldots\leq \enumber,
$
which finishes the proof of the first part. Following the constants in the proof of~\cite[Proposition 2.5.2]{vershyninbook} directly seems to give  $4\enumber\simeq 10.8$ in the denominator of the exponent instead of $6$.

For the 
second one, note that
\[
\bbE|\fkx|^p=K^p\left(2\ln C\right)^{\frac{p}{2}}+\int_0^\infty pu^{p-1}\enumber^{-\frac{u^2}{2K}}\,\mathrm{d}u,
\]
which follows from
\[
\bbP(|\fkx|>u)\leq \begin{cases}1&\text{if }u\leq K\sqrt{2\ln C}\\
\enumber^{-\frac{u^2}{2K^2}}&\text{if }u\geq K\sqrt{2\ln C},
\end{cases}
\]
dividing the integration domain into $[0,K\sqrt{2\ln C}]$ and 
$[K\sqrt{2\ln C},\infty]$, and applying some straightforward 
calculations and bounds.

Now, applying the change of variables $t=\frac{u^2}{2K}$, we obtain 
\[
\int_0^\infty pu^{p-1}\enumber^{-\frac{u^2}{2K}}\,\mathrm{d}u
=pK^p2^{\frac{p}{2}-1}\Gamma\left(\frac{p}{2}\right) \leq
K^p\left(\frac{\pi p}{2}\right)^{\frac{p}{2}}.
\]
Hence
\[
 \bbE|\fkx|^p\leq K^p\left(\left(2\ln C\right)^{\frac{p}{2}}+\left(\frac{\pi p}{2}\right)^{\frac{p}{2}}\right),
\]
from where the second part follows.
\end{proof}

\begin{proof}[Proof of Proposition~\ref{prop:subgaussianprojection}]
This is a version 
of~\cite[Proposition 2.6.1]{vershyninbook}. Let us sketch a proof to see the values of the chosen constants.

Let $\fky\in\bbR$ be a centered random variable with $\psi_2$-norm at most $K$. Arguing as in the part `\textbf{ii} $\Rightarrow$ \textbf{iii}' of the proof of~\cite[Proposition 2.5.2]{vershyninbook}, we have that for all $\lambda\in [-1/\sqrt{2\enumber},1/\sqrt{2\enumber}]$,
\[\bbE \enumber^{\lambda^2\fky^2}\leq \enumber^{\enumber K^2\lambda^2},\]
using $n!\geq \sqrt{2\pi}(n/e)^n$, and that for $x\in [-1/2,1/2]$ we have $1+\frac{1}{\sqrt{2\pi}}\frac{x^2}{1-x^2}\leq \enumber^{x^2/2}$.
Then, arguing as in the part `\textbf{iii} $\Rightarrow$ \textbf{v}' of the proof of~\cite[Proposition 2.5.2]{vershyninbook}, we get that for all $\lambda\in\bbR$,
\begin{equation}\label{eq:ineqcenteredsubgaussian}
    \bbE\enumber^{\lambda\fky}\leq \enumber^{\enumber K^2\lambda^2}.
\end{equation}
In this way, we have that
\begin{align*}
    \bbP(|a^*\fkx|\geq t)&\leq 2\bbP(a^*\fkx\geq t)&\text{(Symmetry)}\\
    &=2\bbP(\enumber^{a^*\fkx}\geq \enumber{t})&\\
    &\leq 2\enumber^{-\lambda t}\bbE\enumber^{\lambda a^*\fkx}&\text{(Markov's inequality)}\\
    &= 2\enumber^{-\lambda t}\prod_{i=1}^N\bbE\enumber^{\lambda a_i\fkx_i}&\text{(}a_1\fkx_1,\ldots,a_N\fkx_N\text{ independent)}\\
    &\leq 2\enumber^{-\lambda t}\prod_{i=1}^N\enumber^{\enumber a_i^2K^2\lambda^2}&\eqref{eq:ineqcenteredsubgaussian}\\
    &= 2\enumber^{-\lambda t+\enumber K^2\lambda^2}&(\|a\|_2=1).
\end{align*}
Taking $\lambda=\frac{t}{2\enumber K^2}$, we get the desired tail bound. The last claim immediately follows from 
Proposition~\ref{prop:subgaussiantailbound}.
\end{proof}

\begin{proof}[Proof of Proposition~\ref{prop:anticoncentrationprojection}]
This is a rewriting of~\cite[Theorem~1.1]{rudelsonvershynin2015} 
using~\cite{livshytspaourispivovarov2016} to get explicit
constants. This rewriting was first 
given in~\cite[Proposition~2.5]{TCTcubeI}. We provide 
the argument for the sake of completeness.

By the SVD, we have $A=P\Sigma Q$ where $P$ is an isometry,
$\Sigma\in\bbR^{k\times k}$ a positive diagonal matrix and 
$Q$ an orthogonal projection. Hence
\[
\bbP_\fkx\left(A\fkx\in U\right)=\bbP_\fkx\left(Q\fkx\in \Sigma^{-1}P^*U\right)
\]
and, since $\vol(\Sigma^{-1}P^*U)=\vol(U)/\det\Sigma=\vol(U)/\sqrt{\det(AA^*)}$, we only have to prove the claim for the case in which $A$ is an orthogonal projection.

Now, by~\cite[Theorem~1.1]{rudelsonvershynin2015} (see~\cite[Theorem~1.1]{livshytspaourispivovarov2016} for getting the constant), we have that $A\fkx$ has density bounded by $\sqrt{2}\rho$. Thus $\bbP(A\fkx\in U)\leq \vol(U)(\sqrt{2}\rho)^k$, as we wanted to show.
\end{proof}

\subsubsection{\texorpdfstring{$\cK$ vs. $\kappa$}{K vs. kappa}: Measuring the effect of the \texorpdfstring{$L_\infty$}{Linfty}-norm on the Grid Method}

The condition-based complexity estimates we obtained in this section 
essentially substitute the $\kappa$ in the cost estimates of 
the original algorithm by $\cK$. In this way, the comparison between 
the two algorithm reduces to estimate $\cK/\kappa$. The 
following proposition shows that, in turn, this amounts to look at the 
quotient $\|f\|_\infty\uR/\|f\|_W$.

\begin{prop}\label{prop:kappavsck}
Let $f\in\Hd\uR[q]$ and $x\in\bbS^n$. Then
\[
\frac{\|f\|_\infty\uR}{\|f\|_W}\leq \frac{\cK(f,x)}{\kappa(f,x)}\leq \sqrt{2q\D}\frac{\|f\|_\infty\uR}{\|f\|_W}
\]
and
\[
\frac{\|f\|_\infty\uR}{\|f\|_W}\leq \frac{\cK(f)}{\kappa(f)}\leq \sqrt{2q\D}\frac{\|f\|_\infty\uR}{\|f\|_W}
.\]
\end{prop}

\begin{proof}
It follows from
\[
\frac{\cK(f,x)}{\kappa(f,x)}=\sqrt{q}\frac{\|f\|_\infty\uR}{\|f\|_W}\frac{\sqrt{\|f(x)\|^2+\sigma_q\left(\Delta^{-\frac{1}{2}}\diff_xf\right)^2}}{\max\left\{\|f(x)\|,\sigma_q\left(\Delta^{-1}\diff_xf\right)\right\}}
\]
and 
\begin{equation}\tag*{\qed}
\frac{1}{\sqrt{\D}}\;\sigma_q\left(\Delta^{-\frac{1}{2}}\diff_xf\right)
\leq \sigma_q\left(\Delta^{-1}\diff_xf\right)
\leq \sigma_q\left(\Delta^{-\frac{1}{2}}\diff_xf\right).
\end{equation}
\renewcommand{\qedsymbol}{}
\end{proof}

In general, we have that
$\frac{\|f\|_\infty\uR}{\|f\|_W}\leq 1$ 
so the corresponding quotient of condition numbers worsens 
by a factor of at most $\sqrt{2q\D}$. Our main 
result derives from the fact that $\frac{\|f\|_\infty\uR}{\|f\|_W}$  
is, for a 
substantial number of $f$'s, much smaller than~1: we can expect 
it to be smaller than $\sqrt{n\ln(\enumber\D)/N}$ with very high probability. Recall that $K \rho$ is a constant coming from the randomness model.

\begin{theo}\label{theo:normquitentreal}
Let $q\leq n+1$, $\fkf\in\Hd\uR[q]$ be dobro with parameters $K$ and 
$\rho$ and $\ell\in\bbN$. For any power $\ell$ with  $1 \leq \ell < \frac{N}{2}$ we have
\[
\bbE_{\fkf}\left(\frac{\|\fkf\|_\infty\uR}{\|\fkf\|_W}\right)^\ell \leq
\left(\frac{890\sqrt{2}K\rho\sqrt{n\ln(eD)\ell}}{\sqrt{N-2\ell}}\right)^\ell.
\]
In particular,
\[ 
\bbE_{\fkf} \frac{\|\fkf\|_\infty\uR}{\|\fkf\|_W}  \leq
 \Oh\left(K\rho\sqrt{ \frac{n \ln(e \D)}{N}} \right).
\]
\end{theo}

\begin{remark}\label{rem:quotientnorms}
In the study of tensors the quotients $\|\fkf\|_\infty\uR/\|\fkf\|_W$ and their non-symmetric analogue
play an important role. Because of this, we can consider
Theorem~\ref{theo:normquitentreal} as a symmetric analogue of
the results shown in~\cite{grossflammiaeisert2009}
and~\cite{nguyendrineastran2015}. In a paper under 
preparation by Kozhasov and the third author~\cite{KTC20-normratio}, the probabilistic techniques introduced in this paper are developed further to study $\|\fkf\|_\infty\uR/\|\fkf\|_W$ in several settings.
\end{remark}

\begin{cor}\label{cor:expectationconditionnumbers}
Let $q\leq n+1$ and $\fkf\in\Hd\uR[q]$ be dobro with parameters $K$ and $\rho$. Then for $1 \leq \ell < \frac{N}{2}$ we have,
\begin{equation}\tag*{\qed}
\bbE_{\fkf}\left(\frac{\cK(\fkf)}{\kappa(\fkf)}\right)^\ell \leq
\left(\frac{1780}K\rho\sqrt{qn\D\ln(eD)\ell}{\sqrt{N-2\ell}}\right)^\ell.
\end{equation}
\end{cor}

Let \textsc{PolyBetti}$_W$ be the version of \nameref{alg:polybetti} 
using the Weyl norm and $\kappa$. An analysis along the lines 
of~\cite{CKS16} (or~\cite{BCL17}) shows that the cost of
\textsc{PolyBetti}$_W$ is
\[2^{\Oh(n^2\log n)}\D^{10 n}\kappa(f)^{10 n}\]
which is very similar to the cost bound for~\nameref{alg:polybetti} 
in Proposition~\ref{prop:runtimebetti}. 
Let us denote by
$\textrm{est-run-time}(\text{\nameref{alg:polybetti}},f)$ 
and 
$\textrm{est-run-time}(\textsc{PolyBetti}_W,f)$ 
these cost bounds.
It follows that
\[
\frac{\textrm{est-run-time}(\text{\nameref{alg:polybetti}},f)}{\textrm{est-run-time}(\textsc{PolyBetti}_W,f)}\leq\left(\frac{\cK(f)}{\kappa(f)}\right)^{10n}.
\]
Using Corollary~\ref{cor:expectationconditionnumbers} and Markov's inequality, it is easy to prove the following estimate.

\begin{cor}\label{cor:betti-gain}
Let $q\leq n+1$, $N>20n$ and $\fkf\in\Hd\uR[q]$ be dobro with 
parameters $K$ and $\rho$,
\[
\frac{\mathrm{est}\text{-}\mathrm{run}\text{-}\mathrm{time}(\text{\nameref{alg:polybetti}},f)}{\mathrm{est}\text{-}\mathrm{run}\text{-}\mathrm{time}(\textsc{PolyBetti}_W,f)} \leq \left(\frac{5700}K\rho n\sqrt{q\D\ln(\enumber\D)}{\sqrt{N-20n}}\right)^{10n}
\]
with probability at least $1-1/N$. Note that for fixed 
$n$ and large $\D$, the ratio in the right-hand side is of 
the order of
\begin{equation}\tag*{\qed} 
\left( \frac{ K\rho\sqrt{\ln(\enumber\D)}}{\D}^{\frac{n-1}{2}}  \right)^{10n} .
\end{equation}
\end{cor}

We proceed to prove Theorem~\ref{theo:normquitentreal}. 

\begin{prop}\label{prop:taildobroinfty}
Let $\fkf\in\Hd\uR[q]$ be dobro with parameters $K$ and $\rho$. 
Then, for all $t>0$,
\[
\bbP\left(\|\fkf\|_\infty\uR\geq t\right) \leq
q\sqrt{2\pi}\sqrt{n+1}\left(\frac{\enumber \D}{2}\right)^n\enumber^{-\frac{t^2}{17K^2}}.
\]
In particular, if $q\leq n+1$, for all $\ell\geq 1$,
$\left(\bbE_\fkf\left(\|\fkf\|_\infty\uR\right)^\ell\right)^{\frac{1}{\ell}}
\leq 63K\sqrt{n\ln(\enumber\D)\ell}$.
\end{prop}

\begin{proof}[Proof of Theorem~\ref{theo:normquitentreal}]
By the Cauchy-Schwarz inequality,
\begin{align*}
\bbE_{\fkf}\left(\frac{\|\fkf\|_\infty\uR}{\|\fkf\|_W}\right)^\ell
&\leq \sqrt{\bbE_{\fkf}\left(\|\fkf\|_\infty\uR\right)^{2\ell}}
\sqrt{\bbE_{\fkf}\frac{1}{\|\fkf\|_W^{2\ell}}}.
\end{align*}
The first term in the right is bounded by
Propositions~\ref{prop:taildobroinfty}. 

For the second term, we will use \cite[Theorem~1.11]{mendelsonpaouris2019}. 
We note that $\fkx\in \bbR^N$ satisfies the \emph{Small Ball Assumption (SBA)
with constant $\mcL$}~\cite[Assumption 1.1.]{mendelsonpaouris2019} if for
every $k\in\{1,\ldots,N-1\}$, every orthogonal projection 
$P\in \bbR^{k\times N}$, every $y\in\bbR^k$ and every $\varepsilon>0$,
\[
\bbP\left(\|P\fkx-y\|_2\leq \sqrt{k}\varepsilon\right)
\leq (\mcL\varepsilon)^k.
\]
By Proposition~\ref{prop:anticoncentrationprojection} (applied with 
coordinates orthogonal with respect to the Weyl inner product) and Stirling's
approximation, we have that $\fkf$ has the SBA with constant 
$2\sqrt{\pi e}\rho$. 
Thus, by~\cite[Theorem~1.11]{mendelsonpaouris2019},
\[
\bbE_{\fkf}\frac{1}{\|\fkf\|_W^{2\ell}}\leq (14\rho)^{2\ell} \bbE_{\fkg}\frac{1}{\|\fkg\|_W^{2\ell}}
\]
where $\fkg\in\Hd[q]$ is KSS. Since $\fkg$ is a Gaussian vector 
for all coordinate systems orthogonal with respect to the 
Weyl inner product, $\|\fkg\|_W^{2}$ 
is distributed according to a $\chi^2$-distribution with $N$ degrees 
of freedom. Therefore
\[
\bbE_{\fkg}\frac{1}{\|\fkg\|_W^{2\ell}}
=\int_{0}^\infty t^{-\ell}\frac{1}{2^{\frac{N}{2}}
\Gamma\left(\frac{N}{2}\right)}t^{\frac{N}{2}-1}
\enumber^{-\frac{t}{2}}\,\mathrm{d}t
=\frac{\Gamma\left(\frac{N}{2}-\ell\right)}{2^{\ell}
\Gamma\left(\frac{N}{2}\right)}=\frac{1}{(N-2)(N-4)\cdots(N-2\ell)}.
\]
The desired claim now follows.
\end{proof}

\begin{proof}[Proof of Proposition~\ref{prop:taildobroinfty}]
Fix $\delta\in[0,1/\D]$. By the proof of
Proposition~\ref{prop:estimationnorm}, we have that 
$\|\fkf\|_\infty\uR>t$ implies 
$\vol\left\{x\in\bbS^n\mid \|\fkf(x)\|_\infty \geq
\left(1-\frac{\D^2}{2}\delta^2\right)t\right\} \geq 
\vol B_{\bbS}(x_*,\delta)$, 
where $x_*\in\bbS^n$ maximizes $\|f(x)\|_\infty$. Therefore
\begin{align*}
    \bbP\left(\|\fkf\|_\infty\uR\geq t\right)&
    \leq \bbP_{\fkf}\left(\bbP_{\fkx\in\bbS^n}
    \left(\|\fkf(\fkx)\|_\infty\geq
    \left(1-\frac{\D^2}{2}\delta^2\right)t\right)\geq 
    \vol B_{\bbS}(x_*,\delta)/\vol\bbS^n \right).
\end{align*}
By~\cite[Lemma~2.25]{Condition},~\cite[Lemma~2.31]{Condition} and $\int_0^\delta\,n\sin^{n-1}\theta\,\mathrm{d}\theta\geq (1-\delta^2/6)^n\delta^n$, we have that
\[
\vol_n B_{\bbS}(x_*,\delta)/\vol_n\bbS^n\geq  \frac{\left(1-\delta^2/6\right)^n}{\sqrt{2\pi}\sqrt{n+1}}\delta^n.
\]
In this way,
\begin{align*}
\bbP&\left(\|\fkf\|_\infty\uR\geq t\right)\\
&\leq \bbP_{\fkf}\left(\bbP_{\fkx\in\bbS^n}\left(\|\fkf(\fkx)\|_\infty
\geq \left(1-\frac{\D^2}{2}\delta^2\right)t\right)
\geq \frac{\left(1-\delta^2/6\right)^n}{\sqrt{2\pi}\sqrt{n+1}}\delta^n\right) \\
&\leq \frac{\sqrt{2\pi}\sqrt{n+1}}{\left(1-\delta^2/6\right)^n\delta^n}\,
\bbE_\fkf \bbP_{\fkx\in\bbS^n}\left(\|\fkf(\fkx)\|_\infty \geq
\left(1-\frac{\D^2}{2}\delta^2\right)t\right)&\text{(Markov's inequality)}\\
&\leq \frac{\sqrt{2\pi}\sqrt{n+1}}{\left(1-\delta^2/6\right)^n\delta^n}\,
\bbE_{\fkx\in\bbS^n} \bbP_{\fkf}\left(\|\fkf(\fkx)\|_\infty
\geq \left(1-\frac{\D^2}{2}\delta^2\right)t\right)&\text{(Tonelli's theorem)}\\
&\leq \frac{\sqrt{2\pi}\sqrt{n+1}}{\left(1-\delta^2/6\right)^n\delta^n}\,
\max_{x\in\bbS^n} \bbP_{\fkf}\left(\|\fkf(x)\|_\infty \geq
\left(1-\frac{\D^2}{2}\delta^2\right)t\right)\\
&\leq \frac{q\sqrt{2\pi}\sqrt{n+1}}{\left(1-\delta^2/6\right)^n\delta^n}\,
\max_{i,x\in\bbS^n} \bbP_{\fkf}\left(|\fkf_i(x)|_\infty
\geq \left(1-\frac{\D^2}{2}\delta^2\right)t\right)&\text{(Union bound)}.
\end{align*}

In the coordinates of an
monomial basis orthogonal for the Weyl inner product, the following 
holds: 
(1) a dobro random polynomial $\fkf$ looks like a random vector whose 
components are independent and subgaussian of $\psi_2$-norm 
at most $K$, and (2) evaluation at a point of the sphere, $\fkf(x)$, becomes inner product with a vector of norm~1 (by Proposition~\ref{prop:evaluationorthogonal}). Hence, by Proposition~\ref{prop:subgaussiantailbound},
\[
\bbP\left(\|\fkf\|_\infty\uR\geq t\right)
\leq \frac{q\sqrt{2\pi}\sqrt{n+1}}{\left(1-\delta^2/6\right)^n\delta^n}\,
\exp\left(-\left(1-\frac{\D^2}{2}\delta^2\right)^2\frac{t^2}{11K^2}\right).
\]
The claim follows taking $\delta=5/(6\D)$ and $\left(1-\frac{1}{2}\left(\frac{5}{6}\right)^2\right)\frac{1}{11}\geq \frac{1}{17}$. For the other inequalities on the moments use Proposition~\ref{prop:subgaussiantailbound}.
\end{proof}

\subsubsection{Complexity of Plantinga-Vegter Algorithm}\label{sec:PV-average}

In~\cite{CETC-PVjorunal} (cf.~\cite{CETC-PV}), we proved the following 
result (which we are just adapting to the notation\footnote{There is 
a slight difference in the way the anti-concentration constant
is defined in~\cite{CETC-PVjorunal} and here.} of 
this paper).  

\begin{theo}\cite[Theorem~8.4 and Theorem 7.3]{CETC-PVjorunal}\label{thm:kappaonedobropolynomial}
Let $\fkf\in\Hd\uR[1]$ be dobro with parameters $K$ and $\rho$. 
For all $x\in\bbS^n$ and $t\geq e$,
\[
\bbP \left(\kappa(\fkf,x) \geq t\right) \leq 
 2 \left(\frac{N}
 {n+1}\right)^{\frac{n+1}{2}}(30 K\rho)^{n+1}
 \frac{\ln^{\frac{n+1}{2}}t}{t^{n+1}} .
\]
In particular, for Plantinga-Vegter algorithm 
with input $\fkf$ over 
the domain $[-a,a]^n$ the expected number of hypercubes in the 
final subdivision is at most 
\begin{equation}\tag*{\qed}
 a^n \D^{n}  N^{\frac{n+1}{2}} 2^{n \log n + 13n + \frac{3}{2}\log n+\frac{17}{2}}
(K\rho)^{n+1}.
\end{equation}
\end{theo}

Our objective is the following theorem, which shows how the
$N^{\frac{n+1}{2}}$ factor vanishes from these estimates when 
we pass from $\kappa$ to $\cK$. This shows that the version 
of  Plantinga-Vegter using $\cK$ yields better cost bounds than the one 
using $\kappa$, i.e., the one in~\cite{CETC-PVjorunal}.

\begin{theo}\label{thm:cKonedobropolynomial}
Let $\fkf\in\Hd\uR[1]$ be dobro with parameters $K$ and $\rho$. 
For all $x\in\bbS^n$ 
and $t\geq e$,
\[
\bbP \left(\cK(\fkf,x) \geq t\right) \leq 
 \D^{\frac{n}{2}}(\ln \enumber\D)^{\frac{n+1}{2}}2^{6n+4}(K\rho)^{n+1}
 \frac{\ln^{\frac{n+1}{2}}t}{t^{n+1}} .
\]
It follows that for every compact $\Omega\subseteq \bbS^n$,
\[
\bbE_\fkf\bbE_{\fkx\in \Omega}\left(\cK(\fkf,\fkx)^{n}\right)
\leq  \D^{\frac{n}{2}}(\ln \enumber\D)^{\frac{n+1}{2}} 2^{\frac{1}{2}n\log n+5n+2\log(n)+7}
(K\rho)^{n+1}.
\]
In particular, for Plantinga-Vegter algorithm with input $\fkf$
over the domain $[-a,a]^n$ 
the expected number of hypercubes in the 
final subdivision is at most 
\[ a^n \D^{\frac{3n}{2}}(\ln \enumber\D)^{\frac{n+1}{2}} 2^{\frac{3}{2}n\log n+13n+2\log(n)+7} \]
\end{theo}

\begin{remark}\label{rem:costPV}
Theorem~\ref{thm:cKonedobropolynomial} allows us to compare the 
efficiency of Plantinga-Vegter for the versions based on the 
Weyl-norm and the $\infty$-norm. One can observe that (in 
the region of interest $\D > n$) the term $N^{\frac{n}{2}} \sim
\D^{\frac{n^2}{2}}$  in the estimate for the Weyl-norm version is 
replaced with $(\D \log \D)^{\frac{n}{2}} $ in the $\infty$-norm. 
Basically the exponent of $\D$ goes from $\Oh(n^2)$ to 
$\Oh(n)$. If we focus on the original cases of interest
(cf.~\cite{PV}), that is $n=2$ and $n=3$, with the average complexity 
analysis from~\cite{CETC-PVjorunal}, it is shown 
in Theorem~3.1 there that {\sc PV-Interval}$_W$ has an average complexity 
of 
\begin{eqnarray*}
\Oh\left(d^8\max\{1,a^2\}(K\rho)^3\right) 
&\qquad& \mbox{for $n=2$, and}\\
\Oh\left(d^{13}\max\{1,a^3\}(K\rho)^4\right)
&\qquad& \mbox{for $n=3$.}
\end{eqnarray*}
It follows from Theorems~\ref{thm:PVMAIN2} 
and~\ref{thm:cKonedobropolynomial} that the 
average complexity of {\sc PV-Interval}$_\infty$ is 
\begin{eqnarray*}
\Oh\left(d^7\log^{1.5}(d)\max\{1,a^2\}(K\rho)^3\right)
&\qquad& \mbox{for $n=2$, and}\\
\Oh\left(d^{10}\log^2(d)\max\{1,a^3\}(K\rho)^4\right)
&\qquad& \mbox{for $n=3$.}
\end{eqnarray*}
\end{remark}

We next proceed to prove Theorem~\ref{thm:cKonedobropolynomial}.

\begin{proof}[Proof of Theorem~\ref{thm:cKonedobropolynomial}]
Let $u,t\geq 0$, then
\begin{align*}
    \bbP_\fkf&\left(\cK(\fkf,x)\geq t\right)\\
    &\leq \bbP_\fkf\left(\|f\|_\infty\uR\geq u\text{ or }\max\left\{|\fkf(x)|,\frac{\|\diff_x\fkf\|}{\D}\right\}\leq \frac{u}{t}\right)&\text{(Implication bound)}\\
    &\leq \bbP_\fkf\left(\|f\|_\infty\uR\geq u\right)+\bbP_\fkf\left(\max\left\{|\fkf(x)|,\frac{\|\diff_x\fkf\|}{\D}\right\}\leq \frac{u}{t}\right),&\text{(Union bound)}
\end{align*}
where we used the fact that for $f\in\Hd\uR[1]$,
$\cK(f,x)=\|f\|_\infty\uR/\max\left\{|f(x)|,\|\diff_xf\|/\D\right\}$.

On the one hand, $\bbP_\fkf\left(\|f\|_\infty\uR\geq u\right)$ is bounded 
by Proposition~\ref{prop:taildobroinfty}. On the other hand, 
the map
\[
f\mapsto \begin{pmatrix}
f(x)&\frac{\diff_xf}{\D}
\end{pmatrix}
\]
has singular values $1,1/\sqrt{\D},\ldots,1/\sqrt{\D}$ in the
coordinates of a monomial basis orthogonal with respect to 
the Weyl inner product. And since in such a basis a dobro 
polynomial is a vector whose coefficients are independent 
and have the anti-concentration property with constant $\rho$, 
we deduce that
\begin{multline*}
 \bbP_\fkf\left(\max\left\{|\fkf(x)|,\frac{\|\diff_x\fkf\|}{\D}\right\}
 \leq \frac{u}{t}\right)\leq
 \D^{\frac{n}{2}}\vol\left\{(x_0,x)\in\bbR^{n+1}\mid |x_0|,\|x\|
 \leq u/t\right\}\left(\sqrt{2}\rho\right)^{n+1}\\
 \leq \omega_n\D^{\frac{n}{2}}\left(\frac{\sqrt{2}\rho u}{t}\right)^{n+1}
 \leq 9^n\D^{\frac{n}{2}}\left(\frac{u\rho}{\sqrt{n}}\right)^{n+1}
 \frac{1}{t^{n+1}},   
\end{multline*}
where $\omega_n$ is the volume of the unit $n$-ball and we used 
Proposition~\ref{prop:anticoncentrationprojection} and 
Stirling's estimation~\cite[Eq.~(2.14)]{Condition}.

Hence, combining the inequalities above,
\[
\bbP_\fkf\left(\cK(\fkf,x)\geq t\right)\leq
\sqrt{2\pi(n+1)}\left(\frac{\enumber\D}{2}\right)^n\enumber^{-\frac{u^2}{17K^2}}
+9^n\D^{\frac{n}{2}}\left(\frac{u\rho}{\sqrt{n}}\right)^{n+1}
\frac{1}{t^{n+1}}.
\]
Taking $t\geq e$ and $u=\sqrt{17}K\sqrt{n\ln(\enumber^2\D)\ln t}\geq \sqrt{17}K\sqrt{n\ln \D+(n+1)\ln t}$, we get
\[
\bbP_\fkf\left(\cK(\fkf,x)\geq t\right)\leq \frac{\sqrt{2\pi(n+1)}}{(2\enumber)^n}\frac{1}{t^{n+1}}+9^n\D^{\frac{n}{2}}\left(\sqrt{17}K\rho\ln^{\frac{1}{2}(\enumber^2 \D)}\right)^{n+1}\frac{\ln^{\frac{n+1}{2}}t}{t^{n+1}}.
\]
This proves the first statement.

By Tonelli's theorem, in order to prove the second statement it is 
enough to bound $\bbE_\fkf\cK(\fkf,x)^n$ for a fixed $x\in\bbS^n$. 
Now,
\begin{multline*}
\bbE_\fkf\cK(\fkf,x)^n=\int_{0}^\infty\bbP_\fkf\left(\cK(\fkf,x)
\geq t^{\frac{1}{n}}\right)\\
\le \enumber^n+\int_{\enumber^n}^\infty\bbP_\fkf
\left(\cK(\fkf,x)\geq t^{\frac{1}{n}}\right)
\leq \enumber^n+\int_{\enumber^n}^\infty
\D^{\frac{n}{2}}\ln^{\frac{n+1}{2}}(\enumber\D)2^{6n+4}(K\rho)^{n+1}
 \frac{\ln(t^{\frac{1}{n}})^{\frac{n+1}{2}}}{t^{1+\frac{1}{n}}}\,
 \mathrm{d}t\\
 \leq \enumber^n+\D^{\frac{n}{2}}\ln^{\frac{n+1}{2}}(\enumber\D)2^{6n+4}(K\rho)^{n+1}
 \int_{1}^\infty\frac{\ln(t^{\frac{1}{n}})^{\frac{n+1}{2}}}
 {t^{1+\frac{1}{n}}}\,\mathrm{d}t.
\end{multline*}
By changing variables, $t=\enumber^{sn}$, we can see that
\[
\int_{1}^\infty\frac{\ln(t^{\frac{1}{n}})^{\frac{n+1}{2}}}
{t^{1+\frac{1}{n}}}\,\mathrm{d}t
=n\Gamma\left(\frac{n+3}{2}\right)
\leq \sqrt{2\pi\enumber}\,n\sqrt{n+1}
\left(\frac{n+1}{2\enumber}\right)^{\frac{n+1}{2}}
\]
where the inequality comes from Stirling's
approximation~\cite[Eq.~(2.14)]{Condition}. Hence we get
\begin{multline*}
\bbE_\fkf\cK(\fkf,x)^n\leq 
\enumber^n+\sqrt{2\pi\enumber}n\sqrt{n+1}\D^{\frac{n}{2}}
\ln^{\frac{n+1}{2}}(\enumber\D)\frac{2^{6n+4}}{(2\enumber)^{\frac{n+1}{2}}}(\sqrt{n+1}K\rho)^{n+1}\\
\leq 8n\sqrt{n+1}\D^{\frac{n}{2}}\ln^{\frac{n+1}{2}}
(\enumber\D)\frac{2^{6n+4}}{(2\enumber)^{\frac{n+1}{2}}}(\sqrt{n+1}K\rho)^{n+1}.
\end{multline*}
The second statement now follows after some easy bounds.
\end{proof}

\section{Linear Homotopy for Computing Complex Zeros}\label{sec:homotopy}
Smale's 17th problem asks  if a complex zero of $n$ complex polynomial
equations in $n+1$ homogeneous unknowns can be found on average 
polynomial time \cite{Smale00}.
A  probabilistic solution to Smale's 17th  problem was given 
by Beltr\'an and Pardo in 2009~\cite{BePa:08,BePa:09}. The construction of 
Beltr\'an and Pardo was probabilistic in the sense that they exhibited a 
{\em randomized} algorithm. 

The distribution underlying the average-case analysis for Beltr\'an-Pardo
algorithm is the complex version of the KSS distribution 
(see Example~\ref{ex:KSS}). Finally, the expected running time 
of Beltr\'an-Pardo's algorithm is polynomial in 
$N=\dim_{\bbC}\Hd^{\bbC}[n]$. 

A generic square system of equations with degrees $d_1,d_2,\ldots,d_n$ has 
$\mcD:=d_1\cdot \cdots\cdot d_n$  many zeros, and Smale's 17th problems 
asks to compute one of these zeros. Following the initial work by Shub 
and Smale~\cite{Bez1}, the hearth of Beltr\'an-Pardo solution is a 
{\em linear homotopy}, let's call 
it {\sc ALH}. It takes as input the system $f$ for which a zero 
is sought, along with an initial pair 
$(g,\zeta)\in\Hd^{\bbC}[n]\times{\bbP}^n$ satisfying $g(\zeta)=0$. 
If we define $q_t:=tf+(1-t)g$, for $t\in[0,1]$, then  
generically, the segment $[g,f]$ in $\Hd^{\bbC}[n]$ lifts to 
a curve $\{(q_t,\zeta_t)\mid t\in[0,1]\}$ in the 
{\em solution variety}
$$
   \mcV:=\{(f,\zeta)\in\Hd^{\bbC}[n]\times{\bbP}^n
 \mid f(\zeta)=0\}.
$$
The idea of {\sc ALH}, in a nutshell, is to ``follow" this 
curve (for which we know its origin $(g,\zeta)$) close enough 
so that we end up with an approximation to the zero $\zeta_1$ 
of $f=q_1$. 

The breakthrough in~\cite{BePa:08,BePa:09} was to come up 
with a randomized algorithm to produce the (long sought) initial 
pair $(g,\zeta)$. To state this result, we endow 
$\mcV$ with the {\em standard distribution} $\rst$ 
defined via the following procedure:
\begin{itemize}
\item draw a complex KSS system $\fkf\in\Hd^{\bbC}[n]$.
\item draw $\zeta$ from the $\mcD$ zeros of $\fkf$ with the uniform distribution.   
\end{itemize}
For details on $\rst$ see~\cite[\S17.5]{Condition}. 
The description of $\rst$ above is not constructive: 
it merely describes the distribution. It is remarkable however 
that it is possible to efficiently sample from $\rst$.

\begin{prop}(\cite[Proposition~17.21]{Condition})\label{prop:BePa}
There is a randomized algorithm which, with input $n$ and 
$\bfd$, returns a pair $(g,\zeta)\in\mcV$ drawn from $\rst$. 
The algorithm performs $2(N+n^2+n+1)$ draws of random real numbers 
from the standard Gaussian distribution and $O(\D nN+n^3)$ arithmetic operations.\eproof
\end{prop}

With this randomization procedure at hand, the structure of the 
algorithm to compute approximate zeros is simple. 

\begin{minipage}[t]{0.9\textwidth}
\begin{algorithm*}[H]
\DontPrintSemicolon
\SetKwInOut{input}{Input}
\SetKwInOut{output}{Output}
\SetKwInOut{halt}{Halting cond.}
\caption{\textsc{Solve}}\label{alg:solve}
\input{$f\in\Hd[n]$}
\precondition{$f\neq0$}
\hrulefill\\
draw $(g,\zeta)\in\mcV$ from $\rst$\;
run {\sc ALH} on input $(f,g,\zeta)$\;
\hrulefill\\
\output{$z\in\bbC^{n+1}_*$}
\postcondition{$z$ is an approximate zero of $f$}
\halt{The lifting of $[g,f]$ at $\zeta$ does not cut $\tilde{\Sigma}\subseteq\mcV$}
\end{algorithm*}
\end{minipage}
\medskip

Here $\tilde{\Sigma}:=\{(f,\zeta)\in\mcV\mid \det\diff_\zeta f=0\}$.  
This set has complex codimension~1 in $\mcV$. Hence, because 
the lifting of the segment $[g,f]$ corresponding to 
$\zeta$ has real dimension~1, generically, it does not cut 
$\tilde{\Sigma}$. That is, algorithm \nameref{alg:solve} almost surely 
terminates for almost all inputs $f\in\Hd[n]$.

Regarding complexity, the total cost of {\sc Solve} is dominated 
by that of running {\sc ALH}, which is given by the 
number of steps $K$ performed by the homotopy times the cost 
of each step. In previous work~(\cite{Bez1,BePa:08,BePa:09,BuCu11,ABBCS:16} 
among others) the latter is essentially optimal as it is 
$O(N+n^3)$ (which is $O(N)$ if $d_i\ge2$ for $i=1,\ldots,n$). 
The former depends on the input at hand and it is there where 
average considerations play a role. 
In~\cite{BePa:11,BuCu11} {\sc ALH} was 
implemented 
using the Weyl norm to compute step-lengths. Its average 
number of iterations is $O(n\D^{3/2}N)$. The average total complexity 
of the resulting algorithm, let us call it~{\sc Solve$_W$},  
is then $O(n\D^{3/2}N^2)$. 

The goal of this section is to analyze a version~{\sc ALH$_\infty$} 
of {\sc ALH} with step-lengths based on $\|~\|_{\infty}$. We show that 
this can be done in a straightforward manner and that, maybe surprisingly, 
the  average number of iterations of ALH with step-lengths based on our 
new condition number is $\Oh(n^3\D^2\ln(n\D))$: a bound  
independent of $N$. 
Unfortunately, this gain is not decisive for a general
input model due to the high cost 
of computing $\|~\|_{\infty}$ norms. 

Nonetheless, for the particular ---but highly relevant--- case 
of quadratic polynomials, we can efficiently compute the 
$\infty$-norm. As a result we derive bounds that show the expected complexity of~\nameref{alg:solve}$_\infty$ 
is smaller than the expected complexity of~\nameref{alg:solve}$_W$. 

\subsection{Description of the linear homotopy}

The algorithm below is, essentially, the one in~\cite{BuCu11} and~\cite[Ch.~17]{Condition}. The only change is in the 
computation of the step-length $\Delta_t$ where we replace 
the original (here $\dist_{\bbS}$ denotes angle)
$$
  \frac{0.008535284}{\dist_{\bbS}(f,g)
  \D^{3/2}\mu_{\mathrm{norm}}^2(q,z)}
$$
by 
\begin{equation}\label{eq:step}
\frac{0.03\,\|q\|_{\infty}\uC}
{\|f-g\|_{\infty}\uC\D\cM^2(q,z)}.
\end{equation}
This change amounts ---leaving aside the
difference in the constants and a smaller 
exponent in $\D$--- to the use of 
the $\infty$-norm instead of the Weyl one and,   
consequently, the use of $\cM$ instead of $\mu_{\mathrm{norm}}$. 
Recall that $N_q$ is the Newton operator associated 
to $q\in\Hd[n]$.
\begin{minipage}[t]{0.9\textwidth}
\begin{algorithm*}[H]
\DontPrintSemicolon
\SetKwInOut{input}{Input}
\SetKwInOut{output}{Output}
\caption{\textsc{ALH}$\infty$}\label{alg:linhomotopy}
\input{$f,g\in\Hd[n]$ and $\zeta\in\bbP^n$}
\precondition{$g(\zeta)=0$}
\hrulefill\\
$t\leftarrow 0$, $q\leftarrow g$, 
$z\leftarrow \zeta$\;
\Repeat{$t=1$}{
$\Delta t\leftarrow \frac{0.03\,\|q\|_{\infty}}
{\|f-g\|_{\infty}\D\cM^2(q,z)}$ \;
$t\leftarrow \min\{t+\Delta t,1\}$\;
$q\leftarrow tf+(1-t)g$\;
$z\leftarrow N_q(z)$\;
}
\KwRet{$z$} and halt\\
\hrulefill\\
\output{$z\in\bbC^{n+1}_*$}
\postcondition{The algorithm halts if $q_t\not\in\Sigma_{\zeta_t}$ 
for all $t\in [0,1]$. In this case, $z$ is an approximate zero of $f$}
\end{algorithm*}
\end{minipage}
\medskip

\subsection{A bound on the number of iterations}

The analysis of \nameref{alg:linhomotopy} closely follows the steps 
in~\cite{Condition}. It uses the properties of $\cM$ shown 
in Theorem~\ref{theo:propertiescomplexcond} and one more result 
(we know for $\mu_{\mathrm{norm}}$) namely, that 
$\cM$ is a condition number in the standard sense of this 
expression, it measures how solutions change when data is 
perturbed (see Proposition~\ref{prop:growth} below). 
To simplify the notation, in the rest of this section, we will 
often omit the reference to the base field $\bbC$. 

\begin{theo}\label{thm:ALH}
Suppose that 
the lifting of the segment $[g,f]$ in $\mcV$ corresponding to 
$\zeta$ does not cut $\Sigma'$.
Then the algorithm 
\nameref{alg:linhomotopy} stops after at most $K$ steps with 
$$
  K\le 1+45\, \D\,\|f-g\|_{\infty} \int_0^1 
  \frac{\cM^2(q_t,\zeta_t)}{\|q_t\|_{\infty}} {\mathrm d}t.
$$
The returned point $z$ is an approximate zero of $f$ 
with associated zero $\zeta_1$.
\end{theo}

\begin{cor}\label{cor:Kbound}
The bound $K$ in Theorem~\ref{thm:ALH} satisfies 
$$
  K\le 1+45\, n\,\D \int_0^1 
(\| f \|_{\infty} + \| g \|_{\infty})^2  \|\diff_{\zeta_t}q_t^{-1}\Delta\|^2 {\mathrm d}t.
$$
\end{cor}

\begin{prop}\label{prop:growth}
Let $t\mapsto (f_t,\zeta_t)\in V$ be a smooth path. Then, 
for all $t$, 
\[\|\dot{\zeta}_t\|\leq \cM(f_t,\zeta_t)
\frac{\|\dot{f_t}\|_{\infty}}{\|f_t\|_{\infty}}.
\]
\end{prop}

\begin{proof}[Proof in Theorem~\ref{thm:ALH}]
The proof follows the lines of~\cite[Theorem~17.3]{Condition}. 
We will therefore only offer a brief sketch.
Set $\varepsilon:=\frac{1}{4}$ and 
$C=\frac{\varepsilon}{4}=\frac1{16}$. 
Let $q_t:=tf+(1-t)g$.
Also, let $0<t_1<\ldots<t_K=1$ and 
$\zeta_0=z_0,\ldots,z_K$ be the sequence of 
$t$-values and points in $\bbP^n$, respectively, generated by the 
algorithm in its first $K$ iterations. To simplify notation 
we write $q_i$ and 
$\zeta_i$ instead of $q_{t_i}$ and $\zeta_{t_i}$. 

As in~\cite[Theorem~17.3]{Condition}, but using 
Proposition~\ref{prop:B61} in the place of~\cite[Proposition~16.2]{Condition}
and Theorem~\ref{theo:propertiescomplexcond} in the place 
of~\cite[Theorem~16.1]{Condition}, one proves by 
induction the following statements for $i=0,\ldots,K-1$:
\medskip

(a,i) $\dist_{\bbP}(z_i,\zeta_i)\le 
\frac{C}{\D\cM(q_i,\zeta_i)}$

(b,i) $\frac{\cM(q_i,z_i)}{1+\varepsilon} \le 
\cM(q_i,\zeta_i) \le (1+\varepsilon)\cM(q_i,z_i)$

(c,i) $\|q_i-q_{i+1}\|_{\infty} \le
\frac{C\|q_i\|_{\infty}}{\D\cM(q_i,\zeta_i)}$

(d,i) $\dist_{\bbP}(\zeta_i,\zeta_{i+1})\le 
\frac{C}{\D\cM(q_i,\zeta_i)}
\frac{1-\varepsilon}{1+\varepsilon}$

(e,i) $\dist_{\bbP}(z_i,\zeta_{i+1})\le 
\frac{2C}{(1+\varepsilon)\D\cM(q_i,\zeta_i)}$

(f,i) $z_i$ is an approximate zero of $q_{i+1}$ with 
associated zero $\zeta_{i+1}$

By Proposition~\ref{prop:B61}, $(\mathrm{c},i)$, $(\mathrm{d},i)$ and our choice of $C$ and $\varepsilon$, we have that for all $t\in[t_i,t_{i+1}]$,
\begin{equation}\label{eq:B633}
\frac{4}{5}\cM(q_i,\zeta_i)\leq 
\cM(q_t,\zeta_t)\leq \frac{5}{4}\cM(q_i,\zeta_i).
\end{equation}
And, by the triangle inequality and $(\mathrm{b},i)$, for $t\in [t_i,t_{i+1}]$,
\begin{equation}\label{eq:B67}
   \frac{\|q_t\|_{\infty}}{\|q_i\|_{\infty}}
   \le 1+C=\frac{17}{16}.
\end{equation}
The statement now easily follows. 
Consider any $i\in\{0,1,\ldots,K-2\}$. Then
\begin{align*}
    \int_{t_i}^{t_{i+1}}\frac{\cM^2(q_t,\zeta_t)}{\|q_t\|_\infty} \mathrm{d} t &\geq \frac{64}{85}\int_{t_i}^{t_{i+1}}\frac{\cM^2(q_i,z_i)}
{\|q_i\|_\infty} \mathrm{d} t=\frac{64}{85}\frac{\cM^2(q_i,z_i)}
{\|q_i\|_\infty}|t_{i+1}-t_i|  &
\text{(\eqref{eq:B633} and \eqref{eq:B67})}\\
& =\frac{64}{85}\frac{0.03}{\|f-g\|_\infty \D}&\text{(Choice of $\Delta t$)}
\end{align*}
Hence
$$
   \int_{0}^1\frac{\cM^2(q_t,z_t)}{\|q_t\|_{\infty}}\mathrm{d} t 
   \ge \frac{192}{8500}\frac{K-1}{\|f-g\|_\infty \D}
   \ge  \frac{K-1}{45\|f-g\|_{\infty}\D}
$$
and the result follows.
\end{proof}

\begin{proof}[Proof of Corollary~\ref{cor:Kbound}]
It immediately follows from the definition of 
$\cM(q_t,\zeta_t)$ and the inequality 
$\|q_t\|_{\infty}\le \|f \|_{\infty} + \| g\|_{\infty}$.
\end{proof}

\begin{proof}[Proof of Proposition~\ref{prop:growth}]
Recall from~\cite[\S~14.3]{Condition} that the zero $\zeta_t$ is given by 
$\zeta_t=G(f_t)$ where $G:U\subset\Hd[n]\to\bbP^n$ is a local inverse of 
the projection $\pi_1:\mcV\to\Hd[n]$. Hence, for all $\dot{f_t}\in\Hd[n]$ 
we have 
\begin{equation}\label{eq:16.10}
   \dot{\zeta_t}=\diff_{f_t}G(\dot{f_t})=-(\diff_{\zeta_t} f_t)^{-1}(\dot{f_t}(\zeta_t))
\end{equation}
where the second equality is shown in the course of the 
proof of~\cite[Prop.~16.10]{Condition}. Using this equality along with 
the fact that $(\diff_{\zeta_t} f_t)^{-1}=(\diff_{\zeta_t} f_t)^{\dagger}$ (as $q=n$) 
we deduce that 
\begin{align*}
\|\dot\zeta_t\| &= \max_{\|\dot{f_t}\|_{\infty}=1}
\|(\diff_{\zeta_t} f_t)^{-1}(\dot{f_t}(\zeta_t))\|&\text{(By \eqref{eq:16.10})}\\
&\le \left(\max_{\|\dot{f_t}\|_{\infty}=1}\|\dot{f_t}(\zeta_t)\|\right)\|(\diff_{\zeta_t} f_t)^{-1}\|&\text{(Operator norm inequality)}\\
&\le \sqrt{n}\left(\max_{\|\dot{f_t}\|_{\infty}=1}\|\dot{f_t}(\zeta_t)\|_\infty\right)\|(\diff_{\zeta_t} f_t)^{-1}\|&\|~\|\leq \sqrt{n}\|~\|_\infty\\
&\le \sqrt{n}\|(\diff_{\zeta_t} f_t)^{-1}\|&\text{(Definition of }\|~\|_\infty\text{)}\\
&= \frac{\sqrt{n}\|f_t\|_{\infty}\|(\Delta^{-1}\diff_{\zeta_t} f_t)^{-1}\|}{\|f_t\|_{\infty}}
\;=\; \frac{\cM(f_t,\zeta_t)}{\|f_t\|_{\infty}}.&\text{(Definition of }\cM)\text{}
\end{align*}
We recall that the norms where we have omitted subscripts form are  
the usual norm in the case of vectors and the usual operator norm in 
the case of linear maps.
\end{proof}

\subsection{Average complexity analysis of 
\texorpdfstring{{\sc Solve}$_\infty$}{SOLVEinfty}}\label{sec:solvecomplexity}

The execution of \nameref{alg:solve}$_\infty$ on an input
$f\in\Hd\uC[n]$ 
amounts to calling \nameref{alg:linhomotopy} on input
$(f,\fkg,\fkz)$ 
where $(\fkg,\fkz)\in\Hd\uC[n]\times\bbP^n$ is a standard 
random pair. 
Consequently, the number of iterations 
of~\nameref{alg:solve}$_\infty$ 
amounts to the number of iterations done 
by~\nameref{alg:linhomotopy}. 
The latter is a random variable as $(\fkg,\fkz)$ is random. 
We will 
further consider $f$ random and bound the average complexity 
of~\nameref{alg:solve} by taking the expectation over both
$(\fkg,\fkz)$ 
and $f$. Recall that a \emph{KSS complex random polynomial
system}
$\fkf\in\Hd\uC[n]$ is a tuple of random polynomials
\[
\left(\sum_{|\alpha|=d_1}\binom{d_1}{\alpha}^{\frac{1}{2}}\fkc_{1,\alpha}X^\alpha,\ldots,\sum_{|\alpha|=d_n}\binom{d_n}{\alpha}^{\frac{1}{2}}\fkc_{n,\alpha}X^\alpha\right)
\]
such that the $\fkc_{i,\alpha}$ are i.d.d. complex normal random 
variables of mean~0 and variance~1. 

Our main result is the following. 

\begin{theo}\label{theo:maincomplex}
Let $\fkf\in\Hd\uC[n]$. On input $\fkf$, Algorithm~\nameref{alg:solve}$_\infty$ 
halts with probability~1 and performs
\[ \Oh( n^3\D^2\ln(\enumber\D)) \] 
iteration steps on average. 
\end{theo}

\begin{remark}
The bound in Theorem~\ref{theo:maincomplex} is independent 
on $N$: it is a polynomial in $n$ and $\D$. The possibility of 
such a bound for the number of iterations of a linear homotopy 
was explored in~\cite{ABBCS:16}, where the dependence on $N$ 
was reduced from linear to $\Oh(\sqrt{N})$. Pierre Lairez 
subsequently exhibited one such bound but for a {\em rigid} 
homotopy~\cite{Lairez20}. To the best of our knowledge, 
Theorem~\ref{theo:maincomplex} is the first such bound  
for a linear homotopy.
\end{remark}

We will use the following two results. The first is the complex
version 
of Proposition~\ref{prop:taildobroinfty} and has an almost
identical proof. 
The main difference lies in the needed volume computations as 
the geometry 
of the complex projective space $\bbP^n$ is somewhat different
from that 
of the real sphere $\bbS^n$. The second is a known result 
on random 
complex Gaussian matrices.

\begin{prop} \label{prop:taildobroinfty_complex}
Let $\fkf\in\Hd\uC[n]$ be a KSS complex random polynomial tuple. 
Then, for all $t>0$,
\[
\bbP\left(\|\fkf\|_\infty\uC\geq t\right)\leq 2n\left(\frac{3\D}{2}\right)^{2n}\enumber^{-(t/3)^2}.
\]
In particular, for all $\ell\geq 1$,
$
\left(\bbE_{\fkf}\left(\|\fkf\|_\infty\uC\right)^\ell\right)^{\frac{1}{\ell}}\leq 12\sqrt{\ell\,n\,\ln(eD)}
$.
\eproof
\end{prop}

\begin{prop}\label{prop:matrix_tail}\cite[Proposition~4.27]{Condition}.
Let $\fkA\in\bbC^{n\times (n+1)}$ be a random complex matrix
whose entries are i.i.d. complex normal Gaussian variables. 
Then for all $t\geq 0$,
\[ 
  \prob \left\{ \norm{\fkA^{\dagger}} \geq t \right\} \leq \frac{1}{16}\frac{n^2}{t^4}.
\]
In particular, for $\ell\in [1,4)$, $\left(\bbE_{\fkA}\|\fkA^{\dagger}\|^{\ell}\right)^{\frac{1}{\ell}}\leq \frac{\sqrt{n}}{2}\left(\frac{4}{4-\ell}\right)^{\frac{1}{\ell}}$
 \eproof
\end{prop}

\begin{proof}[Proof of Theorem~\ref{theo:maincomplex}]
We are calling Algorithm~\nameref{alg:linhomotopy} with input
$(\fkf,\fkg,\fkz)$ where $\fkf\in\Hd\uC[n]$ is a KSS complex 
polynomial system and $(\fkg,\fkz)\in\Hd[n]$ is a standard pair.

Let $\Sigma:=\{h\in\Hd[n]\mid \exists \zeta\in \bbP^n
\text{ such that }(h,\zeta)\in\tilde{\Sigma}\}$. 
By classic results in algebraic geometry, this set is a complex 
algebraic hypersurface and so it has real codimension $2$. 
Hence,  with probability one, the segment $[\fkg,\fkf]$ does not 
intersect it and, for each zero $\zeta^{(i)}$ of $\fkg$, 
we obtain a unique lifted path
\[t\mapsto (\fkq_t,\zeta^{(i)}_t)\in\mcV.\]
Here, for each $t$, the $\zeta_t^{(i)}$ cover all the 
$d_1\cdots d_n$
different zeros of $\fkq_t:=t\fkf+(1-t)\fkg$. Recall that 
behind this lifting lies the fact that the map 
$\mcV\setminus \tilde{\Sigma}\mapsto \Hd\uC[n]\setminus\Sigma$, 
$(f,\eta)\mapsto f$, is a regular covering map of degree 
$\mcD=d_1\cdots d_n$.

In this way, the random zero $\fkz$ of $\fkg$ defines, 
following its 
lifted path, a zero $\fkz_t$ of $\fkq_t$. Moreover, since the original 
$\fkz$ is chosen uniformly from the $\mcD$ zeros 
of $\fkg$, the $\fkz_t$ is a uniformly chosen zero of
$\fkq_t$. Hence
\[
 \left(\frac{\fkq_t}{\sqrt{t^2+(1-t)^2}},\fkz_t\right)\in\mcV
\]
is a standard random pair, since
$\frac{\fkq_t}{\sqrt{t^2+(1-t)^2}}$ 
is a KSS complex random polynomial and $\fkz_t$ is a 
uniformly drawn zero of this system.

By Corollary~\ref{cor:Kbound}, the expected number of iterations 
of~\nameref{alg:solve}$_\infty$ with input $\fkf$ is bounded by
\begin{equation}\label{eq:Hold0}
45n\,\D\int_{0}^1\,\bbE_{(\fkf,\fkg,\fkz)}
\left((\|\fkf\|_\infty^2+\|\fkg\|_\infty^2)^2
\|\diff_{\fkz_t}\fkq_t^{-1}\Delta\|^2\right)\,\mathrm{d}t,
\end{equation}
where we have moved the expectation inside the integral using 
Tonelli's theorem. Now, by H\"older's inequality,
\begin{equation}\label{eq:Hold1}
\bbE_{(\fkf,\fkg,\fkz)}\left((\|\fkf\|_\infty^2+\|\fkg\|_\infty^2)^2
\|\diff_{\fkz_t}\fkq_t^{-1}\Delta\|^2\right)\leq \left(\bbE_{(\fkf,\fkg,\fkz)}(\|\fkf\|_\infty^2
+\|\fkg\|_\infty^2)^6\right)^{\frac{1}{3}}\left(\bbE_{(\fkf,\fkg,\fkz)}
\|\diff_{\fkz_t}\fkq_t^{-1}\Delta\|^3\right)^{\frac{2}{3}}.
\end{equation}

By Proposition~\ref{prop:taildobroinfty_complex}, we have that
\[
\left(\bbE_{(\fkf,\fkg,\fkz)}(\|\fkf\|_\infty^2+\|\fkg\|_\infty^2)^6\right)^{\frac{1}{3}}=\Oh(n\ln(\enumber\D)).
\]
To apply the proposition we expanded the binomial and used the fact that
$\fkf$ and $\fkg$ are independent. 

Because $(\fkq_t/\sqrt{t^2+(1-t)^2},\fkz_t)$ is a random standard pair, 
we have that
\begin{equation}\label{eq:Hold2}
\bbE_{(\fkf,\fkg,\fkz)}\|\diff_{\fkz_t}\fkq_t^{-1}\Delta\|^3
=\left(t^2+(1-t)^2\right)^{\frac{3}{2}}
\bbE_{(\fkh,\fky)\sim\rst}\|\diff_{\fky}\fkh^{-1}\Delta\|^3.
\end{equation}

Now, since $(\fkh,\fky)$ is a random standard pair, the matrix
\[
\Delta^{-1/2}\diffa_{\fky}\fkh\in\bbC^{n\times (n+1)}
\]
is a random complex Gaussian matrix. This is the so-called Beltr\'{a}n-Pardo trick~\cite[Proposition~17.21(a)]{Condition}. Moreover,
$\|\diff_{\fky}\fkh^{-1}\Delta^{\frac{1}{2}}\|
=\|\diffa_{\fky}\fkh^{\dagger}\Delta^{\frac{1}{2}}\|$, 
since $\fky$ is a zero of $\fkh$ and $\diff_{\fky}\fkh$ 
is just $\diffa_{\fky}\fkh$ restricted to the orthogonal
complement of $\fky$, which we can view as $\Tg_{\fky}\bbP^n$. 
Because of this, by Proposition~\ref{prop:matrix_tail},
\[
\bbE_{(\fkh,\fky)\sim\rst}\|\diff_{\fkz_t}\fkq_t^{-1}\Delta\|^3
\leq \D^{\frac{3}{2}}\bbE_{(\fkh,\fky)\sim\rst}
\left\|\left(\Delta^{-\frac{1}{2}}\diffa_{\fkz_t}
\fkq_t\right)^{\dagger}\right\|^3\leq \frac{1}{2}\D^{\frac{3}{2}}n^{\frac{3}{2}}.
\]
Hence, integrating~\eqref{eq:Hold2}, 
\begin{equation}\label{eq:Hold3}
\left(\int_0^1\,\bbE_{(\fkf,\fkg,\fkz)}\|\diff_{\fkz_t}\fkq_t^{-1}\Delta\|^3
\,\mathrm{d}t\right)^{\frac{2}{3}}=\Oh(n\,\D).
\end{equation}
Putting together~\eqref{eq:Hold0}, \eqref{eq:Hold1} and~\eqref{eq:Hold3}
the desired result follows.
\end{proof}

\subsection{Systems of quadratic equations}

Theorem~\ref{theo:maincomplex} is an improvement over the 
average number of iterations of~\nameref{alg:solve}$_W$ 
---which is $\Oh(nDN)$. Furthermore, in the case of quadratic
systems, we can compute each iteration 
with low cost, ensuring that the average total complexity keeps smaller 
than the one for~\nameref{alg:solve}$_W$ ---which is $\Oh(n^7)$. 
The major task left, unsurprisingly, is to compute $\|q\|_\infty\uC$
in~\eqref{eq:step}. 
But we can use that, for a quadratic polynomial 
$q_i$, we can write $q_i(X)$ as $X^TA_iX$ with $A_i$ complex symmetric 
and that $\|q_i\|_\infty=\|A_i\|$. We can then compute,
for a quadratic system $q\in \mcH_{\bbtwo}[n]$ the norm 
$\|q\|_\infty=\max\|q_i\|_{\infty}$. A naive approach to compute 
each $\|q_i\|_{\infty}$ leads to an $\Oh(n^4)$ cost for the computation 
of $\|q\|_{\infty}$ as it uses $\Oh(n^3)$ operations to compute each
$\|q_i\|_\infty$. Proposition~\ref{prop:normquadraticestimator} 
below shows we can do better. All in all, we obtain the following result.

\begin{theo}[Solving Systems of Quadratic
Equations]\label{theo:quadraticsolving}
Algorithm~\nameref{alg:solve}$_\infty$ finds a common complex 
zero of 
a system of quadratic equations $f\in\mcH_{\bbtwo}[n]$ within 
$\Oh(n^{4.5+\omega})$ time on average, where $\omega$ is 
the exponent for the cost of matrix multiplication.
We currently have $\omega<2.375$.
\end{theo}

\begin{prop}\label{prop:normquadraticestimator}
Let $q\in \mcH_{\bbtwo}[n]$ be a quadratic system such that for each $i$, $q_i=X^TA_iX$. Then
\[
\|q\|_\infty\uC\leq \sqrt{\left\|\sum_{i=1}^nA_i^*A_i\right\|}\leq \sqrt{n}\|q\|_\infty\uC
\]
where the norm $\|~\|$ in the middle formula is the usual operator norm. Moreover, the number $\sqrt{\left\|\sum_{i=1}^nA_i^*A_i\right\|}$ can be computed with $\Oh(n^{1+\omega})$ operations, where $\omega$ is the exponent of matrix multiplication.
\end{prop}

\begin{proof}[Proof of Theorem~\ref{theo:quadraticsolving}]
By Proposition~\ref{prop:normquadraticestimator}, we can estimate the 
step-length of our homotopy
\[
\frac{0.015\,\|q\|_{\infty}\uC}
{\|f-g\|_{\infty}\uC\cM^2(q,z)}=\frac{0.06}
{\|f-g\|_{\infty}\uC\D\|q\|_{\infty}\uC\|\diff_{z}q^{-1}\|^2}
\]
by the smaller
\[
\frac{0.06}
{\|f-g\|_{\infty}\uC\sqrt{\left\|\sum_{i=1}^nA_i^*A_i\right\|}\,\|\diff_{z}q^{-1}\|^2}
\]
where $q=(X^TA_iX)_i$. In doing so, the algorithm still terminates, 
but gets an extra factor of $\sqrt{n}$. 

Now, $\|f-g\|_\infty$ can be computed in $\Oh(n^4)$ operations at the beginning of the algorithm a single time, so we don't need to compute it in each iteration. By Proposition~\ref{prop:normquadraticestimator}, we can compute $\sqrt{\left\|\sum_{i=1}^nA_i^*A_i\right\|}$ in 
$\Oh(n^{1+\omega})$ operations, and, by~\cite[Proposition~16.32]{Condition},
the remaining arithmetic operations can be done in $\Oh(n^3)$ operations.
Combining this with the bound of Theorem~\ref{theo:maincomplex} and 
adding the extra factor $\sqrt{n}$ gives the desired estimate.
\end{proof}

\begin{proof}[Proof of Proposition~\ref{prop:normquadraticestimator}]
By the so-called Autonne–Takagi
factorization~\cite[Problem~33]{matrixanalysis94}, we have that
\[A_i=U^T_iD_iU_i\]
for some real diagonal matrix $D_i$ with non-negative entries and 
some unitary matrix $U_i$. Now, it is easy to check that 
\[
\|q_i\|_\infty\uC=\|D_i\|=\sqrt{\|D_i^*D_i\|}=\sqrt{\|A_i^*A_i\|}\leq \sqrt{\left\|\sum_{i=1}^nA_i^*A_i\right\|},
\]
where the last inequality follows from the fact that the operator 
norm is non-decreasing with respect to the order of psd matrices. 
So $\|q\|_\infty\uC\leq \sqrt{\left\|\sum_{i=1}^nA_i^*A_i\right\|}$, as we wanted to show.

For the other inequality, observe that 
\[
\sqrt{\left\|\sum_{i=1}^nA_i^*A_i\right\|}\leq
\sqrt{\sum_{i=1}^n\left\|A_i^*A_i\right\|}
=\sqrt{\sum_{i=1}^n\left(\|q_i\|_\infty\uC\right)^2}
\leq \sqrt{n}\|q\|_\infty\uC,
\]
where the equality follows from reversing the equalities in 
the previously displayed formula. This finishes
the proof of the inequalities.

Regarding cost, note that computing $A_i^*A_i$ takes
$\Oh(n^\omega)$
operations, so computing all the $A_i^*A_i$ requires
$\Oh(n^{1+\omega})$
operations. Then adding the $A_i^*A_i$ requires $\Oh(n^3)$
operations and
computing $\left\|\sum_{i=1}^nA_i^*A_i\right\|$ another
$\Oh(n^3)$
operations. We thus get $\Oh(n^{1+\omega})$ operations in total, 
as we wanted to show. 
\end{proof}

\section{Extension to spaces of \texorpdfstring{$C^1$}{C1}-maps}\label{sec:condnumber}

We now prove some condition number theorems for the space of
$C^1$-functions over $\bbS^n$, $C^1[q]:=C^1(\bbS^n,\bbR^q)$. 
Note that $C^1[q]$ is not complete with respect to $\|~\|_{\infty}$. 
Consider instead, for $f\in C^1[q]$,
\[
\|f\|_{\binfty}:=\max_{x\in\bbS^n}
\sqrt{\|f(x)\|^2_2+\|\diff_xf\|^2_{2,2}}
=\max_{\substack{x\in\bbS^n\\v\in\Tg_x\bbS^n}}
\sqrt{\|f(x)\|^2+\frac{\|\diff_xfv\|^2_2}{\|v\|_2^2}}.
\]
This is a variant of the $C^1$-norm and so one can show that 
$C^1[q]$ is complete with respect to $\|~\|_{\binfty}$.
Let's see how this norm looks like on an easy kind of $C^1$-maps.

\begin{exam}[Linear functions]\label{ex:normoflinearfunction2}
Let $A\in q\times (n+1)$ be a linear matrix and consider the 
map $\mcA\in C^1[q]$ given by $x\mapsto Ax$. 
We can show that
\[
\|\mcA\|_{\binfty}=\sqrt{\sigma_1(A)^2+\sigma_2(A)^2}
\]
where $\sigma_1$ and $\sigma_2$ are, respectively the first and second singular values. Recall that $\sigma_1$ is also the operator norm.

To see the above equality, note that
\[
\|\mcA\|_{\binfty}=\max_{\substack{v,w\in\bbS^n\\v\perp w}}\sqrt{\|Av\|_2^2+\|Aw\|_2^2}.
\]
Since $\begin{pmatrix}Av&Aw\end{pmatrix}$ has rank at most~2,
\[
\sqrt{\|Av\|_2^2+\|Aw\|_2^2}
=\left\|\begin{pmatrix}Av&Aw\end{pmatrix}\right\|_F
=\sqrt{\sigma_1\left(\begin{pmatrix}Av&Aw\end{pmatrix}\right)^2
+\sigma_2\left(\begin{pmatrix}Av&Aw\end{pmatrix}\right)^2};
\]
and, since $\begin{pmatrix}Av&Aw\end{pmatrix}$ is an orthogonal projection, by the Interlacing Theorem for Singular Values (c.f.~\cite[3.1.3]{matrixanalysis94},
\[
\sigma_1\left(\begin{pmatrix}Av&Aw\end{pmatrix}\right)
\leq\sigma_1(A)~\text{ and }~\sigma_2\left(\begin{pmatrix}Av&Aw\end{pmatrix}\right)
\leq\sigma_2(A).
\]
Hence $\|\mcA\|_\infty\leq \sqrt{\sigma_1(A)^2+\sigma_2(A)^2}$. 
And we actually have equality as we can take $v$ and $w$ to be, 
respectively, the~1st and~2nd (right) singular vectors of $A$.
\end{exam}

\subsection{Condition Number Theorems for \texorpdfstring{$C^1[q]$}{C1[q]}}

Given $x\in\bbS^n$, we can consider the set of $C^1$-maps whose zero 
set in $\bbS^n$ have a singularity at $x$,
\[
\Sigma_x^1[q]:=\left\{g\in C^1[q]\mid 
g(x)=0,\,\mathrm{rank}\diff_xg<q\right\}.
\]
Similarly, we can consider the set of $C^1$-maps
having a singular zero,
\[
\Sigma^1[q]:=\bigcup_{x\in\bbS^n}\Sigma_x^1[q].
\]
The following result shows a way to compute the distance of a 
$C^1$-map to these sets.

\begin{theo}[Condition Number Theorem]\label{thm:GCNT}
Let $f\in C^1[q]$ and $x\in\bbS^n$, then
\[
\mathrm{dist}_{\binfty}(f,\Sigma_x^1[q])
=\sqrt{\|f(x)\|^2+\sigma_q(\diff_xf)^2}
\]
and
\[
\mathrm{dist}_{\binfty}(f,\Sigma^1[q])
=\min_{x\in\bbS^n}\sqrt{\|f(x)\|^2+\sigma_q(\diff_xf)^2}
\]
where $\dist_{\binfty}$ is the distance induced by $\|~\|_{\binfty}$ 
and $\sigma_q$ is the $q$th singular value.
\end{theo}

We call this result ``Condition Number Theorem'' as it is so 
for the following condition numbers for $C^1$-maps:
\[
\cK_{\binfty}(f,x):=\frac{\|f\|_{\binfty}}{\sqrt{\|f(x)\|^2+\sigma_q\left(\diff_xf\right)^2}}
\]
and
\[
\cK_{\binfty}(f):=\sup_{x\in\bbS^n}\cK_{\binfty}(f,x).
\]
These condition numbers are very similar to $\cK$ and one might try 
(but we won't here)
to prove an analogous of Theorem~\ref{theo:propertiesrealcond} for them 
when restricted to polynomial maps. 
For $C^1$-maps, instead, such a theorem would require dealing 
with multiple technical problems. 

For $\cK_{\binfty}(f)$, one has 
\[
\cK_{\binfty}(f) =
\frac{\max\left\{\sqrt{\|f(x)\|_2^2+\|a^*\diff_xf\|^2_2}\mid
x\in\bbS^n,\,a\in \bbS^{q-1}\right\}}
{\min\left\{\sqrt{\|f(x)\|_2^2+\|a^*\diff_xf\|^2_2}\mid 
x\in\bbS^n,\,a\in \bbS^{q-1}\right\}}.
\]
This formula shows that $\cK_{\binfty}(f)$ is similar to 
the 
condition number associated with an operator norm of a linear map.

\begin{proof}[Proof of Theorem~\ref{thm:GCNT}]
Using the triangular inequality and that $\sigma_q$ is Lipschitz 
with respect to the operator norm, we can see that, for $f,g\in C^1[q]$,
\[
 \left| \sqrt{\|f(x)\|^2+\sigma_q(\diff_xf)^2}
 -\sqrt{\|g(x)\|^2+\sigma_q(\diff_xg)^2}\right|
 \leq \|f-g\|_{\binfty}.
\]
From here, we deduce that
\[
\sqrt{\|f(x)\|^2+\sigma_q(\diff_xf)^2}
\leq\mathrm{dist}_{\binfty}(f,\Sigma_x^1[q])
\]
by taking $g\in \Sigma_x^1[q]$ and minimizing over the right-hand side. 
For the reversed inequality, let
\[
 \diff_x f=U\begin{pmatrix}
 \begin{matrix}s_1&&\\&\ddots&\\&&s_q\end{matrix}
 &{\Large\mathbf{0}}\end{pmatrix}V
\]
be the SVD of $D_xf$, where $U$ and $V$ are orthogonal
and $\mathbf{0}$ is the zero matrix.

Since orthogonal transformations leave invariant $\|~\|_{\binfty}$, we can assume, without loss of generality, that $x=e_0$ and that $V$ is the identity matrix. Consider now
\[g_i:=f_i-f_i(e_0)X_0-u_{i,q}s_qX_q.\]
We have then that $g\in\Sigma_{e_0}^1[q]$, since $g(e_0)=0$ and $\sigma_q(\diff_{e_0} g)=0$, and that
\[
f-g=f(e_0)X_0+s_qu_qX_q.
\]
By arguing as in Example~\ref{ex:normoflinearfunction} and noting that $f(e_0)X_0+s_qu_qX_q$ has rank at most 2, we have that
\begin{multline*}
  \|f(e_0)X_0+s_qu_qX_q\|_{\binfty}=\left\|\begin{pmatrix}f(e_0)&s_qu_q\end{pmatrix}\right\|_{F}\\=\sqrt{\|f(e_0)\|^2_2+\|s_qu_q\|^2_2}=\sqrt{\|f(e_0)\|^2+\sigma_q(\diff_{e_0} f)^2}/
\end{multline*}
Hence
\[
\mathrm{dist}_{\binfty}(f,\Sigma_{e_0}^1[q])\geq \|f-q\|_{\binfty}=\sqrt{\|f(e_0)\|^2+\sigma_q(\diff_{e_0} f)^2}
\]
finishing the proof of the first equality.

The second equality follows immediately from the first one.
\end{proof}

\subsection{Structured Condition Number Theorem for \texorpdfstring{$C^1[q]$}{C1[q]}}

Recall that, for $\bfd\in\bbN^q$, $\Delta$ is the diagonal $q\times q$
matrix whose diagonal is $\bfd$. We consider the following variant of 
$\|~\|_{\binfty}$,
\[
\|f\|_{\binfty,\bfd}:=\max_{x\in\bbS^n}
\sqrt{\|f(x)\|^2_2+\|\Delta^{-\frac{1}{2}}\diff_xf\|^2_{2,2}}
=\max_{\substack{x\in\bbS^n\\v\in\Tg_x\bbS^n}}
\sqrt{\|f(x)\|^2+\frac{\|\Delta^{-\frac{1}{2}}\diff_xfv\|^2_2}{\|v\|_2^2}}
\]
for $f\in C^1[q]$. The following example shows a class of functions 
for which this norm can be computed exactly.

\begin{exam}\label{exam:almostmonomial}
Let
\[
M_{a,b}:=\left(aX_0^{d_i}+\Delta^{\frac{1}{2}}bX_0^{d_i-1}X_1\right)_i
\in\Hd\uR[q].
\]
Then, we can see that
\[
\|M_{a,b}\|_{\binfty,\bfd}=\|M_{a,b}\|_W=\sqrt{\|a\|^2+\|b\|^2}.
\]
Indeed, by Proposition~\ref{prop:evaluationorthogonal}, we have 
that for all $x\in\bbS^n$,
\[
\sqrt{\|M_{a,b}(x)\|^2_2+\left\|\Delta^{-\frac{1}{2}}\diff_x M_{a,b}\right\|^2_2}\leq \|M_{a,b}\|_W.
\]
Thus $\|M_{a,b}\|_{\binfty,\bfd}\leq\|M_{a,b}\|_W$, where we have equality for $x=e_0$.
\end{exam}

We can also associate to $\|~\|_{\binfty,\bfd}$, for $f\in C^1[q]$ and $x\in\bbS^n$, the quantities 
\[
\cK_{\binfty,\bfd}(f,x):=\frac{\|f\|_{\binfty}}{\sqrt{\|f(x)\|^2+\sigma_q\left(\Delta^{-\frac{1}{2}}\diff_xf\right)^2}}
\]
and
\[
\cK_{\binfty,\bfd}(f):=\sup_{x\in\bbS^n}\cK_{\binfty,\bfd}(f,x).
\]
For these variants of $\cK_{\binfty}$, we have the following
structured condition number theorem for perturbations by 
homogeneous polynomials.

\begin{theo}[Structured Condition Number Theorem]\label{thm:structuredGCNT}
Let $f\in C^1[q]$, $x\in\bbS^n$ and $\bfd\in\bbN^q$, then
\[
 \mathrm{dist}_{\binfty,\bfd}\left(f,\Sigma_x^1[q]\cap
 (f+\Hd\uR[q])\right)
=\sqrt{\|f(x)\|^2+\sigma_q\left(\Delta^{-\frac{1}{2}}\diff_xf\right)^2}
\]
and
\[\mathrm{dist}_{\binfty,\bfd}\left(f,\Sigma^1[q]
\cap(f+\Hd\uR[q])\right)=\min_{x\in\bbS^n}\sqrt{\|f(x)\|^2+\sigma_q\left(\Delta^{-\frac{1}{2}}\diff_xf\right)^2}\]
where $\dist_{\binfty,\bfd}$ is the distance induced by
$\|~\|_{\binfty,\bfd}$ and $\sigma_q$ is the $q$th singular value.
\end{theo}

\begin{cor}\label{cor:structuredGCNT}
Let $\bfd\in\bbN^d$, $f\in \Hd\uR[q]$, and $x\in\bbS^n$. Then
\[
 \mathrm{dist}_{\binfty,\bfd}(f,\Sigma_{\bfd,x}[q])
=\sqrt{\|f(x)\|^2+\sigma_q\left(\Delta^{-\frac{1}{2}}\diff_xf\right)^2}
=\dist_W(f,\Sigma_{\bfd,x}[q])
\]
and
\[
\mathrm{dist}_{\binfty,\bfd}(f,\Sigma_{\bfd}[q])
=\min_{x\in\bbS^n}\sqrt{\|f(x)\|^2+\sigma_q
\left(\Delta^{-\frac{1}{2}}\diff_xf\right)^2}
=\mathrm{dist}_{W}(f,\Sigma_{\bfd}[q])
\]
where $\dist_{\binfty,\bfd}$ is the distance induced by
$\|~\|_{\binfty,\bfd}$ and $\sigma_q$ is the $q$th singular value.
\end{cor}

Note that the adjective `structured' refers to the fact that we 
only allow perturbations of $f$ by $C^1$-maps in $\Hd\uR[q]$. 
However, we might still be interested in general perturbations. 
If this is the case, we can get them using the relationship between
$\|~\|_{\infty,\bfd}$ and $\|~\|_{\binfty}$. We will explore 
this in more detail in the next subsection.

\begin{proof}[Proof of Theorem~\ref{thm:structuredGCNT}]
This proof is almost the same as the one of
Theorem~\ref{thm:GCNT}. 
We only have to modify the part where we find 
an explicit minimizer for the distance. Again, we write
\[\Delta^{-\frac{1}{2}}\diff_x f=U\begin{pmatrix}\begin{matrix}s_1&&\\&\ddots&\\&&s_q\end{matrix}&{\Large\mathbf{0}}\end{pmatrix}V\]
where $s_1,\ldots,s_q>0$, $U$ and $V$ are orthogonal and $\mathbf{0}$ 
is the zero matrix. Again, without loss of generality, we assume 
that $x=e_0$ and that $V$ is the identity.
We consider 
\[g_i:=f_i-x_0^{d-1}(f_i(e_0)X_0-\sqrt{d_i}u_{i,q}s_qX_q)\]
so that $g\in\Sigma_{e_0}^1[q]$, as $g(e_0)=0$ and 
$\sigma_q(\diff_{e_0} g)=0$, and 
\[
f-g=\left(f_i(e_0)X_0^{d_i}+\sqrt{d_i}s_qu_qX_q\right)_i.
\]
Because of Example~\ref{exam:almostmonomial}, for
\begin{equation*}
h=\left(a_iX_0^{d_i}+\sqrt{d_i}bX_0^{d_i-1}X_1\right)_i\in\Hd\uR[q],
\end{equation*}
we have that $\|h\|_{\binfty,\bfd}= \sqrt{\|a\|^2_2+\|b\|^2_2}$. 
Hence,
\[
\mathrm{dist}_{\binfty,\bfd}(f,\Sigma_{e_0}^1[q])\geq \|f-g\|_{\binfty}=\sqrt{\|f(e_0)\|^2+\sigma_q(\Delta^{-\frac{1}{2}}\diff_{e_0} f)^2}
\]
and the first equality follows. The second equality 
immediately follows from the first one.
\end{proof}

\begin{proof}[Proof of Corollary~\ref{cor:structuredGCNT}]
This is Theorem~\ref{thm:structuredGCNT} together 
with~\cite[Theorem~4.4]{BCL17}.
\end{proof}

\subsection{Relationship between norms}

As it happens with $\cK$ and $\kappa$ (see~\S\ref{subsec:ckvskappa}), 
the relations between the 
condition numbers $\cK$, $\kappa$, $\cK_{\binfty}$ and $\cK_{\binfty,\bfd}$ reduces to the relations between 
the corresponding norms.

We therefore prove the following propositions relating these norms. 
Note that for $C^1[q]$, we compare $\|~\|_{\binfty}$ with   $\|~\|_{\binfty,\bfd}$, and for $\Hd\uR[q]$, we compare $\|~\|_{\infty}\uR$, $\|~\|_W$, $\|~\|_{\binfty}$ and $\|~\|_{\binfty,\bfd}$. 

\begin{prop}\label{prop:normsC1q}
Let $f\in C^1[q]$. Then for all $\bfd,\widetilde{\bfd}\in\bbN^q$,
\[
\frac{1}{\max_i\sqrt{d_i}}\|f\|_{\binfty}\leq \|f\|_{\binfty,\bfd}\leq \|f\|_{\binfty}
\]
and
\[
\min\left\{1,\min_i\sqrt{\frac{\tilde{d_i}}{d_i}}\right\}
\|f\|_{\binfty,\widetilde{\bfd}}\leq \|f\|_{\binfty,\bfd}\leq \max\left\{1,\max_i\sqrt{\frac{\tilde{d_i}}{d_i}}\right\}
\|f\|_{\binfty,\widetilde{\bfd}}.
\]
\end{prop}

\begin{prop}\label{prop:normsHdq}
Let $f\in\Hd\uR[q]$. Then the following inequalities hold:
\begin{align}
\frac{1}{\sqrt{2q}\,\D}\|f\|_{\binfty}&\leq \|f\|_\infty\uR
\leq \|f\|_{\binfty,\bfd}\leq \|f\|_{\binfty}\label{eq:ineqnorms1}\\
\frac{1}{\sqrt{2q\D}}\|f\|_{\binfty,\bfd}&\leq \|f\|_\infty\uR
\leq \|f\|_{\binfty,\bfd}\label{eq:ineqnorms2}\\
\|f\|_\infty\uR&\leq \|f\|_{\binfty,\bfd}\leq \|f\|_W\label{eq:ineqnorms3}
\end{align}
\end{prop}

\begin{proof}[Proof of Proposition~\ref{prop:normsC1q}]
It is enough to show that
\[
 \|f\|_{\binfty,\bfd}\leq
 \max\left\{1,\max_i\sqrt{\frac{\tilde{d_i}}{d_i}}\right\}
 \|f\|_{\binfty,\tilde{\bfd}},
\]
since the rest of the inequalities are derived 
from this claim in a straightforward way. 
For the latter, note that
$\|~\|_{\binfty}=\|~\|_{\binfty,\bbone}$ where $\bbone=(1,\ldots,1)$.

Now, one can easily check that for $A\in \bbR^{q\times n}$,
\[
\left\|\Delta^{-\frac{1}{2}}A\right\|_{2,2}
=\left\|\Delta^{-\frac{1}{2}}\tilde{\Delta}^{\frac{1}{2}}\,
\tilde{\Delta}^{-\frac{1}{2}}A\right\|_{2,2}
\leq \left\|\Delta^{-\frac{1}{2}}
\tilde{\Delta}^{\frac{1}{2}}\right\|_{2,2}
\left\|\tilde{\Delta}^{-\frac{1}{2}}A\right\|_{2,2}
= \max_i\sqrt{\frac{\tilde{d_i}}{d_i}}
\left\|\tilde{\Delta}^{-\frac{1}{2}}\right\|_{2,2},
\]
and that, for $a,b,t\in\bbR^2$,
\[\sqrt{a^2+(tb)^2}\leq \max\{1,|t|\}\sqrt{a^2+b^2}.\]
Combining these bounds together, we get
\[
\sqrt{\|f(x)\|^2+\left\|\Delta^{-\frac{1}{2}}\diff_xf\right\|_{2,2}^2}\leq \max\left\{1,\max_i\sqrt{\frac{\tilde{d_i}}{d_i}}\right\}\sqrt{\|f(x)\|^2+\left\|\tilde{\Delta}^{-\frac{1}{2}}\diff_xf\right\|_{2,2}^2}
\]
and so the desired claim.
\end{proof}

\begin{proof}[Proof of Proposition~\ref{prop:normsHdq}]
Arguing as in Proposition~\ref{prop:normsC1q}, we can prove that,
for all $x\in\bbS^n$,
\begin{multline*}
    \frac{1}{\sqrt{2q}\,\D}\sqrt{\|f(x)\|^2+\left\|\diff_xf\right\|_{2,2}^2}\leq \max\left\{\|f(x)\|_{\infty},\left\|\tilde{\Delta}^{-1}\diff_xf\right\|_{\infty,2}\right\}\leq \sqrt{\|f(x)\|^2+\left\|\diff_xf\right\|_{2,2}^2}
\end{multline*}
and
\begin{eqnarray*}
\frac{1}{\sqrt{2q\D}}\sqrt{\|f(x)\|^2
+\left\|\Delta^{-\frac{1}{2}}\diff_xf\right\|_{2,2}^2}
&\leq& \max\left\{\|f(x)\|_{\infty},
\left\|\tilde{\Delta}^{-1}\diff_xf\right\|_{\infty,2}\right\}\\
&\leq& \sqrt{\|f(x)\|^2
+\left\|\Delta^{-\frac{1}{2}}\diff_xf\right\|_{2,2}^2}.
\end{eqnarray*}
Maximizing over $z\in\bbS^n$ gives the  inequalities
in~\eqref{eq:ineqnorms1} and~\eqref{eq:ineqnorms2}.

It only remains to prove $\|f\|_{\infty,\bfd}\leq\|f\|_W$
in~\eqref{eq:ineqnorms3}. To do this, note that by
Proposition~\ref{prop:evaluationorthogonal}, for all $x\in\bbS^n$,
\[
\sqrt{\|f(x)\|^2+\left\|\Delta^{-\frac{1}{2}}\diff_xf\right\|_{2,2}^2}\leq \|f\|_W.
\]
The result follows from maximizing over $x\in\bbS^n$.
\end{proof}

We finish with the following theorem, similar in flavour 
to~\cite[Proposition 3]{niangdiattalerario2018} and~\cite[Theorem~7]{breidingkeneshloulerario2020}, where it was shown 
that the distance of a polynomial tuple to polynomial tuples with singularities
bounds the distance of this polynomial to $C^1$-functions with
singularities.

\begin{theo}\label{theo:distancesHdqtoSigma}
Let $f\in\Hd\uR[q]$ and $x\in\bbS^n$. Then
\[\frac{1}{\sqrt{\D}}\dist_{\binfty}(f,\Sigma^1_x[q])\leq \dist_{\binfty,\bfd}(f,\Sigma_{\bfd,x}[q])=\dist_W(f,\Sigma_{\bfd,x}[q])\leq \dist_{\binfty}(f,\Sigma^1_x[q]),\]
and
\[\frac{1}{\sqrt{\D}}\dist_{\binfty}(f,\Sigma^1[q])\leq \dist_{\binfty,\bfd}(f,\Sigma_{\bfd}[q])=\dist_W(f,\Sigma_{\bfd}[q])\leq \dist_{\binfty}(f,\Sigma^1[q]),\]
where $\dist_{\binfty}$ and $\dist_{\binfty,\bfd}$ are, respectively, the distances induced by $\|~\|_{\binfty}$ and $\|~\|_{\binfty,\bfd}$. 
\end{theo}
\begin{proof}[Sketch of proof]
The proof is similar to that of Proposition~\ref{prop:normsC1q}. Arguing as there, we can prove that for all $x\in\bbS^n$,
\[
\frac{1}{\sqrt{\D}}\sqrt{\|f(x)\|_2^2+\sigma_q\left(\Delta^{-\frac{1}{2}}\diff_xf\right)^2}\leq \sqrt{\|f(x)\|_2^2+\sigma_q\left(\diff_xf\right)^2}\leq \sqrt{\|f(x)\|_2^2+\sigma_q\left(\Delta^{-\frac{1}{2}}\diff_xf\right)^2}.
\]
Minimizing over $x\in\bbS^n$ and applying Theorems~\ref{thm:GCNT} and Corollary~\ref{cor:structuredGCNT}, we conclude.
\end{proof}